\documentclass[a4paper,11pt,reqno]{amsart}
\usepackage{amsmath, enumerate}
\usepackage{amsthm}
\usepackage{graphicx}
\usepackage{amssymb}
\usepackage{appendix}
\usepackage{mathtools}
\usepackage[utf8x]{inputenc}
\usepackage{fancyhdr}
\usepackage{calc}
\usepackage{url}
\usepackage[text={6in,8.6in},centering]{geometry}
\usepackage[english]{babel}
\usepackage[T1]{fontenc}
\usepackage[usenames,dvipsnames]{color}
\usepackage{esint}
\usepackage{hyperref}

\makeatletter
\def\cleardoublepage{\clearpage\if@twoside \ifodd\c@page\else%
	\hbox{}%
	\thispagestyle{empty}
	\newpage%
	\if@twocolumn\hbox{}\newpage\fi\fi\fi}
\makeatother

\DeclareMathOperator*{\esslimsup}{ess\,lim\,sup}
\DeclareMathOperator*{\esssup}{ess\,sup}
\DeclareMathOperator*{\esslim}{ess\,lim}

\let\cleardoublepage\clearpage

\newcommand{\R}{\mathbb{R}}
\newcommand{\norm}[1]{\left\lVert#1\right\rVert}

\newtheorem{thm}{Theorem}[section]
\newtheorem{cor}[thm]{Corollary}
\newtheorem{lem}[thm]{Lemma}
\newtheorem{pro}[thm]{Proposition}
\newtheorem{den}[thm]{Definition}

\newtheorem{rem}[thm]{Remark}
\numberwithin{equation}{section}

\def\Yint#1{\mathchoice
    {\YYint\displaystyle\textstyle{#1}}%
    {\YYint\textstyle\scriptstyle{#1}}%
    {\YYint\scriptstyle\scriptscriptstyle{#1}}%
    {\YYint\scriptscriptstyle\scriptscriptstyle{#1}}%
      \!\iint}
\def\YYint#1#2#3{{\setbox0=\hbox{$#1{#2#3}{\iint}$}
    \vcenter{\hbox{$#2#3$}}\kern-.51\wd0}}
\def\longdash{{-}\mkern-3.5mu{-}}

\def\fiint{\Yint\longdash}

\begin{document}
	
	\title[]{A porous medium equation with rough weights:\\  sharp Widder theory}
	
	\author{Gabriele Grillo, Matteo Muratori, Troy Petitt, Nikita Simonov}
	
	\address{Gabriele Grillo and Matteo Muratori: Dipartimento di Matematica, Politecnico di Milano, Piazza Leonardo da Vinci 32, 20133 Milano (Italy)}
	\email{gabriele.grillo@polimi.it}
	\email{matteo.muratori@polimi.it}
	
	\address{Troy Petitt: Dipartimento di Matematica ``Federigo Enriques'',
		Universit\`a degli Studi di Milano, Via Cesare Saldini 50, 20133 Milano (Italy)}
	\email{troy.petitt@unimi.it}
	
	\address{Nikita Simonov: Sorbonne Université, Université Paris Cité, CNRS, Laboratoire Jacques-Louis Lions, LJLL, F-75005 Paris, France}
	\email{nikita.simonov@sorbonne-universite.fr}
	
	\makeatletter
	\@namedef{subjclassname@2020}{%
		\textup{2020} Mathematics Subject Classification}
	\makeatother



	\begin{abstract}
    We establish an optimal \emph{Widder theory} for a weighted porous medium equation with rough and inhomogeneous density that may be singular at a point and tends to zero at spatial infinity. Specifically, for this equation, we identify a class $X$ of initial measure data that give rise to very weak solutions, we show that non-negative very weak solutions necessarily admit an initial trace in $X$ at time $t=0$, and we prove that any two non-negative solutions having the same initial trace are equal. The corresponding theory for the classical (unweighted) equation was established by exploiting various properties that are not available in our weighted setting, such as the continuity of solutions, the explicit scale invariance of the equation, Aleksandrov's reflection principle, and the Aronson--B\'enilan inequality. Therefore, to complete the Widder theory, we must devise several proofs by means of entirely new methods. We also establish an optimal quantitative \emph{a priori} smoothing estimate for unsigned local solutions without resorting to scale invariance, which seems to be new in this form even for the classical porous medium equation. Finally, we show that non-negative very weak solutions are always locally bounded, and in particular that they have locally finite energy.
	\end{abstract}
	
	\maketitle

	\tableofcontents

	\section{Introduction}\label{sec:intro}
	We investigate initial traces, \emph{a priori} estimates, existence, and uniqueness for solutions to the following \emph{weighted porous medium equation} (WPME):
	\begin{equation}\label{wpme-noid}
		\rho(x) \, u_t = \Delta (u^m) \qquad \text{in} \;\;  \R^N\times(0,T) \, ,
	\end{equation}
	for $m>1$, $T>0$, and $N\ge3$, where $ \rho $ is a weight (or density) satisfying suitable assumptions that will be made clear shortly. The requirement $N\ge3$ is crucial for most of our results, since we use extensively potential methods. As we will see, the solvability of \eqref{wpme-noid} is strictly related to the initial value problem
    \begin{equation}\label{measure}
     \begin{cases}
     \rho(x) \, u_t = \Delta (u^m)  & \text{in} \;\; \R^N\times(0,T) \, , \\
     \rho(x) \, u = \mu  & \text{on} \;\; \R^N \times \{ 0 \} \, ,
    \end{cases}
    \end{equation}
where $\mu$ is a Radon measure (possibly sign-changing), having in a suitable sense appropriate growth at infinity. Although most of our results concern non-negative solutions to \eqref{wpme-noid}, when dealing with sign-changing solutions, we implicitly set $u^m=|u|^{m-1}u$.

The weight $\rho$ appearing in \eqref{wpme-noid} and \eqref{measure} describes an \it inhomogeneous mass density\rm, on which we make \it no regularity assumptions\rm.  Equation \eqref{wpme-noid} has been extensively studied for decades stemming from \cite{KRo,KRo2}, where it was proposed as a model for nonlinear heat transfer in an inhomogeneous one-dimensional medium. In this context, we consider solutions in the weakest possible sense, namely \emph{very weak solutions}, and in particular we avoid any \it a priori \rm mention of (even local) weak energy, though we will show here that for some class of solutions this can be recovered \it a posteriori\rm, a fact that we find interesting in itself. The possible roughness of $\rho$, which is also allowed to be singular at a point (for simplicity we choose it to be the origin) and to tend to zero at infinity, requires several \it ad hoc \rm procedures that cannot be borrowed from the case $\rho=1$, in a sense analogous to the non-trivial analysis of second-order parabolic PDEs in divergence form, even with uniformly elliptic coefficients that are merely measurable. In particular, this involves several technical adjustments that are necessary to deal with key properties of solutions that are satisfied only for \emph{almost every} $ x $ or $t$.
Our first main goal is to prove initial-trace properties for \emph{non-negative} solutions $u$ to \eqref{wpme-noid}, in the spirit of the celebrated work \cite{AC}. We recall that an initial trace $\mu\in\mathcal{M}\!\left(\R^N\right)$ (the set of Radon measures -- that is, measures with a locally finite total variation) of $u$ is defined by the validity of the limit
\begin{equation*}
    u(x,t) \, \rho(x) \underset{t \to 0}{\longrightarrow} \mu \qquad \mathrm{weakly}^* \ \mathrm{in}\;\; \mathcal{M}\!\left(\R^N\right) .
\end{equation*}
Therefore, before introducing our results, we give a brief overview of basic issues concerning the theory of initial traces for the (unweighted) heat equation and the porous medium equation (PME).

\subsection{\for{toc}{Widder theory for the heat equation and for the classical PME ($\rho=1$)}\except{toc}{Widder theory for the heat equation and for the classical PME (\texorpdfstring{$ \boldsymbol{\rho=1}$}{})}}\label{widder}
Let us first consider the heat equation. In  the paper \cite{Wid} by D.V. Widder (see also the monograph \cite{WidBook}), it was proved that all positive solutions to $u_t=u_{xx}$ in a finite strip $\R\times(0,T)$ have a unique initial trace $\mu$ at $ t=0 $ with growth
\begin{equation*}
    \mu(B_R)=\mathcal{O}\!\left(e^{aR^2}\right) \qquad \mathrm{as}\;\;R\to+\infty\,,
\end{equation*}
for some $a>0$ depending on $T$, where $B_R$ denotes the open ball of radius $R>0$ centered at the origin. In \cite{Aronson-IT}, the higher-dimensional case was handled, along with the corresponding result for general linear divergence-form parabolic equations. We recall that solutions of both the heat equation and the porous medium equation may blow up in finite time if the growth at spatial infinity is critical, thus explaining the need to treat solutions defined only in strips $\R^N\times(0,T)$.

With regards to the classical (unweighted) porous medium equation, that is
 \begin{equation}\label{pme}
 u_t = \Delta \!\left( u^m \right) \qquad \text{in}\;\; \R^N\times(0,T) \, ,
     \end{equation}
the existence of Barenblatt-type solutions (see \cite[Chapter 4]{Vazquez}), namely solutions of the corresponding initial-value problem
 \begin{equation}\label{wpme-funid-pme}
 \begin{cases}
 u_t = \Delta \!\left( u^m \right) & \text{in} \;\; \R^N\times(0,T) \, , \\
 u = u_0  & \text{on} \;\; \R^N \times\{ 0 \} \, ,
     \end{cases}
     \end{equation}
taking as $u_0$ the \emph{Dirac delta} measure, immediately makes it necessary to consider a weaker class of initial data, rather than a function $u_0\in L^1\!\left(\R^N\right)$, when treating general non-negative solutions. The article \cite{AC} by D.G. Aronson and L.A. Caffarelli was a breakthrough for the study of initial traces in non-linear diffusion problems. In that work, it is found that for each non-negative and continuous solution to \eqref{pme} there necessarily exists a unique initial trace measure $\mu$, with the explicit growth bound
\begin{equation}\label{growth}
	\mu(B_R)=\mathcal{O}\!\left(\!R^{N+\frac{2}{m-1}}\right) \qquad \mathrm{as}\;\; R\to+\infty \, .
\end{equation}
 It is noteworthy that this result was obtained via the careful study of special solutions of Barenblatt and blow-up type, and without the use of a \emph{representation formula}, which is of course not available for non-linear problems. The continuity assumption was later removed in the remarkable paper \cite{DK2}, but the non-negativity assumption is unavoidable, see \cite[Section 12.8]{Vazquez}.

The growth estimate \eqref{growth}, which actually also holds for solutions at any given positive time,
can be used to prove uniqueness for general non-negative solutions taking the same initial trace. The close connection between initial-trace estimates and uniqueness for parabolic equations is well known, and originally took the form of a representation formula for the heat equation. Indeed, the main result in \cite{Wid} is that a non-negative solution is \emph{uniquely identified} as a convolution between the heat kernel and an initial measure trace. For the PME, the stated initial-trace estimates in \cite{AC} were exploited in \cite{DK} by B.E.J. Dahlberg and C.E. Kenig to prove uniqueness for general positive solutions, \emph{without any a priori growth assumptions} on the solutions themselves.

To determine the precise results that are necessary to provide a complete theory of initial traces, which together constitute the so-called \emph{Widder theory} for \eqref{wpme-funid-pme}, we note that the main goal is to establish bijections between an initial data space $X$ and a solution space $Y$ (where both spaces depend on the particular equation in consideration), according to the following program:
\begin{enumerate}
    \item \emph{Identify a large enough initial class of possibly sign-changing measures $X$ and a regularity/growth class of solutions $Y$ such that
    $\mu\in X\implies\exists u\in Y$ taking $\mu$ as its initial trace}. This was proven in \cite{BCP} for the PME. 
    \item \emph{Prove that if $u\geq0$ is a very weak solution $\implies\exists! \, \mu\in X_+$ (the non-negative measures in $X$) an initial trace of $u$, and $u \in Y$}. This was proven in \cite{AC} (and \cite{DK,DK2}) for the PME. 
    \item \emph{General non-negative solutions taking the same initial trace are equal}. This was proven in \cite{DK} for the PME. 
\end{enumerate}

We also comment that when $u_0$ is a \emph{sign-changing} measure, no uniqueness result for \eqref{wpme-funid-pme} is available to the best of our knowledge, even if such a measure has a \emph{finite} total variation.

For the heat equation, $X$ and $Y$ are spaces of exponential-growth type, and all of these results are by now classical; see again \cite{WidBook}. Finally, we remark that the Widder theory for the \emph{fractional} heat equation was completed in \cite{BPSV,BSV}. Recently, significant generalizations have been proved in \cite{GQSV}, always in the linear case, where in particular fractional-type operators with possibly rough kernels are considered, which are the natural counterparts of elliptic operators with measurable coefficients in the local case.

\subsection{Widder theory for the WPME with rough weights: our results}
We start by stating our main assumption on the weight $\rho$, which reads as follows: $\rho\equiv\rho(x)$ is a measurable function in $ \R^N $ satisfying
	\begin{equation}\label{weight-cond}
    \begin{gathered}
		\underline{C} \left( 1 + |x| \right)^{-\gamma}  \le \rho(x) \le \overline{C} \, |x| ^{-\gamma} \qquad \text{for a.e.} \;\; x \in \mathbb{R}^N \, , \\
    \text{for some}\;\;\gamma\in[0,2) \;\; \text{and constants} \;\; \overline{C}>\underline{C}>0  \, .
        \end{gathered}
	\end{equation}
We stress that our results, to be stated below, are new \emph{even for a measurable weight $\rho$ that is bounded and bounded away from zero}, \emph{i.e.}~$\gamma=0$ in \eqref{weight-cond}, but we can allow for weights tending to zero at infinity and possibly being singular at the origin, \emph{i.e.}~$\gamma\in(0,2)$. We will explain in a few lines why $\gamma=2$ appears as a natural threshold for these kinds of problems.

As recalled above, since its introduction in \cite{KRo,KRo2}, equation \eqref{wpme-noid} has been widely studied. More recently, efforts have been made to extend the classical theory of the unweighted porous medium equation in higher dimensions to the weighted case; see for example \cite{GMPo,KRV,RV1,RV2,RV3}. We comment that the case $\gamma\ge2$ is in several respects completely different from the one studied here, \emph{e.g.}\ as concerns asymptotic behavior (see \cite{KRV,NR}). Besides, non-negative bounded solutions taking suitable arbitrary conditions at infinity have been shown in \cite{GMP3} to exist when $\gamma>2$, therefore our main results on uniqueness simply cannot hold if $\gamma >2$. On the other hand, in the critical case $\gamma=2$, many of our other results fail, for instance the Aronson--Caffarelli estimate Theorem \ref{thm:weighted-AC} has diverging exponents as $\gamma\to2^-$, and the multiplying constants of the local smoothing estimates \eqref{l.infty.balls} actually blow up as $\gamma\to2^-$. Therefore, this borderline case is delicate and should be dealt with separately.

Remarkably, the \emph{radial} version of \eqref{wpme-noid} with a decaying (radial) weight $\rho$ has proven to be closely connected to the unweighted porous medium equation on a \emph{Riemannian manifold with non-positive curvature}; we refer to \cite{GMV1,GMV2,Vazquez-hyp} for the details. It turns out that $\gamma\in(0,2)$ corresponds (via a suitable transformation) to a non-compact manifold with negative curvature \emph{decaying at spatial infinity}, whereas $\gamma=2$ corresponds to the \emph{hyperbolic space}. Furthermore, $\gamma = 2$ with suitable \it logarithmic perturbations \rm is associated with manifolds having divergent negative curvature, which can also be \emph{very negative} in the sense of \cite{GMV2}, a setting in which it is known that \it uniqueness does not hold \rm even for bounded solutions.


Our new results, in combination with those of \cite{MP}, complete the \emph{Widder theory for \eqref{wpme-noid} with rough weights}, which consists, with the same notation and numbering given in Subsection \ref{widder}, in establishing a bijection between an initial data space $X$ and a solution space $Y$ as follows:
\begin{enumerate}

    \item The identification of an initial class of possibly sign-changing measures $X$ and a regularity/growth class of solutions $Y$ for the WPME such that
    $\mu\in X\implies\exists u\in Y$ taking $\mu$ as its initial trace, is shown by combining \cite[Theorem 2.2]{MP} and Theorem \ref{mBCP} below. In our context, the spaces $X$ and $Y$ are introduced in Definition \ref{def-Morrey} and in \eqref{def-space-Y}, respectively. We comment that in the corresponding results of \cite{BCP} the so-called \it Aronson--B\'enilan inequality \rm established in \cite{AB} is crucial, but the latter is not available in the weighted framework (see \cite[Introduction]{MP}).

    \item The fact that any non-negative very weak solution to the WPME has a unique initial trace $\mu\in X_+$ is proved in Theorem \ref{thm:weighted-AC}, which in particular provides an Aronson--Caffarelli-type estimate (whose sharpness is discussed in Subsection \ref{ac space optimal}). We stress that in the PME case, continuity of solutions is a crucial tool for the proof, but such a property is not available in our rough weighted setting. We replace it by Theorem \ref{apriori-bounded}  establishing \it local boundedness \rm of non-negative solutions, which is enough for our purposes, in combination with the optimal local smoothing estimates of Theorem \ref{local moser iter lem}, ensuring that $u$ has the right pointwise growth and is therefore an element of $Y$. Also, we point out that the results of \cite{AC, DK, DK2} take advantage of the scale invariance of the equation and the validity of Aleksandrov's reflection principle (which is replaced by our potential methods), all of such results being lost here due to the presence of $\rho$.

    \item  We prove in Theorem \ref{uniq thm} that non-negative very weak solutions, \it without growth assumptions\rm, taking the same initial trace are equal. Again, in contrast to \cite{DK, DK2}, we avoid any use of continuity properties of solutions.
\end{enumerate}

To be more specific, in the weighted case in which assumption  \eqref{weight-cond} is required,
we find that $X$ is the set of initial traces satisfying
\begin{equation}\label{weighted-growth-rate}
	|\mu|(B_R)=\mathcal{O}\!\left(\!R^{N-\gamma+\frac{2-\gamma}{m-1}}\right)\qquad\text{as}\;\;R\to+\infty\, ,
\end{equation}
where $|\mu|$ is the total variation of $\mu$.
Furthermore, the pointwise growth condition required to belong to $Y$ reads
\begin{equation}\label{weighted-inf-growth-rate}
    \norm{u(t)}_{L^\infty(B_R)}=\mathcal{O}\!\left(R^{\frac{2-\gamma}{m-1}}\right)\qquad\text{as}\;\;R\to+\infty\,,
\end{equation}
for all $t\in(0,T)$, and it is a key tool for uniqueness since it allows us to exploit \cite[Theorem 2.3]{MP}. Of course, in the case $\gamma=0$, namely possibly rough weights that are (essentially) bounded and bounded away from zero, we get the same growth rates valid for the classical PME. However, we stress again that our results are new in this case as well.

\subsection{Further results: sharp smoothing estimates and local boundedness} Finally, it is worth describing more in detail another two main theorems of ours, of independent interest:
\begin{itemize}

\item We prove the \emph{a priori} smoothing estimate Theorem \ref{local moser iter lem}, for possibly sign-changing and purely local solutions, that allows us to pass from $L^1$-growth of type \eqref{weighted-growth-rate} to pointwise growth of type \eqref{weighted-inf-growth-rate}. This result is new in the present form even for the PME, and we believe that it may have independent interest. We discuss its sharpness, in suitable senses, in Subsection \ref{gen-rad-rem}.

\item We prove in Theorem \ref{apriori-bounded} that non-negative very weak  solutions to \eqref{wpme-noid} are always locally bounded. This result, which is also true for merely local solutions, is somewhat expected in view of \cite{DK2}, but we employ a suite of refined potential methods to overcome significant technical challenges resulting from the inclusion of a rough weight $\rho$. As a remarkable consequence, \emph{non-negative very weak solutions are locally weak}, \emph{i.e.}~they have a locally finite energy.

\end{itemize}

Although we will not pursue any further extension here, we expect that our potential methods might have a wider degree of applicability, for instance to the Riemannian setting, at least under controlled curvature or volume growth.

	\section{Preliminary material and statement of the main results}\label{prelim}
After an introductory subsection concerning functional preliminaries and notations, we will state all of our main results, dividing them into Aronson--Caffarelli estimates, existence and uniqueness, and local smoothing effects. Finally, we will devote the last subsection to a discussion of the optimality of the estimates we prove.

From now on, and in the subsequent sections, if not explicitly stated we will take for granted that $N \ge 3$, $m>1$ and $ \rho $ complies with \eqref{weight-cond}.

	\subsection{Functional spaces, auxiliary definitions, and notation}
    In the next definitions we introduce the main functional spaces (and related norms) we will work with.
	\begin{den}[Weighted Lebesgue spaces]\label{def:weighted-norm}
		Let $p\in[1,\infty)$ and $Q=\Omega\times(t_1,t_2)$, where $\Omega\subset \R^N$ is a (possibly unbounded) domain and $0 \le t_1<t_2$. Given a measurable function $f$ in $Q$, we set
		\begin{equation*}
			\norm{f}_{L^p_\rho(Q)}=\left(\int_{t_1}^{t_2}\int_{\Omega}\left|f(x,t)\right|^p\rho(x)\,dxdt\right)^{\frac 1 p} ,
		\end{equation*}
and we denote by $ L^p_\rho(Q) $ the weighted Lebesgue space of all measurable functions $f$ in $Q$ such that $ \|f \|_{L^p_{\rho}(Q)} <+\infty $. Similarly, for measurable functions that depend only on the spatial variable $x$, we set
	\begin{equation}\label{weighted-leb-norm}
			\norm{f}_{L^p_\rho(\Omega)}=\left(\int_\Omega \left|f(x)\right|^p\rho(x)\,dx\right)^{\frac 1 p} ,
		\end{equation}
        with the usual definition of $ L^p_\rho(\Omega) $. In view of the assumptions on $\rho$, the weighted $ L^\infty $ spaces coincide with the standard ones with respect to the Lebesgue measure; we then denote by $L^\infty_c(\Omega)$ the space of $L^\infty(\Omega)$ functions with compact support in $\Omega$.

Also, we use the notation $ L^p_{\rho,\mathrm{loc}}(Q) $ to refer to functions that belong to $ L^p_{\rho}(Q') $ for all $ Q' \Subset Q $, with a completely analogous definition for $ L^p_{\rho,\mathrm{loc}}(\Omega) $.
When referring to unweighted Lebesgue spaces, that is the above ones with $\rho = 1 $, we will simply write $ L^p $ or $ L^p_{\mathrm{loc}} $, whereas in the special case $ \rho(x) = |x|^{-\gamma} $ we will write $ L^p_\gamma $ or $ L^p_{\gamma,\mathrm{loc}} $.

Upon calling $\sigma$ the $(N-1)$-dimensional Hausdorff measure on $\partial\Omega$, we will let $L^p_\sigma( \partial \Omega \times (t_1,t_2) )$ denote the corresponding Lebesgue space of measurable functions $g$ on $ \partial \Omega \times (t_1,t_2) $ such that $ |g| ^p$ is integrable with respect to the boundary measure $ d\sigma \otimes dt $.
	\end{den}

    In the special case when $ \Omega $ is the open ball of radius $R>0$ centered at some $x_0 \in \R^N$, we write $ B_R(x_0) $, and if $x_0=0$ we just write $B_R$.

    Up to standard identifications, it is plain that $ L^p_{\rho}(Q) $ coincides with the Bochner space $ L^p\!\left((t_1,t_2);L^p_\rho(\Omega)\right) $, but in fact we will need to consider more general versions of these spaces.

    \begin{den}[Bochner spaces -- see \cite{Cohn,Ev,Mik}]\label{boch}
For a Banach space $V$, a time interval $I \subset \R$, and some $ p \in [1,\infty] $, we denote by $ L^p(I;V) $ the Banach space of all strongly measurable functions $ f : I \to V $ such that
$$
\left\| f \right\|_{L^p(I;V)} = \left( \int_I \left\| f(t) \right\|_V^p
dt \right)^{\frac 1 p} <+\infty \quad \text{if}\;\; p < \infty \, , \quad \left\| f \right\|_{L^\infty(I;V)} = \underset{t \in I}{\esssup} \left\| f(t) \right\|_V < +\infty \, .
$$
For all $f \in L^p(I;V) $ admitting a weak derivative $ f_t \in L^p(I;V) $, we can define the Banach space $ W^{1,p}(I;V) $ endowed with the norm
$$
\left\| f \right\|_{W^{1,p}(I;V)} = \left\| f \right\|_{L^p(I;V)} + \left\| f_t \right\|_{L^p(I;V)} ,
$$
where, as is customary, in the special case $p=2$ we set $ H^1(I;V) = W^{1,2}(I;V) $. Furthermore, we will let $ \mathrm{Lip}(I;V) $ denote the space of all functions $f:I \to V$ such that
$$
\left\| f(t)-f(s) \right\|_V \le L \, |t-s| \qquad \forall t,s \in I \, ,
$$
for some constant $L>0$.

Normally, we will add the subscript $ _\mathrm{loc} $ to mean that the above properties are satisfied in any subinterval $ I' \Subset I $. If it makes sense in $V$, which is the case for Lebesgue spaces, we may also write $ V_{\mathrm{loc}} $, so that a notation of the form $ H^1_{\mathrm{loc}}\!\left(I; L^p_{\mathrm{loc}}(\Omega) \right) $ means that $ f \in H^1\!\left(I'; L^p_{\mathrm{loc}}(\Omega') \right) $ for all $ I' \Subset I $ and all $ \Omega' \Subset \Omega $.

Concerning continuous curves, we will let $C\!\left(I ; V \right)$ denote the space of all (strongly) continuous functions $ f : I \to V $, and when $ I $ is not compact, convergence will always be meant locally uniformly in time, that is, we say that $ f_k \to f $ in $C\!\left(I ; V \right)$ as $ k \to \infty $ if
\begin{equation*}
    \sup_{t\in I' }\norm{f_k(t)-f(t)}_V \to 0 \qquad \text{as}\;\;k\to\infty\,, \ \text{for all}\;\; I' \Subset I \, .
\end{equation*}
	\end{den}

    In the sequel we will systematically use the following cutoff functions.

    \begin{den}[Spatial cutoff functions]\label{den-cutoff}
        For any $ R>0$ we define $  \phi_R \in  C^\infty_c\!\left( \R^N \right)$ to satisfy
		\begin{equation}\label{defcutoff1}
			0\leq\phi_R\leq1 \quad \text{in}\;\; \R^N \, , \qquad \phi_R=1 \quad \text{in}\;\; B_R \, , \qquad \phi_R=0 \quad \text{in}\;\; B_{2R}^c \,,
		\end{equation}
		and the pointwise estimates
		\begin{equation}\label{cutoff-est}
			\left|\nabla\phi_R\right|\leq\frac{C}{R} \quad \text{and} \quad \left|\Delta\phi_R\right|\leq\frac{C}{R^2} \qquad \text{in} \;\; \R^N   \, ,
		\end{equation}
		for some constant $C>0$ depending only on $N$.
	\end{den}
	
	\begin{den}[Morrey-type spaces for large data]\label{def-Morrey}
		Let $r\geq1$. Given a (possibly sign-changing) Radon measure $\mu$ in $\R^N$, we define
		\begin{equation*}\label{def-norm-1-r}
			\norm{\mu}_{1,r}=\sup_{R\geq r}R^{-(N-\gamma)-\frac{2-\gamma}{m-1}} \, |\mu|(B_R)\,,
		\end{equation*}
		where $|\mu|$ is the total variation of $\mu$. We denote by $X$ the Banach space of all Radon measures $ \mu $ in $\R^N$ such that $\norm{\mu}_{1,r}<+\infty$. By the local finiteness of the total variation of $\mu$, such a definition is independent of $r\geq1$ and all norms $ \| \cdot \|_{1,r} $ are equivalent. Furthermore, let
		\begin{equation*}\label{lim-lim}
			\ell(\mu):=\lim_{r\to+\infty}\norm{\mu}_{1,r} .
		\end{equation*}
         We say that $\mu\in X_0$ if $
         \ell(\mu)=0 $.
    Occasionally, it will be convenient to work with the following (equivalent) norm:
            \begin{equation*}\label{def-norm-1-rbis}
			\vert{\mu}\vert_{1,r}=\sup_{R\geq r}R^{-(N-\gamma)-\frac{2-\gamma}{m-1}} \, \int_{\mathbb R^N} \phi_R \, d|\mu| \, .
		\end{equation*}

In the special case when $ \mu \equiv f \rho  \in L^1_{\mathrm{loc}}\!\left(\R^N\right)$, the above definitions have already been given in \cite[Subsection 2.1]{MP}; in order to be consistent with the notation introduced therein, we will write $ \| f \|_{1,r} $ in the place of (the more appropriate) $ \| f \rho \|_{1,r} $. The same rule applies to $|\cdot|_{1,r}$.

		Finally, for a function $f\in L^\infty_{\mathrm{loc}}\!\left(\R^N\right)$, we set
		\begin{equation}\label{def-norm-inf-r}
			\norm{f}_{\infty,r}=\sup_{R\geq r}R^{-\frac{2-\gamma}{m-1}} \, \norm{f}_{L^\infty(B_R)}  ,
		\end{equation}
        and it is plain that such a norm is stronger than $ \| f \|_{1,r} $.
	\end{den}

In agreement with \cite[Subsection 2.1]{MP}, given any measurable function $f$ in $\R^N$, we define
\begin{equation*}\label{alpha-weight-norm}
\norm{f}_{L^1(\Phi_\alpha)}=\int_{\R^N}|f(x)|\,\Phi_\alpha(x)\,\rho(x)\,dx \, ,
\end{equation*}
and we denote by $ L^1(\Phi_\alpha) $ the weighted Lebesgue space of all measurable functions $f$ in $ \R^N $ such that $ \| f \|_{L^1(\Phi_\alpha)} < +\infty $, where $ \Phi_\alpha $ is another weight of the form
\begin{equation}\label{alpha-cond}
\Phi_\alpha(x)=\left(1+|x|^2\right)^{-\alpha} \quad \text{for some} \;\; \alpha>\frac{2-\gamma}{2(m-1)}+\frac{N-\gamma}{2} \, .
\end{equation}
It was shown in \cite[Proposition A.1]{MP} that $X$ is continuously embedded in $ L^1(\Phi_\alpha) $.

\smallskip

We let $ W^{k,p}(\Omega) $ denote the usual Sobolev space of $L^p(\Omega)$ functions such that all weak partial derivatives of order up to $k$ are also $L^p(\Omega)$ functions, with $ H^k(\Omega) = W^{k,2}(\Omega) $. The $ _\mathrm{loc} $ subscript applies exactly as above. When $\Omega$ is a bounded smooth domain, we let $ W^{k,p}_0(\Omega) $ denote the subspace of  $W^{k,p}(\Omega) $ functions with all derivatives of order up to $k-1$ vanishing on $\partial \Omega$. In the special case $\Omega = \R^N$, the symbol $ \dot H^1\!\left( \R^N \right) $ stands for the closure of $ C_c^\infty\!\left( \R^N \right) $ with respect to the $L^2\!\left( \R^N \right)$ norm of the gradient.

\smallskip

In general, we will use notations that are deemed standard. When dealing with integrals, we will often omit the integration variables to lighten the reading. A space-time function such as $u(x,t)$ will sometimes be displayed as $u(t)$, if it is to be interpreted as a whole spatial function at a fixed time $t$; this will occasionally be done in space-time integrals as well. To denote time derivatives we will use either the subscript $ _t $ or the symbol $ \partial_t $, depending on readability.

\smallskip

	Finally, it is convenient to define the following special constants, which are related to scale invariance and appear many times below (as well as in previous works \cite{MP,MPQ}):
	\begin{equation}\label{deflambdatheta}
		\lambda=\frac{N-\gamma}{(N-\gamma)(m-1)+2-\gamma} \, , \qquad \theta=\frac{2-\gamma}{N-\gamma} \, ;
	\end{equation}
	notice that it holds
	\begin{equation}\label{spexprel}
		\lambda(m-1)+\theta\lambda=1 \, .
	\end{equation}
	
    \subsection{An Aronson--Caffarelli-type initial trace theorem}\label{subsec:AC}
	Our principal definition of solution to \eqref{wpme-noid} is the following.
	\begin{den}[Global very weak solutions]\label{def:vw-gl}
        We say that a function
        $$ u \in L^1_{\rho,\mathrm{loc}}\!\left(\R^N\times(0,T)\right)\cap L^m_{\mathrm{loc}}\!\left(\R^N\times(0,T)\right)
        $$
        is a very weak solution to \eqref{wpme-noid} if
		\begin{equation}\label{vw-cond}
			\int_{0}^{T} \int_{\mathbb{R}^N} \left(u \, \psi_t \, \rho + u^m \, \Delta \psi \right)  dx dt = 0
		\end{equation}
		for all $\psi\in C^\infty_c\!\left(\mathbb{R}^N\times (0,T)\right)$.
	\end{den}

	
	The first main result we present is a sharp estimate on initial traces, for non-negative solutions. In fact, it is a weighted version of the celebrated result \cite[Theorem 4.1]{AC} for the unweighted equation \eqref{pme}.
	\begin{thm}[Aronson--Caffarelli-type estimate {for initial traces}]\label{thm:weighted-AC}
		Let  $N\geq3$, $m>1$, $ T > 0 $, and $\rho$ be a measurable function satisfying \eqref{weight-cond}. Let $u$ be a non-negative solution to \eqref{wpme-noid}, in the sense of Definition \ref{def:vw-gl}. Then there exists a unique non-negative Radon measure $\mu$ that is the initial datum of $u$,
		in the sense that
		\begin{equation}\label{measure-data}
			{\lim_{t\to0}}\int_{\R^N} \varphi(x)\,u(x,t)\,\rho(x)\,dx=\int_{\R^N}\varphi\,d\mu \qquad \forall\varphi\in C_c\!\left(\R^N\right) ,
		\end{equation}
where $  t \mapsto \int_{\R^N} \varphi(x) \, u(x,t) \,\rho(x) \, dx $ has a continuous version in $ (0,T) $. Furthermore, $\mu$ satisfies the growth estimate
		\begin{equation}\label{weighted-AC-inq}
			\mu(B_R)\leq C \left[t^{-\frac{1}{m-1}}R^{N-\gamma+\frac{2-\gamma}{m-1}}+t^{\frac{N-\gamma}{2-\gamma}} \left(\fiint_{B_\varepsilon\times (t,t+\delta)}u\,\rho\,dxds\right)^{1+\frac{N-\gamma}{2-\gamma}(m-1)}\right] \qquad \forall R \ge 1 \, ,
		\end{equation}
		for all $\varepsilon \in (0,1)$, all $t\in(0,T)$, and all $\delta \in(0,T-t)$, where $C>0$ is a constant depending only on $N,m,\gamma,\underline{C}, \overline{C} $.
	\end{thm}

	\subsection{Existence and uniqueness results}\label{subsec:exist-uniq}

Before stating our main existence and uniqueness theorems, we introduce the definition of very weak solution of the Cauchy problem \eqref{measure}.
    \begin{den}[Global very weak solutions with measure initial data]\label{def:vw-gl-m}
        Let $\mu$ be a (possibly sign-changing) Radon measure in $\R^N$. We say that a function $ u $ is a very weak
		solution of problem \eqref{measure} if it is a very weak solution to \eqref{wpme-noid}, in the sense of Definition \ref{def:vw-gl}, and in addition \eqref{measure-data} holds (as an essential limit).
    \end{den}

We are now ready to state our main existence and uniqueness theorems (the latter for non-negative solutions) for problem \eqref{measure}. Recall that the spaces $X$ and $X_0$ have been introduced in Definition~\ref{def-Morrey}.

	\begin{thm}[Existence of solutions with measure initial data]\label{mBCP}
		Let  $N\geq3$, $m>1$, and $\rho$ be a measurable function satisfying \eqref{weight-cond}. Let $\mu\in X $ be a (possibly sign-changing) Radon measure. Then there exists a solution $u$ of problem \eqref{measure} in the sense of Definition \ref{def:vw-gl-m},  which we call the constructed one, with $ T=\mathsf{T}(\mu) $, where
		\begin{equation}\label{def-t-u0m}
			\mathsf{T}(\mu) = \frac{\mathsf{C}_1}{\left[ \ell(\mu) \right]^{m-1}} \quad \text{if}\;\; \mu \in X \setminus X_0 \, , \qquad  \mathsf{T}(\mu) = + \infty \quad \text{if}\;\; \mu\in X_0 \, ,
		\end{equation}
		and $ \mathsf{C}_1>0 $ is a constant depending only on $ N,m,\gamma,\underline{C},\overline{C} $. Moreover, the constructed solution $u$ enjoys the  regularity properties
\begin{equation}\label{Lip-X}
    u \in \mathrm{Lip}_{\mathrm{loc}}\!\left( \left(0,\mathsf{T}(\mu)\right) ; X \right) \cap W^{1,\infty}_{\mathrm{loc}}\!\left(\left(0,\mathsf{T}(\mu)\right);L^1(\Phi_\alpha) \right)  .
\end{equation}
        Furthermore, upon setting
		\begin{equation*}\label{maxtimem}
			\mathsf{T}_r(\mu) = \frac{\mathsf{C}_1}{\left\|\mu \right\|_{1,r}^{m-1}} \qquad \forall r \ge 1 \, ,
		\end{equation*}
		there exist constants $\mathsf{C}_2, \mathsf{C}_3>0$, depending only on $ N,m,\gamma,\underline{C},\overline{C} $, such that
		\begin{equation}\label{smoothing estimatem}
			\left\| u(t) \right\|_{1,r} \leq \mathsf{C}_2 \norm{\mu}_{1,r}  \qquad \forall t \in \left( 0 , \mathsf{T}_r(\mu) \right),
		\end{equation}
		\begin{equation}\label{key estimatem}
			\left\| u(t) \right\|_{\infty,r} \leq \mathsf{C}_3 \, t^{-\lambda}\norm{\mu}_{1,r}^{ \theta \lambda} \qquad  \forall t \in \left( 0 , \mathsf{T}_r(\mu) \right).
		\end{equation}
		We also have the following ordering principle: if $ v $ is the constructed solution associated with another initial datum $ \nu \in X $, then $\mu\leq \nu$ implies $u(t) \leq v(t)$ for all $  t \in \left( 0 , \mathsf{T}(\nu) \right) $.
		
		Finally, if in addition $ \mu \in X_0 $, then $ u \in  W^{1,\infty}_{\mathrm{loc}}\!\left(\left(0,\mathsf{T}(\mu)\right); X_0 \right) $ and
		\begin{equation}\label{subcritical-growth}
			\esslim_{|x|\to+\infty}|x|^{-\frac{2-\gamma}{m-1}}\,u(x,t)=0 \qquad \forall t>0 \, .
		\end{equation}
	\end{thm}
We point out that the Lipschitz- and $W^{1,\infty}$-regularity properties \eqref{Lip-X} seem to be new, even in the classical case $ \rho = 1 $. Specifically, the fact that $u$ turns out to be a $W^{1,\infty}_{\mathrm{loc}}$ curve with values in $ L^1(\Phi_\alpha) $ is particularly relevant, since the inclusion $ \mathrm{Lip}(I;V) \subset W^{1,\infty}(I;V) $ is in general false when $ V $ is not a reflexive space (see the references in Definition  \ref{boch}).


	\begin{thm}[{Uniqueness without global assumptions}]\label{uniq thm}
		Let  $N\geq3$, $m>1$, $ T > 0 $, and $\rho$ be a measurable function satisfying \eqref{weight-cond}. Let $u$ and $v$ be non-negative solutions to \eqref{wpme-noid}, in the sense of Definition \ref{def:vw-gl}, such that
		\begin{equation}\label{same data}
			\esslim_{t\to0}\int_{\R^N}\left[ u(x,t)-v(x,t)\right] \varphi(x)\,\rho(x) \, dx=0 \qquad \forall \varphi \in C_c\!\left(\R^N\right) .
		\end{equation}
		Then $u = v$.
	\end{thm}
As a consequence of Theorems \ref{thm:weighted-AC}, \ref{mBCP} and \ref{uniq thm}, we can assert that, as far as non-negative solutions are concerned, problem \eqref{measure} is solvable if and only if $ \mu \in X $, and in that case it admits exactly one non-negative solution in the sense of Definition \ref{def:vw-gl-m} (with $ T=\mathsf{T}(\mu) $). Moreover, thanks to \eqref{smoothing estimatem}, for all $ t \in (0,T) $ we have
        \begin{equation*}
            \norm{u(t)}_{L^\infty(B_R)}=\mathcal{O}\!\left(R^{\frac{2-\gamma}{m-1}}\right) \qquad \text{as}\;\;R\to+\infty\,.
        \end{equation*}
	
	\subsection{Local smoothing estimates and \emph{a priori} boundedness}\label{subsec:apriori}
	In order to prove the uniqueness Theorem \ref{uniq thm}, we first obtain quantitative uniform estimates for locally bounded solutions, which  might be defined only in space-time sets of the form $\Omega\times(\tau_1,\tau_2)$, where $\Omega \subset \R^N$ is a (possibly unbounded) domain and $0<\tau_1<\tau_2$. That is, we consider the purely local problem
	\begin{equation}\label{wpme-noid-loc}
		\rho \, u_t = \Delta(u^m) \qquad \text{in}\; \; \Omega\times(\tau_1,\tau_2) \, .
	\end{equation}

	
	
	

    	\begin{den}[Local very weak solutions]\label{vw loc def}
        Let $ \Omega \subset \R^N $ be a (possibly unbounded) domain and $ 0 \le \tau_1 < \tau_2 $. We say that a function
        $$
        u \in L^1_{\rho,\mathrm{loc}}\!\left(\Omega \times(\tau_1,\tau_2)\right)\cap L^m_{\mathrm{loc}}\!\left(\Omega\times(\tau_1,\tau_2)\right)
        $$
        is a very weak solution to \eqref{wpme-noid-loc} if
		\begin{equation*}\label{vw-cond-loc}
			\int_{\tau_1}^{\tau_2} \int_{\Omega} \left(u \, \psi_t \, \rho + u^m \, \Delta \psi \right)  dx dt = 0
		\end{equation*}
		for all $\psi\in C^\infty_c\!\left(\Omega\times (\tau_1,\tau_2)\right)$.
        \end{den}

    Next we introduce the concept of \emph{strong energy} solution to \eqref{wpme-noid-loc}, which is \emph{a priori} more restrictive than Definition \ref{vw loc def} (but \emph{a posteriori} it is not -- see Corollary \ref{very weak strong energy} below).

	\begin{den}[{Local strong energy solutions}]\label{def:w-loc}
 Let $ \Omega \subset \R^N $ be a (possibly unbounded) domain and $ 0 \le \tau_1 < \tau_2 $. We say that a function $ u $ is a strong energy solution to  \eqref{wpme-noid-loc} if
		\begin{equation}\label{w-hyp-loc}
			u\in L^\infty_{{\mathrm{loc}}}(\Omega\times(\tau_1,\tau_2))\,,
		\end{equation}
        \begin{equation}\label{w-hyp-energ}
            u^m \in H^1_{{\mathrm{loc}}}\!\left((\tau_1,\tau_2); L^2_{{\rho,\mathrm{loc}}}(\Omega)\right) \cap L^2_{{\mathrm{loc}}}\!\left((\tau_1,\tau_2); H^1_{{\mathrm{loc}}}(\Omega)\right) ,
     \end{equation}
		and for any $ \tau_1<t_1<t_2<\tau_2 $ it holds
		\begin{equation}\label{w-cond-loc}
			\int_{t_1}^{t_2} \int_{\Omega} \left( - u \, \psi_t \, \rho  + \left\langle \nabla (u^m) \, , \nabla \psi \right\rangle \right) dx dt =
            \int_{\Omega} \left[u(t_1)\,\psi(t_1) - u(t_2)\,\psi(t_2)\right] \rho\, dx \, ,
		\end{equation}
		for all $\psi\in L^\infty\!\left(\Omega\times(\tau_1,\tau_2)\right) \cap H^1\!\left((\tau_1,\tau_2); L^2_\rho(\Omega)\right) \cap L^2\!\left((\tau_1,\tau_2); H^1(\Omega)\right)$ such that $ \operatorname{supp} \psi \Subset \Omega \times [t_1,t_2] $.
	\end{den}

We are now in a position to state our local smoothing effect.

	\begin{thm}[Local smoothing estimates for unsigned solutions]\label{local moser iter lem}
		Let $N\geq3$, $m>1$, and $\rho$ be a measurable function satisfying \eqref{weight-cond}. Let $ \Omega \subset \R^N $ be a (possibly unbounded) domain and $ 0 \le \tau_1 < \tau_2 $.
		Let $u$ be a solution to~\eqref{wpme-noid-loc}, in the sense of Definition \ref{vw loc def},  satisfying \eqref{w-hyp-loc}. Assume that there exists $R \ge 1$ such that $B_{2R} \Subset \Omega$, and for any $ \tau_1<t_1<t_2<T^*<\tau_2 $ consider the pair of cylinders
		\begin{equation*}
			\overline{Q}=B_{2R}\times(t_1,T^*) \, , \qquad \underline{Q}=B_{R}\times(t_2,T^*) \, .
		\end{equation*}
		Then the following smoothing estimates hold for all $ \varepsilon \in (0,\varepsilon_0) $:
		\begin{equation}\label{l.infty.balls}
			\begin{aligned}
				R^{-\frac{2-\gamma}{m-1}} \norm{u}_{L^\infty\left(\underline{Q}\right)} & \leq C_1 \left[  \left(R^{-(N-\gamma)-\frac{2-\gamma}{m-1}}\norm{u}_{L_\rho^{1}\left(\overline{Q}\right)}\right)^{\frac{1}{2-m}}+\left(\tfrac{1}{t_2-t_1}\right)^{\frac{1}{m-1}} \right] \quad \text{for}\;\; m\in(1,2) \, ,\\
				R^{-\frac{2-\gamma}{m-1}} \norm{u}_{L^\infty\left(\underline{Q}\right)} & \leq  C_2\left[\left(R^{-(N-\gamma)-\frac{2-\gamma}{m-1}}\,\mathsf{S}\right)^{1+\varepsilon}\left(T^*-t_1\right)^{\frac{\varepsilon}{m-1}}+\left(\tfrac{1}{t_2-t_1}\right)^{\frac{1}{m-1}}\right] \quad
				 \text{for}\;\; m>1 \, ,
			\end{aligned}
		\end{equation}
		 where
        $$
        \mathsf{S}=\sup\limits_{t \in (t_1,T^*) } \int _{B_{2R}} |u(x,t)|\,\rho(x)\,dx \, ,
        $$
        $ \varepsilon_0 > 0 $ is a constant depending on $N,m,\gamma $, $C_1>0$ is a constant depending on $N,m,\gamma,\underline{C},\overline{C}$, and $C_2>0$ is a constant depending on $N,m,\gamma,\underline{C},\overline{C},\varepsilon$.
	\end{thm}	

\begin{rem}[General cylinders] \label{GB} \rm
In order to emphasize the most relevant aspects of Theorem \ref{local moser iter lem}, \emph{i.e.}\ the pointwise growth rate of solutions as $R\to+\infty$, the statement is written only for cylinders of large radii $R$ and $2R$. From the proof (given in Section \ref{sec:local}), however, one can check that analogous estimates also hold for general ordered radii $0<r<R$ -- see in particular \eqref{aaa-bis} for $m\in(1,2)$ and \eqref{aaa-bis-gen} for any $m>1$. Also, it would be possible to obtain similar estimates for cylinders that are not necessarily centered at $ x_0 = 0 $, but in this case the argument becomes technically more involved because one needs to take into account the relative positions of the origin itself and the balls $ B_r(x_0),B_R(x_0) $, in the spirit of the proof of \cite[Theorem 2.1]{BS1}. Nevertheless, it is clear that if $ 0 \in B_r(x_0) $ or $ 0 \not\in B_R(x_0) $, the weight can be handled as in the case $x_0=0$ or treated as a constant (namely as if $\gamma=0$), respectively.
        \end{rem}

The above theorem requires \emph{a priori}, via \eqref{w-hyp-loc}, that very weak solutions are locally bounded. The goal of the next result is to show that, at least for non-negative solutions, this is always the case. In fact, it is a weighted variant of \cite[Theorem 1.1]{DK2}, in whose unweighted framework the authors even proved that non-negative very weak solutions are space-time \emph{continuous}. In our case however, the inclusion of a (possibly rough) weight $\rho$ precludes the possibility of applying H\"older-regularity results in the spirit of \cite{DiBV}; therefore we need to replace continuity with local boundedness, which turns out to be sufficient for our purposes.
	\begin{thm}[Local boundedness]\label{apriori-bounded}
		Let  $N\geq3$, $m>1$, and $\rho$ be a measurable function satisfying \eqref{weight-cond}. Let $ \Omega \subset \R^N $ be a (possibly unbounded) domain and $ 0 \le \tau_1 < \tau_2 $. Let $u$ be a non-negative solution to \eqref{wpme-noid-loc}, in the sense of Definition \ref{vw loc def}. Then $u\in L^\infty_{{\mathrm{loc}}}(\Omega\times(\tau_1,\tau_2)) $.
	\end{thm}

    We stress that such a result is crucially exploited both in the proof of Theorem \ref{thm:weighted-AC} and in that of Theorem \ref{uniq thm}, as it ensures that non-negative solutions can always be approximated (locally) by sequences of more regular solutions (for the details of this construction we refer to the proof of Theorem \ref{local moser iter lem} in Section \ref{sec:local}).

    \begin{cor}[Local smoothing estimates for non-negative solutions]\label{glob-unif-est}
		Let the hypothes\-es of Theorem \ref{local moser iter lem} hold, with \eqref{w-hyp-loc} replaced by $u \ge 0$. Then the smoothing estimates \eqref{l.infty.balls} hold.
	\end{cor}

\begin{cor}[Very weak solutions are strong energy]\label{very weak strong energy}
    Let  $N\geq3$, $m>1$, and $\rho$ be a measurable function satisfying \eqref{weight-cond}. Let $ \Omega \subset \R^N $ be a (possibly unbounded) domain and $ 0 \le \tau_1 < \tau_2 $. Let $u$ be either a locally bounded or a non-negative very weak solution to \eqref{wpme-noid-loc}, in the sense of Definition \ref{vw loc def}. Then $u$ is a strong energy solution, in the sense of Definition \ref{def:w-loc}. In particular,
    \begin{equation}\label{Cont-Lp}
        u\in C\big((\tau_1,\tau_2);L^p_{\rho,\mathrm{loc}}(\Omega)\big) \qquad \forall p \in [1,\infty) \,.
        \end{equation}
\end{cor}

    \begin{rem}[On the assumption $N \ge 3$]\label{N}\rm
 For simplicity, we assumed throughout that $ N\ge3 $; however, we point out that it is crucially needed for our methods of proof only in Theorems \ref{thm:weighted-AC} and \ref{uniq thm}, where the existence of the global Green's function for $-\Delta$ is essential. All of the other main results, up to mild technical modifications, also hold in lower dimensions.  Indeed, the key observation is to replace the weighted Sobolev inequality \eqref{Sobolev.inequality} with Gagliardo--Nirenberg-type inequalities, see the discussion in \cite[Remark 4.4]{MP}. Note that for the validity of the latter the hypotheses \eqref{weight-cond} are in general too weak, but at least in the model case $ \rho(x) = |x|^{-\gamma} $ they are well known to hold, see again \cite{MP} and references therein.
    \end{rem}

    \subsection{Discussion of optimality and comparison to known results}
We now provide some technical comments on the sharpness of our main results. For simplicity, we will work in the time interval $ (0,T) $.

    \subsubsection{On the sharpness of Aronson--Caffarelli estimate}\label{ac space optimal}
		Let us begin by observing that if one fixes any $\varepsilon,t,\delta$ according to the statement of Theorem \ref{thm:weighted-AC}, it clearly follows that
		\begin{equation*}\label{init-datum-growth}
			\mu(B_R)=\mathcal{O}\!\left(R^{N-\gamma+\frac{2-\gamma}{m-1}}\right)\qquad\text{as} \;\; R\to+\infty\,.
		\end{equation*}
		Furthermore, estimate \eqref{weighted-AC-inq} applied to $ u(t) \equiv u(t+\tau) $ yields
     \begin{equation}\label{solution-l1-time}
        \begin{aligned}
			\int_{B_R}u(\tau)\,\rho\,dx \leq C \left[t^{-\frac{1}{m-1}}R^{N-\gamma+\frac{2-\gamma}{m-1}}+t^{\frac{N-\gamma}{2-\gamma}} \left(\fiint_{B_\varepsilon\times (t+\tau,t+\tau+\delta)}u\,\rho\,dxds\right)^{1+\frac{N-\gamma}{2-\gamma}(m-1)}\right],
            \end{aligned}
        \end{equation}
		for all $\tau\in(0,T)$, $t\in(0,T-\tau)$ and small enough $\delta>0$, hence it follows
		\begin{equation}\label{solution-l1-growth}
			\int_{B_R} u(t)\,\rho\,dx=\mathcal{O}\!\left(R^{N-\gamma+\frac{2-\gamma}{m-1}}\right) \qquad \text{as}\;\;R\to+\infty \,,\ \text{for all}\;\;t\in(0,T)\,.
		\end{equation}
		In fact, such estimates are sharp as concerns the spatial behavior, as can be seen by considering suitable explicit or semi-explicit solutions that satisfy \eqref{solution-l1-growth} with $\mathcal{O}$ replaced by $\approx$. We refer to \cite[Subsection 2.4]{MP}, where these solutions were constructed. Specifically, they exhibit a complete blow up at $ t=T $, and in the pure power case $\rho(x)=|x|^{-\gamma}$ they can be taken of the form (up to numerical multiplying constants)
        \begin{equation}\label{blowupsol}
            U(x,t) = |x|^{\frac{2-\gamma}{m-1}}\,(T-t)^{-\frac{1}{m-1}}\, ,
        \end{equation}
so the spatial growth is precisely the one predicted by \eqref{solution-l1-growth}. We stress that the same holds even for a general weight satisfying \eqref{weight-cond}, see \cite[Proposition 2.4]{MP}, the only difference being the replacement of $ |x|^{{(2-\gamma)}/{(m-1)}} $ with a non-explicit function that is asymptotically equivalent as $ |x| \to +\infty $.

		To observe the sharpness of \eqref{solution-l1-time} with respect to time as it approaches the blow-up threshold, it suffices to fix any suitable $R,\varepsilon$ and then set
		\begin{equation*}
			\delta=\tfrac{T-\tau}2 \, , \quad t=\tfrac{\delta}{2}\,.
		\end{equation*}
		Under this choice, we see that both sides of \eqref{solution-l1-time} behave exactly like $\delta^{-{1}/{(m-1)}}$ as $\delta\to0$ (namely $ \tau \to T $), when considering the solutions mentioned in \eqref{blowupsol}.
		
		Finally, we notice that \eqref{solution-l1-time} is also sharp, for large times, when considering the weighted Barenblatt solutions (see \cite{RV3})
		\begin{equation*}
			U_B(x,t)=t^{-\frac{N-\gamma}{(N-\gamma)(m-1)+2-\gamma}}\left(c_1-c_2\,t^{-\frac{2-\gamma}{(N-\gamma)(m-1)+2-\gamma}}\,|x|^{2-\gamma}\right)^{\frac{1}{m-1}}_+ ,
		\end{equation*}
		for suitable positive constants $c_1,c_2$, which solve \eqref{measure} with $\rho(x)=|x|^{-\gamma}$ and $ \mu $ a multiple of the Dirac delta centered at the origin. Indeed, in this case we can fix  $R,\varepsilon,\delta$, choose $\tau$ proportional to $t$, and let $t \to +\infty$; in this way, we see through direct inspection that the right-hand side of \eqref{solution-l1-time} goes to zero exactly with the same rate as the left-hand side (namely the time power in front of the positive part in the above formula for $U_B$). Similar comments apply to more general $L^1_\rho$ solutions, thanks to the smoothing effect \eqref{basic smoothing}.

    \subsubsection{Blow-up and optimality of the smoothing effect}\label{gen-rad-rem}
     Let us now change focus to the sharpness of the local smoothing estimates \eqref{l.infty.balls} by testing them  against a few types of explicit or semi-explicit solutions.
    However, before we begin, let us emphasize again that such estimates hold even for purely local and unsigned solutions. In particular, the latter may blow up at the boundary of $ \Omega \times (0,T) $, or they may exhibit a wild sign-changing behavior that is difficult to capture. For these reasons, it is not surprising that \eqref{l.infty.balls} is far from optimal on certain subclasses of solutions, for instance global $ L^1_\rho $ solutions.

    That being said, let us first go back to the finite-time blow-up solutions introduced in \eqref{blowupsol}.
        If we plug them into \eqref{l.infty.balls}, we immediately see that both cases ($ m \in (1,2) $ and $m>1$) are sharp with respect to $R\to+\infty$, \emph{i.e.}\ the left- and right-hand sides are stable. However, time-optimality is slightly more delicate. Indeed, for a $T$ fixed (the blow-up time of $U$), let us consider \eqref{l.infty.balls} for $\overline{Q}=B_{2R}\times(0,T-\delta)$ and $\underline{Q}=B_R\times (T/2,T-\delta)$, with a small enough $ \delta>0 $. We have
        \begin{equation*}
            \delta^{-\frac{1}{m-1}}\leq \begin{cases}
                C\,\delta^{\frac{m-2}{m-1}\frac{1}{2-m}}=C\,\delta^{-\frac{1}{m-1}} & \text{for}\;\; m\in(1,2)\,,\\
                C\,\delta^{-\frac{1+\varepsilon}{m-1}}& \text{for} \;\; m>1\,,
            \end{cases}
        \end{equation*}
        for a suitable $C>0$ independent of $\delta$, which shows optimality when $m\in(1,2)$, and exhibits almost optimality otherwise, since the power $\tfrac{1}{m-1}$ on the right-hand side of the second estimate in \eqref{l.infty.balls} cannot be replaced by a smaller one. We do not know whether the presence of $\varepsilon$ is an artifact of our proof, or if it is actually necessary to reflect the behavior of certain critical solutions.\ On the other hand, by means of a similar reasoning, it is not difficult to check that if we consider cylinders that ``collapse'' at the blow-up time, namely $\overline{Q}=B_{2R}\times(T-3\delta,T-\delta)$ and $\underline{Q}=B_R\times (T-2\delta,T-\delta)$, and let $ \delta \to 0 $, then sharpness is recovered regardless of $\varepsilon$, since the term $ (T^*-t_1)^{\varepsilon/(m-1)} $ in \eqref{l.infty.balls} compensates for the $1+\varepsilon$ non-optimal exponent, so that both sides of the estimate behave like $ \delta^{-1/{(m-1)}} $.

        Remarkably, the solution $U$ becomes time-integrable in a neighborhood of the blow-up time $T$ precisely when $m>2$. For this reason, it is unavoidable to split \eqref{l.infty.balls} into those two cases, otherwise the first estimate would be inconsistent with such an integrability property.

        The second special type of solutions we  consider are intrinsically local: the separable \emph{friendly giant} solution $V$, defined by
        \begin{equation*}
            V(x,t)=W(x)\,t^{-\frac{1}{m-1}} \qquad \text{for} \;\; (x,t) \in \Omega\times(0,+\infty)\,,
        \end{equation*}
        for some bounded smooth domain $\Omega\subset\R^N$, where $W \ge 0$ satisfies
        \begin{equation*}
        \begin{cases}
            -\Delta W^m=\tfrac{1}{m-1}\,\rho\,W & \text{in}\;\;\Omega\,,\\
            W=0 & \text{on}\;\;\partial\Omega\,.
        \end{cases}
        \end{equation*}
        We refer to \cite{AP,DK2} for more information regarding the existence and main properties of this type of solutions in the unweighted framework, and to \cite{BK} and the specific bibliography quoted therein in the weighted setting.        In particular, such solutions are guaranteed to exist and be bounded if $\gamma < 2$. Let us check the temporal optimality of \eqref{l.infty.balls} (near the origin) when applied to $V$ by considering cylinders of the form $\overline{Q}=B_{2R}\times(\delta,T)$ and $\underline{Q}=B_{R}\times\left(2\delta,T\right)$, for any $T>0$, $\delta<T/2$, and $R \ge 1$ such that $B_{2R} \Subset \Omega$. We obtain
        \begin{equation*}
			\delta^{-\frac{1}{m-1}}\leq
			\begin{cases}
				C\,\delta^{-\frac{1}{m-1}}& \text{for}\;\; m\in(1,2) \, ,\\
				C\left(\delta^{-\frac{1+\varepsilon}{m-1}}+\delta^{-\frac{1}{m-1}}\right) \,
				& \text{for}\;\; m>1 \, ,
			\end{cases}
		\end{equation*}
        for a suitable $ C>0 $ independent of $\delta$. Once again, this shows optimality as $\delta \to 0$ when $m\in(1,2)$, and exhibits almost optimality otherwise, since the power $\tfrac{1}{m-1}$ cannot be replaced by a smaller one.

    \subsubsection{Comparison of \eqref{l.infty.balls} to previous results}
        The original inspiration for Theorem \ref{local moser iter lem} comes from the classical paper \cite{DK}; see also \cite[Section 1.2]{DasK} for a streamlined presentation. In that work none of the exponents are calculated, so the authors are forced to make additional use of an explicit scaling property of the equation to recover the expected spatial growth rate of solutions from the estimate. However, this is not possible in our inhomogeneous setting, and we manage to prove the sharp spatial growth rate differently, with a method that we believe to be applicable in several frameworks.

        A sharp supremum estimate is actually found in \cite[Remarks 3.1--3.2]{Andreucci}, where, \it still in the unweighted case\rm, a cylindrical $L^p \to L^\infty$ smoothing estimate for all $p>m-1$ is proved, which corresponds to the estimate we obtain by coupling \eqref{iteration6}--\eqref{linfty contr mod bis} via further interpolation, when $m\in(1,2)$, whereas for $ m \ge 2 $ the bound in \cite{Andreucci} is to some extent weaker than ours, since it involves an $L^p$ norm in the right-hand side with $p > m-1 \ge 1$. One may also observe that in \cite{Andreucci} the lack of sharp time optimality appears in much the same way as appears here (recall the role of $\varepsilon$).

        Finally, we mention that in the recent papers \cite{BS1,BS2}, similar estimates are obtained for solutions to a weighted fast diffusion equation, and the authors take the additional care to even distinguish whether or not the domains contain the origin (see Remark \ref{GB}). In that context, it is possible to obtain supremum estimates in terms of the weighted cylindrical $L^1$ norm, at least in the range $m\in(m_c,1)$, for an explicit critical exponent $m_c\in(0,1)$ related to Caffarelli--Kohn--Nirenberg inequalities; for the details we refer to \cite[Section 2]{BS1}. We stress that in such a supercritical range
        they get smoothing effects that are coherent with the ones we establish in the case $m\in(1,2)$.

	\subsection{Organization of the paper}\label{po}
	    The rest of the paper is organized as follows. In Section \ref{sec:AC}, we prove the Aronson--Caffarelli-type estimate Theorem \ref{thm:weighted-AC}, and along the way we establish Corollary \ref{small time AC lem}, which concerns the $ L^1_\rho $ growth rate of solutions for positive times. Section \ref{sec:local} contains the proof of the  local smoothing estimate Theorem \ref{local moser iter lem} and of Corollaries \ref{glob-unif-est}--\ref{very weak strong energy}, whereby the estimate also holds for non-negative  very weak solutions. In Section \ref{sec:unique}, we exploit the results of the previous sections to prove Theorem \ref{uniq thm}, that is, uniqueness for non-negative solutions without growth assumptions, and Theorem \ref{mBCP} concerning existence of unsigned solutions with measure data. Finally, we reserve to Section \ref{sec:loc-bdd} the proof of Theorem \ref{apriori-bounded}, namely that non-negative solutions are locally bounded, which is implicitly used at various points of the paper and is shown independently.

	\section{Aronson--Caffarelli-type estimate and initial traces: proofs}\label{sec:AC}
	
	The goal of this section is to prove Theorem \ref{thm:weighted-AC}; this will be carried out in two main steps. In the first one, we will establish the Aronson--Caffarelli-type estimate for non-negative \emph{constructed solutions} of 
     \begin{equation}\label{wpme-funid}
 \begin{cases}
 \rho\, u_t = \Delta \!\left( u^m \right) &\text{in}\;\; \R^N\times(0,T) \, , \\
 u = u_0  & \text{on} \;\; \R^N \times\{ 0 \} \, ,
     \end{cases}
     \end{equation}
    taking a bounded and integrable initial datum $u_0$, by suitably adapting the Green's function approach introduced in \cite{BV}. In the second one, we will carefully pass from constructed solutions to general given solutions, upon resorting to delicate local and global comparison results (see Proposition \ref{local comparison lemma} and Corollary \ref{global comparison lemma} below).

    	Since our strategy centers on the use of \emph{Green's functions}, we recall the classical definition in the whole Euclidean space along with key convergence properties.
	\begin{lem}[Euclidean Green's function and approximations]\label{defgreen}
		Let $N\geq3$. Let us consider the Green's function for $-\Delta$ in $\R^N$
		\begin{equation}\label{green-form}
			G(x,y)=\tfrac{1}{N(N-2)\,\alpha(N)}\left|x-y\right|^{2-N} ,
		\end{equation}
		where $\alpha(N)$ denotes the volume of the $N$-dimensional unit ball, so that
		$-\Delta_x G = \delta_y$ in the sense of distributions in $\R^N$.
		For each $n\in\mathbb{N}$, let us define the smooth approximations $G_n\in C^\infty\!\left(\R^{2N}\right)$ of $G$ by
		\begin{equation}\label{green-approx}
			G_n(x,y)=\tfrac{1}{N(N-2)\,\alpha(N)}\left(\tfrac{n^2}{1+n^2|x-y|^2}\right)^{\frac{N-2}{2}} .
		\end{equation}
		Then, it holds that $G_n\nearrow G$ as $ n \to \infty $ and
		\begin{equation}\label{green-lapl}
			\lim_{n\to \infty}\int_{\R^N} \Delta_x G_n(x,y)\,f(x)\,dx = -f(y)
		\end{equation}
		for all  $ f \in L^\infty_c\!\left(\R^N \right)$ and a.e.~$y\in\R^N$.
	\end{lem}
    \begin{proof}
We will only prove \eqref{green-lapl}, since all of the other properties are rather standard. To this aim, it is convenient to write $G_n$ as
$$
 G_n(x,y)= n^{N-2} \, g(n |x-y|) \, , \quad \text{where} \; \;  g(r) = \tfrac{1}{N(N-2)\,\alpha(N)} \, \left(1+r^2\right)^{-\frac{N-2}{2}} .
$$
In particular, interpreting $g(r)$ as $g(|x|)$, we infer
$$
\Delta_x G_n(x,y) = n^N \, \Delta g (n|x-y|) \, ,
$$
and a direct computation reveals that
\begin{equation}\label{Delta-g}
 -\Delta g (|x|) = \tfrac{1}{\alpha(N)} \left( 1 + |x|^2\right)^{-\frac{N+2}{2}} .
\end{equation}
Now, let us deal with the integral appearing in \eqref{green-lapl}:
\begin{equation}\label{Delta-g-bis}
\begin{aligned}
\int_{\R^N} \Delta_x G_n(x,y)\,f(x)\,dx = & \, n^N \int_{B_R(y)} \Delta g (n|x-y|) \,f(x)\,dx  \\
= &  \, \underbrace{n^N \, f(y) \int_{B_R(y)} \Delta g (n|x-y|) \,dx}_{I} \\
 & + \underbrace{n^N \int_{B_R(y)} \Delta g (n|x-y|) \,[f(x)-f(y)]\,dx}_{II} \, ,
\end{aligned}
\end{equation}
where $ R > |y| + R_0 $ and $R_0$ is such that $\mathrm{supp} f \subset B_{R_0} $. As concerns $I$, we have:
\begin{equation}\label{Leb-p-0}
\begin{aligned}
I = f(y) \,  \int_{B_{nR}} \Delta g (|x|) \,dx \underset{n \to \infty}{\longrightarrow}  f(y) \,  \int_{\R^N} \Delta g (|x|) \,dx = & - f(y) \, \tfrac{1}{\alpha(N)}  \int_{\R^N} \left( 1 + |x|^2\right)^{-\frac{N+2}{2}} dx \\
= & - f(y) \, ,
\end{aligned}
\end{equation}
since it is not hard to show that $\alpha(N)=\int_{\R^N} \left( 1 + |x|^2\right)^{-{(N+2)}/{2}} dx$.
On the other hand,
\begin{equation}\label{Leb-p}
|II| \le \tfrac{1}{\alpha(N)} \, n^N \int_{B_{R/n}(y)} |f(x)-f(y)|\,dx  + 2 \, n^N \left\| f \right\|_{L^\infty\left( \R^N \right)} \int_{B_{R/n}^c(y)} \left| \Delta g (n|x-y|) \right| dx \, .
\end{equation}
If $y$ is a Lebesgue point for $ f $, the first term on the right-hand side~of \eqref{Leb-p} vanishes as $ n \to \infty $, for all $R$ as above. As for the second one, from \eqref{Delta-g} it follows that
\begin{equation}\label{Leb-p-2}
 n^N \int_{B_{R/n}^c(y)} \left| \Delta g (n|x-y|) \right| dx = \tfrac{1}{\alpha(N)} \int_{B_R^c}  \left( 1 + |x|^2\right)^{-\frac{N+2}{2}} dx \, .
\end{equation}
        Hence, \eqref{green-lapl} is a consequence of letting first $ n \to \infty $ and then $ R \to +\infty $ in \eqref{Delta-g-bis}, using \eqref{Leb-p-0}, \eqref{Leb-p}, and \eqref{Leb-p-2}.
    \end{proof}

	\subsection{An intermediate result for constructed solutions}\label{IRes}
	Let us first recall the notion of weak energy solution to \eqref{wpme-funid} for an initial datum $u_0$ in $L^1_\rho\!\left(\mathbb{R}^N\right) \cap L^\infty\!\left(\mathbb{R}^N\right)$, which is by now standard.
	\begin{den}[Weak energy solutions]\label{def2}
        Let $u_0\in L^1_\rho\!\left(\mathbb{R}^N\right) \cap L^\infty\!\left(\mathbb{R}^N\right)$. We say that a function $ u $ is a weak energy solution of problem \eqref{wpme-funid}, with $T=+\infty$, if
		\begin{equation*}\label{energy}
			u \in  C\!\left([0,+\infty); L^1_\rho\!\left(\mathbb{R}^N\right) \right)\cap\, L^\infty\!\left( \mathbb{R}^N \times (0 , +\infty) \right) , \quad u^m \in L^2_{\mathrm{loc}} \big([0,+\infty); \dot{H}^1\!\left(\mathbb{R}^N \right)\! \big) \, ,
		\end{equation*}
		$ u(0)=u_0 $ and
		\begin{equation}\label{energy-formulation}
			\int_0^{+\infty} \int_{\mathbb{R}^N} u\, \psi_t \, \rho \, dx dt  = \int_0^{+\infty} \int_{\mathbb{R}^N}  \left\langle \nabla (u^m) \, , \nabla \psi \right\rangle  dx dt
		\end{equation}
		for all $\psi\in C^\infty_c\!\left( \mathbb{R}^N\times (0, +\infty) \right)$.
	\end{den}
	
	In the following statement we collect all the relevant facts about constructed solutions to \eqref{wpme-funid}, in the sense of the above definition.
	\begin{pro}\label{pro1}
		Let $u_0\in L^1_\rho\!\left(\mathbb{R}^N\right)\,\cap L^\infty\!\left(\R^N\right)$. There exists a unique weak energy solution $u$ of problem \eqref{wpme-funid} (with $T=+\infty$), in the sense of Definition \ref{def2}, satisfying
		\begin{equation}\label{m37a}
			\left\| u(t) \right\|_{L^\infty\left(\mathbb{R}^N\right)} \leq \left\| u_0 \right\|_{L^\infty\left(\mathbb{R}^N\right)} \qquad \forall t\geq 0 \, .
		\end{equation}
		Furthermore, there exists a constant $K>0$ depending only on $N,m, \gamma, \underline{C}, \overline{C}$ such that
		\begin{equation}\label{basic smoothing}
			\norm{u(t)}_{L^\infty(\R^N)}\leq K\,t^{-\lambda}\norm{u_0}_{L^1_\rho(\mathbb{R}^N)}^{\theta\lambda}\qquad \forall t > 0 \,.
		\end{equation}
		Next, if $ v $ is the weak energy solution corresponding to another initial datum $v_0\in L^1_\rho\!\left(\mathbb{R}^N\right) \cap L^\infty\!\left(\mathbb{R}^N\right)$, then the following $ L^1_\rho $-contraction/stability estimate holds:
		\begin{equation}\label{m37}
			\left\| u(t) - v(t) \right\|_{L^1_\rho\left(\mathbb{R}^N\right)} \leq \left\|u_0 - v_0 \right\|_{L^1_\rho\left(\mathbb{R}^N\right)}  \qquad \forall t\geq 0 \, .
		\end{equation}
		Solutions are strong, in the sense that
		\begin{equation}\label{strong energy}
			u_t \in L^\infty_{\mathrm{loc}}\! \left( (0,+\infty) ; L^1_\rho\!\left(\mathbb{R}^N\right) \right) \qquad \text{and} \qquad \partial_t\!\left( u^{\frac{m+1}{2}} \right) \in L^2_{\rho,\mathrm{loc}}\!\left( \R^N \times (0,+\infty) \right).
		\end{equation}
		If in addition $ v_0 \le u_0 $ the corresponding weak energy solutions are ordered, that is $v \le u$. In particular, $ u_0 \ge 0 $ implies $ u \ge 0 $, and  in such case the inequality
		\begin{equation}\label{weak BC estimate}
			\rho \, u_t \geq - \frac{\rho \, u}{(m-1) \, t}  \qquad \text{a.e.\ in} \; \; \mathbb{R}^N \times (0,+\infty)
		\end{equation}
	holds Still under the assumption $u_0\geq0$, mass conservation holds:
		\begin{equation}\label{mass cons}
			\int_{\R^N} u(t) \, \rho\,dx=\int_{\R^N}u_0 \, \rho \,dx  \qquad \forall t>0 \, ,
		\end{equation}
		and, for all $0<t_0<t_1$, the inequality
		\begin{equation}\label{flb equation}
			\int_{\R^N}\left[u(x,t_0)-u(x,t_1)\right]G(x,x_0)\,\rho(x)\,dx\leq(m-1)\,\frac{t_1^{\frac{m}{m-1}}}{t_0^{\frac{1}{m-1}}}\,u^m(x_0,t_1)
		\end{equation}
		is satisfied for a.e.\ $ x_0 \in \R^N $ (depending on $t_0,t_1)$, where $G(x,x_0)$ is the Green's function \eqref{green-form} centered at $x_0$.
	\end{pro}
	\begin{proof}
		It is only needed to establish \eqref{basic smoothing}, \eqref{mass cons}, and \eqref{flb equation}, because the other items have already been proven in \cite[Proposition 3.3]{MP}. Under the running assumptions on $\rho$, the smoothing estimate \eqref{basic smoothing} and the mass conservation \eqref{mass cons} are by now standard. The former was shown in \cite[Theorem 2.1]{RV2}, but under the additional requirements that the solution is positive and that the weight $\rho$ is smooth. However, it is easily seen that such extra assumptions are inessential, the key hypotheses in order to prove both \eqref{basic smoothing} and \eqref{mass cons} being \eqref{weight-cond}; in fact, similar smoothing effects and mass-conservation properties have already been derived in \cite{MP,MPQ} under \eqref{weight-cond}.
		
		Let us now focus on the proof of \eqref{flb equation}, which requires a few modifications from the original one  \cite[Proposition 4.2]{BV} (in a fractional framework actually). Nevertheless, given the importance of such an estimate, we provide the complete argument. The first step is to derive from \eqref{weak BC estimate} the following monotonicity result:
		\begin{equation*}\label{weak-monotone}
			t \mapsto t^{\frac{m}{m-1}}\,u^m(t) \; \; \text{is essentially non-decreasing} \, .
		\end{equation*}
Indeed, this is a simple consequence of \eqref{strong energy}--\eqref{weak BC estimate}, since the above function is absolutely continuous (\emph{e.g.}\ with values in $ L^1_\rho\!\left( \R^N \right) $) and
$$
\partial_t \left( t^{\frac{m}{m-1}}\,u^m  \right) = m \, t^{\frac{m}{m-1}} \, u^{m-1} \left[ \frac{u}{(m-1) \, t} + u_t \right] \ge 0 \qquad \text{a.e.\ in} \; \; \R^N \times (0,+\infty) \, .
$$
In particular, thanks to the time continuity of $u$, we can deduce that for \emph{every} $ t>s>0 $ it holds
\begin{equation}\label{almost-t-mon}
  s^{\frac{m}{m-1}}\,u^m(x,s)   \le t^{\frac{m}{m-1}}\,u^m(x,t)  \qquad \text{for a.e.} \; \; x \in \R^N \; \; \text{(depending on $t,s$)} \, .
\end{equation}
        Now, in order to obtain \eqref{flb equation}, we multiply the differential equation in \eqref{wpme-funid} by the test function
        \begin{equation*}\label{phi-st}
        \phi(x,t) = G_n(x,x_0) \,\phi_R(x-x_0)  \, \tilde{\chi}_{[t_0,t_1]}(t) \in C^\infty_c\!\left(\R^N\times(0,+\infty)\right) ,
        \end{equation*}
        where $G_n$ is a smooth and monotone approximation of the Green's function defined in \eqref{green-approx}, $\tilde{\chi}$ is a smooth approximation of the characteristic function $ \chi_{[t_0,t_1]} $, and $\phi_R$ is a cutoff function introduced in \eqref{defcutoff1}--\eqref{cutoff-est}. Let us first integrate by parts in time; this gives (recall \eqref{strong energy})
		\begin{equation*}\label{first-tint}
			\begin{aligned}
				& \int_{0}^{+\infty}\int_{\R^N}u_t(x,t)\,G_n(x,x_0)\,\phi_R(x-x_0) \, \tilde{\chi}_{[t_0,t_1]}(t) \,\rho(x)\,dxdt\\
				= & -\int_{0}^{+\infty}\int_{\R^N}u(x,t)\,G_n(x,x_0)\,\phi_R(x-x_0) \, \partial_t\tilde{\chi}_{[t_0,t_1]}(t)\,\rho(x)\,dxdt\,.
			\end{aligned}
		\end{equation*}
        Since $u$ is continuous with values in $ L^1_\rho\!\left( \R^N \right) $, upon suitably letting $ \partial_t\tilde{\chi}_{[t_0,t_1]} \to \delta_{t_0}-\delta_{t_1} $ we end up with
		\begin{equation}\label{tint-2}
			\begin{aligned}
				&\int_{t_0}^{t_1}\int_{\R^N}u_t(x,t)\,G_n(x,x_0)\,\phi_R(x-x_0)\,\rho(x)\,dxdt\\
				=&\int_{\R^N}\left[u(x,t_1)-u(x,t_0)\right]G_n(x,x_0)\,\phi_R(x-x_0)\,\rho(x)\,dxdt\,.
			\end{aligned}
		\end{equation}
		 Next, we integrate by parts in space on the right-hand side of the corresponding identity, to find that \eqref{tint-2} is equal to
		\begin{equation}\label{tint-3}
			\begin{aligned}
				&\int_{t_0}^{t_1}\int_{\R^N}\Delta [u^m(x,t)]\,G_n(x,x_0)\,\phi_R(x-x_0)\,dxdt\\
				=&\int_{t_0}^{t_1}\int_{\R^N}u^m(x,t) \, \Delta[G_n(x,x_0)\,\phi_R(x-x_0)] \, dxdt\,\\
				=& \underbrace{\int_{t_0}^{t_1}\int_{\R^N}u^m(x,t) \, \Delta [G_n(x,x_0)]\,\phi_R(x-x_0)\,dxdt}_{I}\\
				&+\underbrace{\int_{t_0}^{t_1}\int_{\R^N}u^m(x,t)\left[2 \left\langle \nabla G_n(x,x_0) \, , \nabla\phi_R(x-x_0) \right\rangle+G_n(x,x_0) \,\Delta\phi_R(x-x_0)\right]dxdt}_{II} \, .
			\end{aligned}
		\end{equation}
		Let us now take $n\to\infty$ to pass $\{G_n\}$ to its limit $G$. Due to \eqref{green-lapl} and the fact that $\phi_R(0)=1$, we find that
		\begin{equation*}\label{tint-rhs}
			I \underset{n \to \infty}{\longrightarrow} \int_{t_0}^{t_1}u^m(x_0,t) \,dt \,
		\end{equation*}
		for a.e.\ $x_0\in\R^N$, more precisely for every $x_0$ that is a Lebesgue point of the function $ x \mapsto \int_{t_0}^{t_1}u^m(x,t) \,dt $. On the other hand, we apply H\"{o}lder's inequality and use \eqref{defcutoff1}--\eqref{cutoff-est}, \eqref{green-approx}, \eqref{basic smoothing}, and \eqref{mass cons} to deduce
		\begin{align*}
			|II|&\leq C\, R^{-N} \norm{u}_{L^\infty\left(\R^N\times(t_0,t_1)\right)}^{m-1}\int_{t_0}^{t_1}\int_{B_{2R}(x_0)\setminus B_R(x_0)}u(x,t)\,dxdt\\
			&\leq C\, R^{-N+\gamma}\norm{u}_{L^\infty\left(\R^N\times(t_0,t_1)\right)}^{m-1}\,(t_1-t_0)\int_{\R^N}u(x,t_0)\,\rho(x)\,dx\\
			&\leq C\, R^{-N+\gamma} \, t_0^{-\lambda(m-1)}\,(t_1-t_0)\,\norm{u_0}_{L^1_\rho\left(\mathbb{R}^N\right)}^{\theta\lambda(m-1)+1}
		\end{align*}
        for a suitable constant $C>0$ independent of $n$ and $R>|x_0|+1$, which clearly vanishes as $R\to+\infty$. Therefore we may pass to the limit first as $n\to\infty$ then as $R\to+\infty$ in \eqref{tint-2} and \eqref{tint-3}, applying the dominated convergence theorem to obtain
		\begin{equation}\label{tint-al}
			\int_{\R^N}\left[u(x,t_0)-u(x,t_1)\right]G(x,x_0)\,\rho(x)\,dx=\int_{t_0}^{t_1}u^m(x_0,t)\,dt \, .
		\end{equation}
		Finally, using \eqref{almost-t-mon}, for all $\varepsilon>0$ we infer that
		\begin{equation*}\label{tint-last}
			\begin{aligned}
				\int_{t_0}^{t_1} \int_{B_\varepsilon(x_0)}u^m(x,t)\,dxdt = & \int_{t_0}^{t_1}  t^{-\frac{m}{m-1}} \int_{B_\varepsilon(x_0)} t^\frac{m}{m-1}\,u^m(x,t)\,dx \, dt\\
				\leq & \int_{t_0}^{t_1}t^{-\frac{m}{m-1}}\,dt \, \int_{B_\varepsilon(x_0)} t_1^\frac{m}{m-1}\,u^m(x,t_1)\,dx\\
				\leq & \, (m-1) \, t_0^{-\frac{1}{m-1}} \, \int_{B_\varepsilon(x_0)} t_1^\frac{m}{m-1}\,u^m(x,t_1)\,dx \, ,
			\end{aligned}
		\end{equation*}
		where in the last inequality we have evaluated the time integral and discarded the negative term at $t=t_1$. Dividing by $|B_\varepsilon(x_0)|$ and letting $ \varepsilon \to 0 $, we end up with
        \begin{equation}\label{tint-last-2}
  \int_{t_0}^{t_1}u^m(x_0,t)\,dt  \le (m-1) \, \frac{t_1^{\frac{m}{m-1}}}{t_0^{\frac{1}{m-1}}} \, u^m(x_0,t_1) \, ,
        \end{equation}
        up to possibly refining the choice of $x_0$ (still in a full-measure set depending on $t_0,t_1$). Combining \eqref{tint-al} and \eqref{tint-last-2} completes the proof.
	\end{proof}
	

    We are now in a position to state and prove the claimed Aronson--Caffarelli-type estimate for non-negative constructed solutions of the approximate problem \eqref{wpme-funid}.

	\begin{pro}\label{weighted AC lemma}
		Let $u_0 \in L^\infty\!\left(\R^N\right)$ be a non-negative initial datum with $\mathrm{supp}(u_0)\subset B_R $ for some $R\ge 1$. Let $u$ be the corresponding weak energy solution of problem \eqref{wpme-funid} (with $T=+\infty$), in the sense of Definition \ref{def2} and Proposition \ref{pro1}, and fix any $ t>0 $. Then for all $\varepsilon\in(0,1)$ and a.e.\ $x_0\in B_\varepsilon$, it holds
		\begin{equation}\label{weighted AC all time est}
			\int_{B_R} u_0 \,\rho \, dx \leq C\left\{ t^{-\frac{1}{m-1}} \, R^{N-\gamma+\frac{2-\gamma}{m-1}}+t^{\frac{N-\gamma}{2-\gamma}} \, [u(x_0,t)]^{1+\frac{N-\gamma}{2-\gamma}(m-1)}  \right\} ,
		\end{equation}
		for a constant $C>0$ depending only on $N,m,\gamma,\underline{C},\overline{C}$.
		
	\end{pro}
		\begin{proof}
		Our starting point is the fundamental inequality \eqref{flb equation}, rewritten here for convenience:
		\begin{equation}\label{fundamental lower bound}
			\underbrace{\int_{\R^N}u(x,t_0)\,G(x,x_0)\,\rho(x)\,dx}_{I}-\underbrace{\int_{\R^N}u(x,t)\,G(x,x_0)\,\rho(x)\,dx}_{II}\leq(m-1) \, \frac{t^{\frac{m}{m-1}}}{t_0^{\frac{1}{m-1}}} \, u^m(x_0,t)
		\end{equation}
		for all $0<t_0<t$ and a.e.\ $x_0\in\R^N$. During the course of the proof, we will select $t_0$ to be a suitable fraction of $t$ (so that the feasible $x_0$ will only depend on $t$ in the end).
		
		We estimate $I$ from below and $II$ from above, from now on assuming that $|x_0|<\varepsilon<1$. In the following calculations, let $C$ denote a positive constant depending only on $N,m,\gamma,\underline{C},\overline{C}$, which may change from line to line. Also, let us denote the mass of the solution
        $$
        M =\int_{\R^N}u_0\,\rho\,dx=\int_{B_R}u_0\,\rho\,dx \, ,
        $$
        which by \eqref{mass cons} is conserved for all $t\geq0$. We recall that $R\ge 1$ is fixed once and for all according to the statement, and for now we consider a larger radius $\overline{R}\geq R$ that is initially fixed but will ultimately be allowed to vary in $[R,+\infty)$.
		
		As a first step to bounding $I$ from below, the basic estimate follows from the very definition of the Green's function \eqref{green-form}:
		\begin{equation}\label{basic I estimate}
			\begin{aligned}
				\int_{\R^N}u(x,t_0)\,G(x,x_0)\,\rho(x)\,dx&\geq
				\int_{B_{2\overline{R}}}u(x,t_0)\,G(x,x_0)\,\rho(x)\,dx\\ &\geq\frac{C_N}{\left(2\overline{R}+|x_0|\right)^{N-2}}\int_{B_{2\overline{R}}}u(x,t_0)\,\rho(x)\,dx \, ,
			\end{aligned}
		\end{equation}
        where $C_N$ stands for the multiplying constant in \eqref{green-form}. In the following passage, we will choose $t_0$ to be small enough in a quantitative way so that the integral on the right-hand side of \eqref{basic I estimate} is proportional to the mass $M$. Recalling the definition of $\phi_{\overline{R}}$ in \eqref{defcutoff1}, there holds
		\begin{equation}\label{cutoff implies}
			\int_{B_{2\overline{R}}}u(x,t_0)\,\rho(x)\,dx\geq\int_{\R^N}u(x,t_0)\,\phi_{\overline{R}}(x)\,\rho(x)\,dx \, .
		\end{equation}
		In order to continue estimating the right-hand side of \eqref{cutoff implies} from below, we proceed as follows:
		\begin{align*}
			\left|\frac{d}{dt}\int_{\R^N}u(x,t)\,\phi_{\overline{R}}(x)\,\rho(x)\,dx\right|& =_{(a)}\left|\int_{\R^N}u_t(x,t)\,\phi_{\overline{R}}(x)\, \rho(x) \,dx\right|=\left|\int_{\R^N}\Delta [u^m(x,t)]\,\phi_{\overline{R}}(x)\,dx\right|\\
			&=\left|\int_{\R^N}u^m(x,t)\,\Delta\phi_{\overline{R}}(x)\,dx\right|\\
			&\leq\norm{u(t)}_{L^\infty\left(\R^N\right)}^{m-1}\int_{\R^N}u(x,t) \left|\Delta\phi_{\overline{R}}(x)\right|dx\\
			&\leq_{(b)}C\,t^{-\lambda(m-1)} \, M^{\theta\lambda(m-1)}\int_{\R^N}u(x,t)\left|\Delta\phi_{\overline{R}}(x) \right| dx\\
			&\leq_{(c)}C\,t^{-\lambda(m-1)}\, M^{\theta\lambda(m-1)}\,\frac{\left(1+2\overline{R}\right)^{\gamma}}{\overline{R}^2}\int_{\R^N}u(x,t)\,\rho(x)\,dx\\
			&\leq_{(d)}C\,t^{-\lambda(m-1)} \, M^{1+\theta\lambda(m-1)}\,{\overline{R}}^{-2+\gamma} \, .
		\end{align*}
		Specifically, in $(a)$ we have used property \eqref{strong energy}, so this computation is justified for a.e.\ $t>0$, in $(b)$ we have used the smoothing estimate \eqref{basic smoothing}, in $(c)$ we have exploited \eqref{defcutoff1} and the lower bound in \eqref{weight-cond} on the weight $\rho$ whereas in $(d)$ we have taken advantage of the mass-conservation property \eqref{mass cons}, along with the elementary observation that $(1+2z)^\gamma\leq 3^\gamma z^\gamma$ for all $z\geq1$. Due to the absolute value in the above inequality, we may use in particular the negatively-signed version
		\begin{equation}\label{I smoothing}
			-\frac{d}{dt}\int_{\R^N}u(x,t)\,\phi_{\overline{R}}(x)\,\rho(x)\,dx\leq C\,t^{-\lambda(m-1)} \, M^{1+\theta\lambda(m-1)} \,{\overline{R}}^{\gamma-2} \,\normalcolor ;
		\end{equation}
		integrating \eqref{I smoothing} from $0$ to $t_0$, we find that
		\begin{equation}\label{spec-const}
			-\int_{\R^N}u(x,t_0)\,\phi_{\overline{R}}(x)\,\rho(x)\,dx+\int_{\R^N} u_0 \,\phi_{\overline{R}}\,\rho \,dx\leq \widetilde{C}\,t_0^{\theta\lambda} \, M^{1+\theta\lambda(m-1)} \, {\overline{R}}^{\gamma-2}\normalcolor
		\end{equation}
		for all $t_0>0$, where $\widetilde{C}>0$ is a constant we will keep track of and still depends only on $N,m,\gamma,\underline{C},\overline{C}$. Since the second term on the left-hand side of \eqref{spec-const} is equal to $M$, we arrive to
		\begin{equation}\label{preselect0}
			\int_{\R^N}u(x,t_0)\,\phi_{\overline{R}}(x)\,\rho(x)\,dx\geq M\left(1-\widetilde{C}\,t_0^{\theta\lambda} \, M^{\theta\lambda(m-1)} \, {\overline{R}}^{\gamma-2} \right) .
		\end{equation}
		 Next, we select a $t_0$ in terms of $M$ and $\overline{R}$ so that the quantity inside the parenthesis in \eqref{preselect0} is a positive constant. That is, we set
		\begin{equation}\label{rule for t_0}
			t_0 = c_0 \, \tfrac{\overline{R}^{(N-\gamma)(m-1)+2-\gamma}}{M^{m-1}} \qquad \text{for} \; \; c_0 = \left[2\widetilde{C}\right]^{-\frac{1}{\theta \lambda}} .
		\end{equation}
		Plugging the choice \eqref{rule for t_0} into \eqref{preselect0}, recalling \eqref{deflambdatheta}, we obtain
		\begin{equation}\label{pre-final-i-est}
			\int_{\R^N}u(x,t_0)\,\phi_{\overline{R}}(x)\,\rho(x)\,dx\geq \left(1-\widetilde{C}\,c_0^{\theta\lambda}\right)M  = \frac{M}{2}\, .
		\end{equation}
		Hence, we conclude from \eqref{basic I estimate}, \eqref{cutoff implies}, and \eqref{pre-final-i-est} that
		\begin{equation}\label{final I estimate}
			\begin{aligned}
				I=\int_{\R^N}u(x,t_0)\,G(x,x_0)\,\rho(x)\,dx\geq  \frac{C_N}{2}\frac{M}{\left(2\overline{R}+|x_0|\right)^{N-2}}
				&\geq \frac{C_N}{2}\frac{M}{\left(2\overline{R}+1\right)^{N-2}}\\
				& \geq \underbrace{\frac{C_N}{2 \, 3^{N-2}}}_{c_1} \, \cdot \, \frac{M}{{\overline{R}}^{N-2}} \, ,
			\end{aligned}
		\end{equation}
		for all $t_0$  as in \eqref{rule for t_0}, and recalling that $|x_0|<\varepsilon<1\le \overline{R}$.
		
		Now, we prepare to estimate $II$ from above by splitting the integral into two parts:
		\begin{equation*}
			II=\underbrace{\int_{B_r}u(x,t)\,G(x,x_0)\,\rho(x)\,dx}_{II_A}+\underbrace{\int_{B^c_r}u(x,t)\,G(x,x_0)\,\rho(x)\,dx}_{II_B} \, ,
		\end{equation*}
		where $r>0$ only serves to optimize the relative sizes of $II_A$ and $II_B$. In particular, we select it as
		\begin{equation}\label{choice of R1}
			r=c_2\, \overline{R}+\varepsilon \, ,
		\end{equation}
		for a suitable constant $c_2>0$ to be chosen later. In order to estimate $II_A$, we use H\"{o}lder's inequality and the upper bound in \eqref{weight-cond} to obtain
        \begin{equation}\label{1hiia}
			\begin{aligned}
				II_A\leq\norm{u(t)}_{L^\infty\left(\R^N\right)}\int_{B_{r}}G(x,x_0)\,\rho(x)\,dx &\leq \overline{C}\norm{u(t)}_{L^\infty\left(\R^N\right)}\int_{B_{r}}G(x,x_0)\,|x|^{-\gamma}\,dx\\
				&\leq C\,t^{-\lambda}\,M^{\theta\lambda}\int_{B_r}G(x,x_0)\,|x|^{-\gamma}\,dx \, ,
			\end{aligned}
		\end{equation}
		where in the last passage we have used the smoothing effect \eqref{basic smoothing}. Next, we claim that the function $f \colon B_\varepsilon\to\R$ defined by
		\begin{equation*}
			y \in B_\varepsilon \mapsto \int_{B_r}G(x,y)\,|x|^{-\gamma}\,dx
		\end{equation*}
		achieves its maximum at $y=0$. To this end, it is not difficult to see that
        $$
f(y) = f(0)-\tfrac{1}{(N-\gamma)(2-\gamma)} \, |y|^{2-\gamma} \qquad \forall y \in B_\varepsilon \,,
        $$
       since $f$ is a radial solution to $ -\Delta f (y)= |y|^{-\gamma} \, \chi_{B_r}(y) $ and $r>\varepsilon$, whence the equation can be solved by ODE methods and the claim easily follows.  Therefore, we may continue from \eqref{1hiia} the following chain of estimates:
		\begin{equation}\label{IIA}
			\begin{aligned}
				II_A \leq C\,t^{-\lambda} \,M^{\theta\lambda}\int_{B_r}|x|^{2-N-\gamma}\,dx & \leq C\,t^{-\lambda} \, M^{\theta\lambda} \, r^{2-\gamma}\\
				&= C\,t^{-\lambda} \, M^{\theta\lambda}\left(c_2 \, \overline{R}+\varepsilon \right)^{2-\gamma}\\
				&\leq C\left(1+c_2\right)^{2-\gamma}t^{-\lambda} \, M^{\theta\lambda} \, {\overline{R}}^{2-\gamma} \, ,
			\end{aligned}
		\end{equation}
        recalling that $ 0<\varepsilon<1\le \overline{R} $.
		On the other hand, our estimate for $II_B$ is a simple consequence of mass conservation and \eqref{choice of R1}:
		\begin{equation}\label{IIB}
			II_B\leq\frac{C_N}{(r-\varepsilon)^{N-2}}\int_{\R^N}u(x,t) \, \rho(x)\,dx=C\,c_2^{2-N} \, \frac{M}{\overline{R}^{N-2}} \, .
		\end{equation}
Combining \eqref{IIA} and \eqref{IIB} gives
		\begin{equation*}
			II\leq C\left(1+c_2\right)^{2-\gamma}t^{-\lambda} \, M^{\theta\lambda}\,{\overline{R}}^{2-\gamma}+C\,c_2^{2-N} \, \frac{M}{\overline{R}^{N-2}} \, .
		\end{equation*}
We consider, for the moment, all times $t$ of the form
		\begin{equation}\label{choice for t}
			t=c_3\,\tfrac{{\overline{R}}^{(N-\gamma)(m-1)+2-\gamma}}{M^{m-1}} \, ,
		\end{equation}
		under the crucial restriction $c_3>c_0$, so that $t>t_0$. At the end of the proof we will explain how to get rid of such an extra assumption. With this choice, we have (again recalling \eqref{deflambdatheta}--\eqref{spexprel})
		\begin{equation}\label{final II estimate}
			II\leq c_4\,\frac{M}{{\overline{R}}^{N-2}} \, ,
		\end{equation}
		where
		\begin{equation*}\label{defcn}
		c_4= C \, \max\left\{\left(1+c_2\right)^{2-\gamma}c_3^{-\lambda} \, , \, c_2^{2-N}\right\}.
		\end{equation*}

Let us insert the estimates \eqref{final I estimate} and \eqref{final II estimate} for $I$ and $II$ in \eqref{fundamental lower bound}, to reach		\begin{equation}\label{c1-c4}
			\left(c_1-c_4\right)\frac{M}{{\overline{R}}^{N-2}}\leq(m-1) \, \frac{t^{\frac{m}{m-1}}}{t_0^{\frac{1}{m-1}}} \, u^m(x_0,t) \, ,
		\end{equation}
provided, in addition, $x_0$ belongs to a suitable full-measure set in $B_\varepsilon$ depending on $t_0,t$. Now, we would like to make a suitable choice of $c_0$, $c_2$, and $c_3$ so that
		\begin{equation*}\label{cond-consts}
			\tfrac{c_1}{2}>c_4\qquad\text{and}\qquad c_3>c_0\,.
		\end{equation*}
		This is achieved by taking $c_2$ and $c_3$ sufficiently large, and $c_0$ sufficiently small; quantitatively, that is
		\begin{equation}\label{cond-consts-1}
			c_2 > (2C)^{\frac{1}{N-2}} \, c_1^{-\frac{1}{N-2}} \qquad \text{and} \qquad c_3 > \max\left\{ (2C)^{\frac{1}{\lambda}} \, c_1^{-\frac{1}{\lambda}} \left(1+c_2\right)^{\frac{2-\gamma}{\lambda}} , \, c_0 \right\} .
		\end{equation}
		In conclusion, owing to \eqref{c1-c4}, under \eqref{rule for t_0}, \eqref{final I estimate}, \eqref{choice for t}, and \eqref{cond-consts-1} we end up with
\begin{equation}\label{c1-c4-bis}
			\frac{M}{{\overline{R}}^{N-2}}\leq 2\, (m-1)  \,  c_1^{-1}  \left(\frac{c_3}{c_0}\right)^{\frac{1}{m-1}}\,t\,u^m(x_0,t) \, .
		\end{equation}
Replacing $\overline{R}$ by $t$ and $M$ according to \eqref{choice for t}, estimate \eqref{c1-c4-bis} becomes
		\begin{equation*}
M^{m\theta\lambda}\leq 2\,(m-1)\,c_0^{-\frac{1}{m-1}}\,c_3^{\frac{m}{m-1}\theta\lambda} \, t^{m\lambda}\,u^m(x_0,t) \, .
		\end{equation*}
Since $c_0$ and $c_1$ are numerical constants, also $c_2$ and $c_3$ can be chosen as such satisfying \eqref{cond-consts-1}, therefore taking the correct power of both sides and using \eqref{deflambdatheta} we conclude
		\begin{equation}\label{final ineq}
			M\leq C\,t^{\frac{N-\gamma}{2-\gamma}}\,[u(x_0,t)]^{\frac{(N-\gamma)(m-1)+2-\gamma}{2-\gamma}} \, ,
		\end{equation}
where, as usual, we have summarized all of the constants depending on $N,m,\gamma,\underline{C},\overline{C}$ by $C>0$. Recalling that $\overline{R}\geq R$ is arbitrary, \eqref{choice for t}, \eqref{cond-consts-1}, and \eqref{final ineq} imply
		\begin{equation}\label{weighted AC lemma est}
			\int_{B_R}u_0\,\rho\,dx\leq C_1\,t^{\frac{N-\gamma}{2-\gamma}}\,[u(x_0,t)]^{1+\frac{N-\gamma}{2-\gamma}(m-1)}
		\end{equation}
		for all
		\begin{equation}\label{weighted AC lemma t}
			t\geq C_2 \, \frac{R^{(N-\gamma)(m-1)+2-\gamma}}{\left(\int_{B_R}u_0\,\rho\,dx\right)^{m-1}}\,,
		\end{equation}
		for some constants $C_1,C_2>0$, with the usual dependencies, to which we give names to avoid confusion.

We are left with establishing \eqref{weighted AC all time est} also for times that do not fulfill \eqref{weighted AC lemma t}. This follows easily by a dichotomy argument. Indeed, if
		\begin{equation}\label{stimes-1}
t < C_2 \, \frac{R^{(N-\gamma)(m-1)+2-\gamma}}{\left(\int_{B_R}u_0\,\rho\,dx\right)^{m-1}} \, ,
\end{equation}
then simply by rearranging terms
		\begin{equation}\label{stimes-2}
			\int_{B_R} u_0 \,\rho \,dx\leq C_2^{\frac{1}{m-1}}\,t^{-\frac{1}{m-1}} \, R^{N-\gamma+\frac{2-\gamma}{m-1}} \, .
		\end{equation}
        Hence, because \eqref{weighted AC lemma est} holds under \eqref{weighted AC lemma t} and \eqref{stimes-2} holds under \eqref{stimes-1}, it is plain that \eqref{weighted AC all time est} is satisfied for all $t>0$.
	\end{proof}
\subsection{From constructed to general solutions}\label{gen-sol}
	Now we need to prove the Aronson--Caffarelli estimate of the form \eqref{weighted AC all time est} for given solutions, not just the constructed  ones dealt with in Subsection \ref{IRes}. To this end, we introduce an approximation scheme, in which the following non-homogeneous Cauchy--Dirichlet problem plays a key role:
	\begin{equation}\label{cd problem}
		\begin{cases}
			\rho \, u_t = \Delta\!\left(u^m \right) & \text{in} \; \; \Omega\times(\tau_1,\tau_2 ) \, , \\
			u^m = g & \text{on} \; \; \partial\Omega\times (\tau_1,\tau_2) \, , \\
			u  = u_0 & \text{on} \; \; \Omega\times \{ \tau_1 \} \, ,
		\end{cases}
	\end{equation}
	for $\Omega$ a bounded smooth domain, $0 \le \tau_1<\tau_2$, and  measurable functions $u_0$ and $g$.
	
	\begin{den}[Very weak solutions to the Cauchy--
Dirichlet problem]\label{vw cd def}
		Let $ \Omega \subset \R^N $ be a bounded smooth domain, $ 0 \le \tau_1<\tau_2 $,  $u_0\in L^1_\rho(\Omega)$, and $g\in L^1_\sigma(\partial\Omega\times(\tau_1,\tau_2))$. We say that a function $ u \in L^m(\Omega\times(\tau_1,\tau_2))\cap L^1_\rho(\Omega\times(\tau_1,\tau_2))$ is a very weak solution of problem \eqref{cd problem} if
		\begin{equation}\label{vw cd eq}
			\int_{\tau_1}^{\tau_2} \int_{\Omega} \left(u \, \psi_t \, \rho + u^m \, \Delta\psi \right) dx dt +\int_\Omega u_0(x)\,\psi(x,\tau_1)\,\rho(x)\,dx= \int_{\tau_1}^{\tau_2} \int_{\partial\Omega} g \, \partial_{\Vec{\mathsf n}}\psi \, d\sigma dt
		\end{equation}
		holds for all test function $\psi\in C^{2,1}\!\left(\overline{\Omega}\times[\tau_1,\tau_2]\right)$ that vanishes on $\left(\partial\Omega\times(\tau_1,\tau_2) \right) \cup \left( \Omega\times\{\tau_2\} \right)$, where $\Vec{\mathsf n}$ is the outer unit normal vector to $\Omega$. If \eqref{vw cd eq} holds with $\leq$ (resp.~$\geq$) instead of $=$ for all non-negative $\psi$ as above, then we say that $u$ is a very weak supersolution (resp.~subsolution) of \eqref{cd problem}.
	\end{den}
	The next lemma establishes that a general solution to \eqref{wpme-noid} always satisfies \eqref{cd problem}, up to choosing proper balls and times.
	\begin{lem}\label{globloc}
		Let $u$ be a very weak solution of equation \eqref{wpme-noid}, in the sense of Definition \ref{def:vw-gl}. Then, for all $\tau_2\in(0,T)$, a.e.~$r>0$, and a.e.~$\tau_1 \in(0,T)$ with $ \tau_1 < \tau_2 $, we have that $u |_{B_r \times(\tau_1,\tau_2)}$ is a very weak solution to \eqref{cd problem}, in the sense of Definition \ref{vw cd def}, with domain $ \Omega=B_r $, initial datum $u_0=u(\tau_1)$, and boundary datum  $g=u^m |_{\partial B_r}$.
	\end{lem}
	\begin{proof}
		It suffices to apply the strategy of proof of \cite[Proposition 3.1]{GMP2}. Since we are working in a Euclidean rather than a Riemannian framework, the arguments therein can even be slightly simplified. The only additional factor is the presence of the weight $\rho$, which however does not play a major role here. We provide anyway a concise proof for the reader's convenience.
		
		Fix any $0<\tau_1<\tau_2 < T$. Since $u\in  L^m_{\mathrm{loc}}\!\left(\R^N\times(0,T)\right)$, we have that for a.e.\ $r>0$ (independently of $\tau_1, \tau_2$) the trace  $u^m |_{\partial B_r }$ is well defined in $ L^1_\sigma\!\left(\partial B_r\times(\tau_1,\tau_2)\right)$ as a Lebesgue value, in the sense that
		\begin{equation}\label{Lebesgue}
			\lim_{\varepsilon \to 0}    \frac{\int_{r-\varepsilon}^{r+\varepsilon} \int_{\mathbb{S}^{N-1}}  \int_{\tau_1}^{\tau_2} \left|u^m(\varrho,\theta,t)  - u^m |_{\partial B_r}(\theta,t) \right| dt d\theta \, \varrho^{N-1} d\varrho }{\varepsilon} = 0 \, ,
		\end{equation}
		where $ d \theta $ stands for the volume measure on the unit sphere $ \mathbb{S}^{N-1} $. Similarly, fix any $r>0$. Since $u\in L^1_{\rho,\mathrm{loc}}\!\left(\R^N\times(0,T)\right)$, for a.e.\ $\tau_1\in (0,T)$ (independently of $r$) we have that $u(\tau_1)\in L^1_\rho(B_r)$ and
		\begin{equation}\label{Lebesgue2}
			\lim_{\varepsilon \to 0}  \frac{\int_{\tau_1-\varepsilon}^{\tau_1+\varepsilon}  \int_{B_r} \left|u(x,t)  - u(x,\tau_1) \right| \rho(x) \, dx dt  }{\varepsilon} = 0 \, .
		\end{equation}	
		From now on, we will assume that $ r>0 $ and $ \tau_1 \in (0,\tau_2) $ are such that both \eqref{Lebesgue} and \eqref{Lebesgue2} hold. In order to prove  \eqref{vw cd eq}, the first step is to use in \eqref{vw-cond} the test function $\psi_n=\psi \, \varphi_n$, where $\psi$ is a test function as in Definition \ref{vw cd def} and $\varphi_n \in C^\infty_c\!\left( \R^N \times (0,T) \right)$, with $0\leq\varphi_n\leq1$, is a family of parabolic radial cutoff functions constructed in such a way that
		\begin{equation}\label{bdd-cutoff-cond}
			\begin{aligned}
				\varphi_n & = 0 \qquad \text{in}\;\; \left( B^c_{r-\frac{1}{n}}\times(\tau_1,\tau_2)\right) \cup \left( \R^N \times \left(0,\tau_1+\tfrac{1}{n} \right) \right) \cup \left( \R^N \times \left(\tau_2-\tfrac{1}{n} , T \right) \right)  , \\
				\varphi_n & = 1 \qquad \text{in}\;\; B_{r-\frac{2}{n}} \times \left( \tau_1+\tfrac{2}{n},\tau_2-\tfrac2n \right) ,
			\end{aligned}
		\end{equation}
		and satisfies
		\begin{equation}\label{bdd-cutoff-cond-bis}
			|\nabla\varphi_n|+\tfrac{1}{n}\, |\Delta\varphi_n|\leq c \, n \, \chi_{ \Big( B_{r-\frac 1 n} \setminus B_{r-{\frac 2 n}} \Big) \times (\tau_1,\tau_2) } \, , \quad |\partial_t\varphi_n| \le c \, n \, \chi_{ B_r \times \left( \left(\tau_1+\frac{1}{n},\tau_1+\frac{2}{n}\right) \cup \left(\tau_2-\frac{2}{n},\tau_2-\frac{1}{n}\right)  \right)} \, ,
		\end{equation}
		for all $n \in \mathbb{N}$ sufficiently large and some $c>0$ independent of $n$. Note that, by construction, $\psi\,\varphi_n$ smoothly vanishes on the whole boundary of $ B_r\times(\tau_1,\tau_2)$, so we tacitly consider its natural extension to all of $\R^N\times(0,T)$ as a test function in \eqref{vw-cond} (this is feasible up to the plain fact that \eqref{vw-cond} also holds for all test functions in $ C_c^{2,1}\!\left( \R^N \times (0,T) \right) $).
		
		For the time-derivative term, we obtain
		\begin{equation*}\label{td-1}
			\begin{aligned}
				& \int_{0}^{T} \int_{\mathbb{R}^N} u \, \partial_t\psi_n \, \rho\,dxdt=\int_{0}^{T} \int_{\mathbb{R}^N} u \, \psi_t\,\varphi_n \, \rho\,dxdt+\int_{0}^{T} \int_{\mathbb{R}^N} u \, \psi\,\partial_t\varphi_n \, \rho\,dxdt\\
				= & \underbrace{\int_{\tau_1}^{\tau_2} \int_{B_r} u \, \psi_t\,\varphi_n \, \rho\,dxdt}_{I_a}+\underbrace{\int_{\tau_1+\frac{1}{n}}^{\tau_1+\frac{2}{n}} \int_{B_{r}} u \, \psi\,\partial_t\varphi_n \, \rho\,dxdt}_{I_b}+\underbrace{\int_{\tau_2-\frac{2}{n}}^{\tau_2-\frac{1}{n}} \int_{B_{r}} u \, \psi\,\partial_t\varphi_n \, \rho\,dxdt}_{I_c} \, .
			\end{aligned}
		\end{equation*}
		Clearly, $I_a$ converges as $n\to\infty$ to the first term on the left-hand side of \eqref{vw cd eq}. As concerns the term $I_b$, taking advantage of \eqref{Lebesgue2}, \eqref{bdd-cutoff-cond}, and \eqref{bdd-cutoff-cond-bis} we notice that
		\begin{equation*}\label{t-conv}
        \begin{aligned}
			& \lim_{n\to \infty} \int_{\tau_1+\frac{1}{n}}^{\tau_1+\frac{2}{n}} \int_{B_{r}} u \, \psi\,\partial_t\varphi_n \, \rho\,dxdt = \lim_{n\to \infty} \int_{\tau_1+\frac{1}{n}}^{\tau_1+\frac{2}{n}} \int_{B_{r}} u(\tau_1) \, \psi\,\partial_t\varphi_n \, \rho\,dxdt \\
            = &  \lim_{n\to \infty} \int_{B_{r}} u(\tau_1) \, \psi\!\left(\tau_1+\tfrac 2 n\right) \varphi_n\!\left(\tau_1+\tfrac 2 n\right)  \rho\,dx - \lim_{n\to \infty} \int_{\tau_1+\frac{1}{n}}^{\tau_1+\frac{2}{n}} \int_{B_{r}} u(\tau_1) \, \psi_t \, \varphi_n \, \rho\,dxdt  \\
            = & \int_{B_{r}} u(\tau_1) \, \psi\!\left(\tau_1\right)  \rho\,dx \, .
            \end{aligned}
		\end{equation*}
As for $I_c$, it easy to check that it tends to zero as $n\to\infty$ thanks to the fact that $ \{ \psi \, \partial_t \varphi_n \} $ stays uniformly bounded in $ B_r \times \left(\tau_2-{2}/{n},\tau_2-{1}/{n}\right) $ (recall that $ \psi(\tau_2) = 0$).
		
		The spatial-derivative terms are handled in an analogous way, observing that
		\begin{equation*}\label{xx-conv}
			\int_{0}^{T} \int_{\mathbb{R}^N}  u^m \, \Delta \psi_n  \, dx dt = \int_{\tau_1}^{\tau_2} \int_{B_r}  u^m \left(\Delta\psi\,\varphi_n+2\left\langle \nabla\psi  , \nabla\varphi_n \right\rangle +\psi\,\Delta\varphi_n\right) dx dt \, .
		\end{equation*}
		Again, it is trivial that
		\begin{equation*}	\label{x-conv}
\lim_{n \to \infty}	\int_{\tau_1}^{\tau_2} \int_{B_r}  u^m \, \Delta\psi \, \varphi_n\,dxdt = \int_{\tau_1}^{\tau_2} \int_{B_r}  u^m \,\Delta\psi\,dxdt \, .
		\end{equation*}
		The remaining two terms, summed together, converge exactly to the negative right-hand side of \eqref{vw cd eq}. The corresponding computations are similar to the just stated time-derivative term, and follow closely the argument of the proof of \cite[Proposition 3.1]{GMP2}, taking advantage of \eqref{Lebesgue}; in fact, since $ \partial_\varrho = \partial_{\Vec{\mathsf n}} $, one can show that
        $$
\lim_{n \to \infty}2 \int_{\tau_1}^{\tau_2} \int_{B_r}  u^m  \left\langle \nabla\psi , \nabla\varphi_n \right\rangle dx dt   =  -2 \, r^{N-1} \int_{\tau_1}^{\tau_2} \int_{\mathbb{S}^{N-1}} u^m |_{\partial B_r} \, \partial_{\Vec{\mathsf n}}\psi \, d\theta dt
        $$
        and
        $$
       \lim_{n \to \infty}  \int_{\tau_1}^{\tau_2} \int_{B_r}  u^m  \, \psi\,\Delta\varphi_n \, dx dt  = r^{N-1} \int_{\tau_1}^{\tau_2} \int_{\mathbb{S}^{N-1}} u^m |_{\partial B_r} \, \partial_{\Vec{\mathsf n}}\psi \, d\theta dt \, ,
        $$
		whence
		\begin{equation*}\label{s-conv}
			\lim_{n\to \infty}\int_{0}^{T} \int_{\mathbb{R}^N}  u^m \, \Delta \psi_n   \, dx dt = \int_{\tau_1}^{\tau_2} \int_{B_r}  u^m \, \Delta\psi \, dxdt -\int_{\tau_1}^{\tau_2}\int_{\partial B_r} u^m |_{\partial B_r } \, \partial_{\Vec{\mathsf n}}\psi\,d\sigma dt \, .
		\end{equation*}
Combining all of the above identities and passages to the limit concludes the proof.
	\end{proof}

    \begin{rem}\label{loc-vw-bdry}\rm
        From the above proof, it is plain that the result continues to hold if $ u $ is a \emph{local} very weak solution to  \eqref{wpme-noid-loc} (relabeling $ \tau_1 $ and $\tau_2$ to avoid confusion), in the sense of Definition \ref{vw loc def}, and in the place of arbitrary balls $ B_r $ one takes $ B_r(x_0) $, for all $ x_0 \subset \Omega $ and a.e.\ $ r>0 $ such that $ B_r(x_0) \Subset \Omega $.
    \end{rem}

 	We now prove a fundamental local comparison property, by adapting duality arguments that are by now well established. In order to avoid an overly technical proof, we carry out a simplified argument that works under the additional assumption that $ \gamma <N/2 $, which becomes a real restriction only if $N=3$. For similar results and tools, see \cite[Theorem 6.5]{Vazquez} for the unweighted equation and \cite[Proposition 4.1]{MP} for the weighted equation in the whole $\R^N$; in particular, in the latter, all the details on how to get rid of the constraint $\gamma < N/2$ are provided (this involves removing a small ball $ B_\delta $ from $\Omega$).

	\begin{pro}[Local comparison for very weak sub/supersolutions]\label{local comparison lemma}
Let $ \Omega \subset \R^N $ be a bounded smooth domain, $ 0 \le \tau_1<\tau_2 $,  $u_0\in L^\infty(\Omega)$, and $g \in L^\infty_\sigma(\partial\Omega\times(\tau_1,\tau_2))$. Let $u \in L^\infty(\Omega \times (\tau_1,\tau_2))$ and $v \in L^\infty(\Omega \times (\tau_1,\tau_2))$ be a very weak subsolution and a very weak supersolution to \eqref{cd problem}, respectively, in the sense of Definition \ref{vw cd def}. Then $u\leq v$ in $\Omega\times(\tau_1,\tau_2)$.
	\end{pro}
	\begin{proof}
	It suffices to prove that $u(\tau)\leq v(\tau)$ for a.e.~$\tau\in(\tau_1,\tau_2)$. To begin, we notice that $u-v$ satisfies the following inequality for a.e.~$\tau\in(\tau_1,\tau_2)$ and all non-negative $\psi \in C^{2,1}\!\left(\overline{\Omega}\times[\tau_1,\tau]\right)$ that vanishes on $\partial\Omega\times(\tau_1,\tau) $:
		\begin{equation}\label{subsuperineq}
			\int_{\Omega}(u(\tau)-v(\tau))\,\psi(\tau)\,\rho\,dx\leq \int_{\tau_1}^{\tau} \int_{\Omega} \left[ (u-v) \, \psi_t \, \rho + \left(u^m-v^m\right) \Delta\psi \right] dx dt \, ,
		\end{equation}
which is simply obtained by subtracting the very weak formulations satisfied by $u$ and $v$, respectively, and using a similar time cutoff argument to the proof of Lemma \ref{globloc}.

		Next, defining the function
		\begin{equation*}\label{defa}
			a=
			\begin{cases}
				\frac{u^m-v^m}{u-v} \quad&\text{if}\;\;u\neq v\, ,\\
				0 &\text{if}\;\;u=v\,,
			\end{cases}
		\end{equation*}
		which is readily seen to be bounded, we can rewrite \eqref{subsuperineq} as
		\begin{equation*}
			\int_{\Omega}(u(\tau)-v(\tau))\,\psi(\tau)\,\rho\,dx\leq \int_{\tau_1}^{\tau} \int_{\Omega} (u-v) \, ( \psi_t \, \rho + a \, \Delta\psi) \, dx dt \, .
		\end{equation*}
		Choosing an arbitrary non-negative function $\omega\in C^\infty_c(\Omega)$, we would ideally like to use as $\psi$ the solution of the dual (backward) problem
		\begin{equation}\label{BP}
			\begin{cases}
				-\rho\,\psi_t = a\,\Delta \psi&\text{in}\;\;\Omega\times(\tau_1,\tau)\,,\\
				\psi = 0 &\text{on}\;\;\partial\Omega\times(\tau_1,\tau)\,,   \\
				\psi=\omega&\text{on}\;\;\Omega\times\{\tau\}\,,
			\end{cases}
		\end{equation}
		which would immediately yield
		\begin{equation*}
			\int_{\Omega}(u(\tau)-v(\tau))\,\omega\,\rho\,dx\leq 0\,,
		\end{equation*}
		and thus the proof would be complete owing to the arbitrariness of $\tau$ and $\omega$. However, \emph{a priori} such a $\psi$ fails to be in the class $C^{2,1}\!\left(\overline{\Omega}\times[\tau_1,\tau]\right)$, so it is necessary to smoothly approximate both $\rho$ and $a$ (positively lifting the latter).
		
		To this end, we define a sequence of smooth and positive weights $\{\rho_n\} \subset C^\infty\!\left(\R^N\right)$ satisfying
		\begin{equation}\label{rhon-rho}
			\rho_n \underset{n \to \infty}{\longrightarrow} \rho \qquad \text{in} \;\; L^2_{\mathrm{loc}}\!\left( \R^N \right) ,
            \end{equation}
		which can be constructed \emph{e.g.}\ by standard mollification methods, under the additional constraint $ \gamma < N/2 $. Likewise, we consider a sequence of smooth and positive functions $ \{a_n\} \subset C^\infty\!\left(\overline{\Omega} \times [\tau_1,\tau_2] \right)$, which will be made more precise in the final step of the proof. Now, for each $n \in \mathbb{N}$, we pick as a test function in \eqref{subsuperineq} the solution $\psi_n$ of the following approximate version of \eqref{BP}:
\begin{equation}\label{approx-prob}
			\begin{cases}
				\rho_n\,\partial_t\psi_n+a_n\,\Delta \psi_n=0&\text{in}\;\;\Omega\times(\tau_1,\tau)\,,\\
				\psi_n = 0 &\text{on}\;\;\partial\Omega \times(\tau_1,\tau)\,,   \\
				\psi_n=\omega &\text{on}\;\;\Omega\times\{\tau\}\,.
			\end{cases}
		\end{equation}
It is standard that $ \psi_n $ is smooth up to the boundary and positive in $ \Omega \times (\tau_1,\tau) $. Then we have, labeling $\chi=u-v$,
\begin{equation}\label{subsuperineq2}
		\int_{\Omega} \chi(\tau)\,\omega\,\rho\,dx\leq \int_{\tau_1}^{\tau} \int_{\Omega} \chi \left[\partial_t\psi_n \, (\rho-\rho_n) + (a-a_n) \, \Delta\psi_n \right] dx dt \,.
		\end{equation}
The proof is completed if we can show that the right-hand side of \eqref{subsuperineq2} vanishes upon taking $n\to\infty$. To this end, the key technical tool is the following energy equality:
		\begin{equation*}\label{ener-eq}
			\frac{1}{2}\int_{\Omega} \left|\nabla\psi_n(\tau_1)\right|^2 dx+\int_{\tau_1}^{\tau}\int_{\Omega} \left|\Delta\psi_n\right|^2 \frac{a_n}{\rho_n}\,dxdt=\frac{1}{2}\int_{\Omega} \left|\nabla\omega\right|^2 dx\,,
		\end{equation*}
		which follows from dividing the differential equation in \eqref{approx-prob} by $\rho_n$, multiplying by $\Delta\psi_n$, and integrating by parts. Let us estimate the first summand on the right-hand side of \eqref{subsuperineq2}:
		\begin{equation*}
			\begin{aligned}
				& \left| \int_{\tau_1}^{\tau} \int_{\Omega} \chi\,\partial_t  \psi_n\,(\rho-\rho_n)\,dxdt \right|\\
				\leq & \norm{\chi}_{L^\infty(\Omega\times(\tau_1,\tau_2))}\left(\int_{\tau_1}^{\tau} \int_{\Omega}(\rho-\rho_n)^2 \, \frac{a_n}{\rho_n} \,dxdt\right)^{\frac{1}{2}}\left(\int_{\tau_1}^{\tau} \int_{\Omega} \left|\Delta\psi_n\right|^2\frac{a_n}{\rho_n} \, dxdt\right)^{\frac{1}{2}} ,
			\end{aligned}
		\end{equation*}
		where we have used H\"{o}lder's inequality and, again, the differential equation in \eqref{approx-prob}. Clearly, it is possible to construct $  \{ \rho_n \}$ and $\{a_n\}$ in such a way that the ratio $\frac{a_n}{\rho_n}$ is bounded above independently of $n$, so we find that
		\begin{equation*}
\left| \int_{\tau_1}^{\tau} \int_{\Omega} \chi\,\partial_t  \psi_n\,(\rho-\rho_n)\,dxdt \right|\leq C \left(\int_{\tau_1}^{\tau} \int_{\Omega}(\rho-\rho_n)^2 \,dxdt\right)^{\frac{1}{2}} ,
		\end{equation*}
		for some constant $C>0$ independent of $n$. The second summand can be treated in a similar way to obtain an analogous bound, leading to
		\begin{equation}\label{minrho}
	\int_{\Omega} \chi(\tau)\,\omega\,\rho\,dx \leq C \left[ \left(\int_{\tau_1}^{\tau} \int_{\Omega}(\rho-\rho_n)^2 \, dxdt\right)^{\frac{1}{2}} + \left(\int_{\tau_1}^{\tau} \int_{\Omega}\frac{(a-a_n)^2}{a_n} \, \rho_n \, dxdt\right)^{\frac{1}{2}} \right] .
		\end{equation}
    It is possible to take $\rho_n$ in such a way that, in addition to \eqref{rhon-rho}, it satisfies $ \rho_n \le C |x|^{-\gamma} $, and by reasoning as in the final passages of the proof of \cite[Proposition 4.1]{MP} it is also possible to construct $a_n$ so that $ (a_n-a)/a_n \to 0 $ in $ L^2_{\gamma}(\Omega \times (\tau_1,\tau_2)) $. Letting $ n \to \infty $ in \eqref{minrho} we end up with
		\begin{equation*}
			\int_{\Omega} \chi(\tau)\,\omega\,\rho\,dx \le 0 \, ,
            \end{equation*}
            which completes the proof since $\omega$ and $\tau$ (a.e.\ in $ (\tau_1,\tau_2) $) are arbitrary.
	\end{proof}

	
        \begin{cor}[Global comparison for non-negative very weak solutions]\label{global comparison lemma}
Let $u$ be a non-negative very weak solution of equation \eqref{wpme-noid}, in the sense of Definition \ref{def:vw-gl}, such that $ u \in L^\infty_{\mathrm{loc}}\!\left( \R^N\times (0,T) \right) $. Take any $ \tau_1 \in (0,T) $ as in Lemma \ref{globloc}. Let $ v $ be the weak energy solution of problem \eqref{wpme-funid}, in the sense of Definition \ref{def2}, with initial datum $v_0\in L^1_\rho\!\left(\mathbb{R}^N\right) \cap L^\infty\!\left(\R^N\right)$. If $0\leq v_0\leq u(\tau_1)$, then $v \leq u(t+\tau_1)$ in $ \R^N \times (0,T-\tau_1) $.
	\end{cor}
	\begin{proof}
		We recall that the standard procedure to construct the solution $v$ in Proposition \ref{pro1} is as a (pointwise monotone) limit of solutions $ \{v_R\}$, as $R \to +\infty$, to Cauchy--Dirichlet problems \eqref{cd problem} with $ \Omega = B_R $ (in the whole time half line $ (0,+\infty) $), initial datum $v_0 \,\chi_{B_R}$, and homogeneous boundary datum $g=0$. For such a construction, see for example \cite[Theorem 9.3]{Vazquez} or \cite{RV1} and the articles cited therein. On the other hand, thanks to Lemma \ref{globloc}, we know that $u|_{B_R\times(\tau_1,\tau_2)}$ is very weak solution of \eqref{cd problem} in the sense of Definition \ref{vw cd def}, for all $ \tau_2 \in (\tau_1,T) $. Furthermore, by the non-negativity of $u$ (hence of $ u^m|_{\partial B_R} $) and the assumption that $v_0\leq u_0$ (along with the local boundedness of $u$), we may apply Proposition \ref{local comparison lemma} to deduce that
		\begin{equation*}
			v_R \leq u(t+\tau_1) \qquad\text{in}\; \; B_R\times(0,\tau_2-\tau_1)\, .
		\end{equation*}
        Passing to the limit as $R\to+\infty$ and recalling that $\tau_2$ is arbitrary yields the thesis.
	\end{proof}

    We are now ready to characterize the initial trace of non-negative very weak solutions of equation \eqref{wpme-noid}.

	\begin{proof}[Proof of Theorem \ref{thm:weighted-AC}] \, \\[0.1cm]
    \noindent\textbf{Step 1: Existence {of the initial trace}.}
		By Theorem \ref{apriori-bounded}, it holds that $$u\in L^\infty_{\mathrm{loc}}\!\left(\R^N\times(0,T)\right),$$
        so the set
		\begin{equation*}
			A=\left\{\tau\in(0,T) \colon \ u(\tau)\in L^\infty_{\mathrm{loc}}\!\left(\R^N\right)\right\}
		\end{equation*}
		has full measure in $(0,T)$.
		Let us choose any $\tau\in A$ and consider the weak energy solution $u_{R,\tau}$ to \eqref{wpme-funid}, in the sense of Definition \ref{def2}, whose initial datum is $u(\tau) \, \chi_{B_R}$ for a fixed arbitrary $R\ge 1$. Note that, in view of Lemma \ref{globloc}, we can assume without loss of generality that $ u(\tau) $ is a true initial trace for $u$ at time $t=\tau$. The existence, uniqueness, and non-negativity of $ u_{R,\tau} $ is guaranteed by Proposition \ref{pro1}.
		
		First of all, we apply Proposition \ref{weighted AC lemma} to $u_{R,\tau}$. In particular, using the elementary inequality $(x+y)^\alpha\leq x^\alpha+y^\alpha$ for $x,y\geq0$ and $\alpha=\theta\lambda < 1$ on the right-hand side of \eqref{weighted AC all time est} (recall \eqref{deflambdatheta}), we obtain
		\begin{equation}\label{loc-glob-est}
			\begin{aligned}
				\int_{B_R}u_{R,\tau}(x,0)\,\rho(x)\,dx
				\leq C\left[ t^{-\frac{\theta\lambda}{m-1}} \, R^{\frac{2-\gamma}{m-1}}+t^{\lambda}\,u_{R,\tau}(x_0,t)  \right]^{\frac{1}{\theta\lambda}} ,
			\end{aligned}
		\end{equation}
		for all $t>0$ and for a.e.\ $x_0\in B_\varepsilon $. By construction we have that $ 0 \le u_{R,\tau}(x,0) = u(x,\tau) \, \chi_{B_R}(x) \le u(x,\tau) \in L^\infty_{\mathrm{loc}}\!\left( \R^N \right)$, hence Corollary \ref{global comparison lemma} applies, yielding
		\begin{equation}\label{comp-urt}
			u_{R,\tau}(x,t)\leq u(x,t+\tau) \qquad \text{for a.e.}\;\; (x,t) \in \R^N\times(0,T-\tau)\,.
		\end{equation}
		Therefore, from \eqref{loc-glob-est} we infer
		\begin{equation}\label{loc-glob-est-bis}
			\int_{B_R} u(x,\tau)\,\rho(x)\,dx\leq C\left[t^{-\frac{\theta\lambda}{m-1}} \, R^{\frac{2-\gamma}{m-1}}+t^{\lambda}\,u(x_0,t+\tau)  \right]^{\frac{1}{\theta\lambda}},
		\end{equation}
		for a.e.\ $(x_0,t)\in B_\varepsilon\times(0,T-\tau)$.
        Raising \eqref{loc-glob-est-bis} to the power $\theta\lambda$, taking the
        $L^1_\rho$-average over $(x_0,s)\in B_\varepsilon \times (t,t+\delta) $, and raising back to $ 1/\theta\lambda $ we find
	\begin{equation}\label{small time AC}
		\begin{aligned}
			\int_{B_R}u(x,\tau)\,\rho(x)\,dx\leq C\left[t^{-\frac{\theta\lambda}{m-1}}\, R^{\frac{2-\gamma}{m-1}}
		+(t+\delta)^{\lambda}\left(\fiint_{B_\varepsilon \times (t,t+\delta)}u(x,s+\tau)\,\rho(x)\,dxds\right)\right]^{\frac{1}{\theta\lambda}} \, ,
		\end{aligned}
	\end{equation}
and it is plain that such an estimate now holds for \emph{every} $ t \in (0,T-\tau-\delta) $. If $ \tau + \delta < T/4$, then
$$
u(x,s+\tau) \le \norm{u}_{L^\infty\left(B_\varepsilon \times\left(\frac T4,\frac{3T}{4}\right)\right)} \qquad \text{for a.e.} \; \; (x,s) \in B_\varepsilon \times \left( \tfrac{T}{4}, \tfrac{T}{2}+\delta\right) ,
$$
whence
	\begin{equation}\label{mean}
			\fiint_{B_\varepsilon \times (t,t+\delta)}u(x,s+\tau)\,\rho(x)\,dxds \leq \norm{u}_{L^\infty\left(B_\varepsilon \times\left(\frac T4,\frac{3T}{4}\right)\right)} \qquad \forall  t \in \left( \tfrac T 4 , \tfrac T 2 \right) .
	\end{equation}
Thanks to \eqref{small time AC}--\eqref{mean}, it follows that
\begin{equation}\label{AC unif bound}
	\esslimsup_{\tau\to0}\int_{B_R}u(x,\tau)\,\rho(x)\,dx\leq
	C\left[\left(\tfrac T4\right)^{-\frac{\theta\lambda}{m-1}}R^{\frac{2-\gamma}{m-1}}
	+(T+\delta)^{\lambda}\norm{u}_{L^\infty\left(B_\varepsilon\times\left(\frac T4,\frac{3T}{4} \right)\right)}\right]^{\frac{1}{\theta\lambda}}.
\end{equation}
By \eqref{AC unif bound} and the local weak$^*$ compactness of Radon measures on $\R^N$, see \emph{e.g.}\ \cite[Theorem 1.41]{EvansGar}, a classical diagonal argument ensures that there exist a subsequence $\tau_k\to0$, with $\{\tau_k\}\subset A$, and a non-negative Radon measure $\mu$ such that
\begin{equation}\label{lim-tk}
	\lim_{k\to \infty}\int_{\R^N}\varphi(x)\,u(x,\tau_k)\,\rho(x)\,dx=\int_{\R^N}\varphi\,d\mu \qquad \forall \varphi\in C_c\!\left(\R^N\right) .
\end{equation}
In order to prove that this measure $\mu$ satisfies \eqref{weighted-AC-inq}, we consider a standard cutoff function $\phi_R$ for any $R \ge 1$ defined in \eqref{defcutoff1} and use \eqref{small time AC} and \eqref{lim-tk} to get:
\begin{equation*}
	\begin{aligned}
		\mu(B_R)&\leq \int_{\R^N}\phi_R \,d\mu=\lim_{k\to \infty}\int_{\R^N}\phi_R(x)\,u(x,\tau_k)\,\rho(x)\,dx \leq \limsup_{k\to \infty}\int_{B_{2R}} u(x,\tau_k)\,\rho(x)\,dx \\
		&\leq \limsup_{k\to \infty} \, C\left[t^{-\frac{\theta\lambda}{m-1}}\, (2R)^{\frac{2-\gamma}{m-1}}+(t+\delta)^{\lambda}\left(\fiint_{B_\varepsilon\times (t,t+\delta)}u(x,s+\tau_k)\,\rho(x)\,dxds\right)\right]^{\frac{1}{\theta\lambda}} \\
		& \le 2^{\frac{2-\gamma}{(m-1)\theta\lambda}} \, C \left[t^{-\frac{\theta\lambda}{m-1}}\, R^{\frac{2-\gamma}{m-1}}+(t+\delta)^{\lambda}\left(\fiint_{B_\varepsilon\times (t,t+\delta)}u(x,s)\,\rho(x)\,dxds\right)\right]^{\frac{1}{\theta\lambda}} ,
	\end{aligned}
\end{equation*}
that is \eqref{weighted-AC-inq} upon noticing that $(x+y)^\alpha \leq C_\alpha\,(x^\alpha+y^\alpha)$ for $x,y \ge 0$ and $\alpha=1/\theta\lambda>1$, up to a different multiplying constant, depending on the usual quantities, that we still denote by $C$.

\smallskip
\noindent\textbf{Step 2: Uniqueness {of the initial trace}.}
The measure $\mu$ could in principle depend on the sequence $\tau_k\to0$. In order to show uniqueness, let us suppose there exist another sequence  $\kappa_j\to0$, $\{\kappa_j\}\subset A$, and another non-negative Radon measure $\nu$, such that
\begin{equation*}
	\lim_{j \to \infty}\int_{\R^N}\varphi(x)\,u(x,\kappa_j)\,\rho(x)\,dx=\int_{\R^N}\varphi\,d\nu \qquad \forall \varphi\in C_c\!\left(\R^N\right) .
\end{equation*}
For all $ \tau \in A $ and all $R\ge 1$, let $ u_{R,\tau} $ be the same approximate solution as in Step 1. In what follows, we tacitly assume that $ \varphi \in C^\infty_c\!\left(\R^N\right) $ is fixed, $ \varphi \ge 0 $, and $R$ is so large that $ \operatorname{supp} \varphi \subset B_R $. Thanks to the smoothing estimate \eqref{basic smoothing} and the fact that $ u_{R,\tau}(x,0) = u(x,\tau) \, \chi_{B_R}(x) $, we have that
		\begin{equation}\label{smoothing-R-tau}
			\norm{u_{R,\tau}(t)}_{L^\infty(\R^N)}\leq K\,t^{-\lambda}\left( \int_{B_R} u(x,\tau) \, \rho(x) \, dx \right)^{\theta\lambda}\qquad \forall t > 0 \,.
		\end{equation}
Also, from \eqref{m37} we infer
\begin{equation}\label{L1-R-tau}
			\norm{u_{R,\tau}(t)}_{L^1_\rho(\R^N)}\leq \norm{u_{R,\tau}(0)}_{L^1_\rho(\R^N)} = \int_{B_R}u(x,\tau)\,\rho(x)\,dx \qquad \forall t > 0 \,.
		\end{equation}
Recalling that $ u_{R,\tau} $ is a strong solution, and using \eqref{smoothing-R-tau}--\eqref{L1-R-tau}, we have for a.e.~$t>0$
\begin{align*}
	\left|\frac{d}{dt}\int_{B_{R}} u_{R,\tau}(x,t)\,\varphi(x)\,\rho(x)\,dx\right|
	= & \left|\int_{B_{R}}\Delta\!\left[u_{R,\tau}^m(x,t)\right] \varphi(x)\,dx\right|
	= \left|\int_{B_{R}} u_{R,\tau}^m(x,t)\,\Delta\varphi(x)\,dx\right|\\
	\leq & \,  C \, R^{\gamma} \norm{u_{R,\tau}(t)}^{m-1}_{L^\infty(\R^N)}\int_{B_{R}} u_{R,\tau}(x,t)\,\rho(x)\,dx\\
	\leq & \, C \, t^{-\lambda(m-1)}\left(\int_{B_{R}}u(x,\tau)\,\rho(x)\,dx\right)^{\theta\lambda(m-1)+1}\\
	\le & \, C\,t^{-\lambda(m-1)} \, ,
    \end{align*}
where we let $ C>0 $ denote a general constant depending on $ N,m,\gamma,\underline{C},\overline{C},R , \| \Delta \varphi \|_{L^\infty(\R^N)} $ and $u$ (but independent of $\tau$), which may change from line to line. Note that in the last passage we have exploited \eqref{small time AC}--\eqref{mean}, under the implicit additional requirement $ \tau < T/8 $, with $\delta = T/8$ and $ t=T/2 $. We now integrate the negative version of the above inequality from $0$ to almost every $ t \in (0,T-\tau) $, to obtain
\begin{equation}\label{smoothing for uniqueness 2}
	\begin{aligned}
		\int_{\R^N}u(x,t+\tau)\,\varphi(x)\,\rho(x)\,dx&\geq
		\int_{B_{R}} u_{R,\tau}(x,t)\,\varphi(x)\,\rho(x)\,dx\\
		&\geq\int_{B_{R}} u_{R,\tau}(x,0)\,\varphi(x)\,\rho(x)\,dx-C\,t^{\theta\lambda}\\
		&=\int_{\R^N}u(x,\tau)\,\varphi(x)\,\rho(x)\,dx-C\,t^{\theta\lambda} \, ,
	\end{aligned}
\end{equation}
where in the first passage we have used the comparison estimate \eqref{comp-urt} (also recall the relation \eqref{spexprel}). Before proceeding further, we notice the following property, which is part of the statement:
\begin{equation}\label{cont-prop-phi}
\text{for all} \; \; \varphi \in C_c\!\left(\R^N \right) \; \text{the function} \; \; t \mapsto \int_{\R^N} u(x,t) \, \varphi(x) \, \rho(x) \,  dx \;\; \text{is continuous in} \;\; (0,T) \, .
\end{equation}
If $ \varphi $ is in addition smooth, by \eqref{vw-cond} and a usual time-cutoff argument we readily obtain
$$
\int_{\R^N} u(x,t_2)\,\varphi(x)\,\rho(x)\,dx - \int_{\R^N} u(x,t_1)\,\varphi(x)\,\rho(x)\,dx = \int_{t_1}^{t_2} \int_{\R^N} u^{m}(x,t) \, \Delta \varphi(x) \, dx dt
$$
for all $ t_1,t_2 \in A $ with $ t_1<t_2 $, and from such an identity it is plain that function in \eqref{cont-prop-phi} has a continuous version. On the other hand, if $ \varphi $ is merely continuous, then we can approximate it uniformly by a sequence $ \{ \varphi_n \} \subset C_c^\infty\!\left( \R^N \right) $, so that the sequence $ t \mapsto \int_{\R^N} u(x,t)\,\varphi_n(x)\,\rho(x)\,dx $ is uniformly equicontinuous and thus \eqref{cont-prop-phi} is preserved at the limit. As a result, we are allowed to extend \eqref{smoothing for uniqueness 2} to \emph{every} $ t \in (0,T-\tau) $. Taking the limit as $ \tau = \tau_k \to 0 $, we reach
\begin{equation*}
	\int_{\R^N}u(x,t)\,\varphi(x)\,\rho(x)\,dx\geq\int_{\R^N}\varphi  \,d\mu-C\,t^{\theta\lambda} \qquad \forall t \in (0,T) \, .
\end{equation*}
Choosing $t=\kappa_j\in A$ and letting $ j \to \infty $ in the above inequality yields
\begin{equation}\label{AB}
	\int_{\R^N}\varphi\,d\nu\,\geq\int_{\R^N}\varphi\,d\mu \, .
\end{equation}
Again, via a density argument, we deduce that \eqref{AB} actually holds for all non-negative $ \varphi \in C_c\!\left( \R^N \right) $. By interchanging the roles of $ \{ \tau_k \} $ and $ \{ \kappa_j \} $, we find that the opposite inequality holds as well, so $\mu$ and $\nu$ coincide as non-negative linear functionals on $C_c(\R^N)$; then by the Riesz representation theorem \cite[Theorem 1.38]{EvansGar}, they also coincide as Radon measures.

We have thus shown that if any sequence $ \kappa_j \to 0  $, $ \{ \kappa_j \} \subset A $, is such that the weak$^\ast$ limit of $u(\kappa_j)$ exists, then the latter must coincide with $\mu$. Moreover, the same compactness argument used in Step 1 to show the existence of $\mu$ implies that for any such $ \{ \kappa_j \} $ there is always a subsequence $ \left\{ \kappa_{j_i} \right\} $ along which the weak$^*$ limit exists, hence $u(\kappa_{j_i})\to \mu$. As the weak$^*$ limit is independent of the particular subsequence, we can finally assert that \eqref{measure-data} holds provided $t \in A$; on the other hand, property \eqref{cont-prop-phi} ensures that the same limit is taken for arbitrary $ t \to 0 $.
\end{proof}

We state, for future use, the following immediate consequence of \eqref{small time AC}--\eqref{mean}.

\begin{cor}\label{small time AC lem}
Let $u$ be a non-negative very weak solution of equation \eqref{wpme-noid}, in the sense of Definition \ref{def:vw-gl}, such that $ u \in L^\infty_{\mathrm{loc}}\!\left( \R^N\times (0,T) \right) $. Then for all $ \varepsilon \in (0,1) $, all $R \ge 1$, and a.e.~$ \tau \in (0,T/8) $, it holds
\begin{equation}\label{AC unif bound-2}
	\int_{B_R} u(x,\tau)\,\rho(x)\,dx\leq
	C\left[T^{-\frac{\theta\lambda}{m-1}}R^{\frac{2-\gamma}{m-1}}
	+T^{\lambda}\norm{u}_{L^\infty\left(B_\varepsilon\times\left(\frac T4,\frac{3T}4 \right)\right)}\right]^{\frac{1}{\theta\lambda}} ,
\end{equation}
for a constant $C>0$ depending only on $N,m,\gamma,\underline{C},\overline{C}$.
\end{cor}

\section{A sharp local smoothing estimate: proofs}\label{sec:local}
In this section we prove the local $L^\infty$ estimates in Theorem \ref{local moser iter lem} and Corollary \ref{glob-unif-est}. The proof  consists in an involved parabolic Moser iteration in the space-time Lebesgue spaces introduced in Definition \ref{def:weighted-norm}; to this aim, we carefully adapt the strategy developed in \cite[Section 2]{DK}, upon handling the presence of the weight in a similar way to \cite[Section 2]{BS1}, where, however, only pure powers were dealt with. We recall that Theorem \ref{local moser iter lem} is valid for \emph{locally bounded} and possibly \emph{sign-changing} solutions, and Corollary \ref{glob-unif-est} for (\emph{a priori} unbounded) \emph{non-negative} solutions is a simple consequence of the latter combined with Theorem \ref{apriori-bounded}. We will first establish the result for \emph{strong energy} solutions (see Definition \ref{def:w-loc} below), for which the computations are justified. In order to pass to general \emph{very weak} solutions in the sense of Definition \ref{vw loc def}, we perform a local approximation method that turns out to be somewhat delicate, and strongly relies on the local comparison of Proposition \ref{local comparison lemma}.

With a view to the Moser iteration below, let us quickly define suitable weighted analogues of the Sobolev space $H^1(\Omega)$, for a fixed bounded and smooth domain $\Omega\subset\R^N$. That is, we introduce the space $H^1_\rho(\Omega)$, which consists of all measurable functions $ f$ satisfying
\begin{equation*}\label{def Sobolev norm}
\norm{f}_{H^1_\rho(\Omega)}=\left(\norm{\nabla f}_{L^2(\Omega)}^2+\norm{f}_{L^2_\rho(\Omega)}^2\right)^{\frac{1}{2}}<+\infty\,,
\end{equation*}
where we recall in \eqref{weighted-leb-norm} the standard $\rho$-weighted Lebesgue $p$-norms for $p\in[1,\infty)$. Also, let us denote here the \emph{weighted} Sobolev critical exponent
\begin{equation}\label{def2star}
2^*=2\,\tfrac{N-\gamma}{N-2}\,.
\end{equation}

We begin by showing that the following weighted Sobolev--Poincar\'{e} inequality holds on centered balls $B_R$ of sufficiently large radii $R$.
\begin{lem}\label{weighted.sobolev}
Let $N\ge3$ and $\rho$ be a measurable function satisfying \eqref{weight-cond}. Let $R\geq1$. Then, for any $f\in H_\rho^1(B_R)$, it holds
\begin{equation}\label{Sobolev.inequality}
	\norm{f}_{L^{2^*}_\rho(B_R)} \leq C \left(R^{-2+\gamma}\norm{f}^2_{L^2_\rho(B_R)} + \norm{\nabla f}^2_{L^2(B_R)}
	\right)^{\frac 1 2} ,
\end{equation}
for a constant $ C>0 $ depending only on $N,\gamma,\underline{C},\overline{C}$.
\end{lem}
\begin{proof}
Let us begin by considering the unit ball $B_1$; the estimate \eqref{Sobolev.inequality} for general $R \ge 1$ will then follow by a simple spatial scaling argument. Given any $f\in H^1(B_1)$, let $E f \in H^1\!\left(\R^N\right)$ denote its standard Sobolev extension in the whole of $ \R^N$. Then, we have
\begin{equation}\label{ext-1}
	\begin{aligned}
		\norm{\nabla(Ef)}^2_{L^2(\R^N)}\leq \norm{Ef}_{H^1(\R^N)}^2\leq c_N \norm{f}_{H^1(B_1)}^2= c_N \left(\norm{f}^2_{L^2(B_1)}+\norm{\nabla f}^2_{L^2(B_1)}\right) ,
	\end{aligned}
\end{equation}
where the constant $c_N>0$ depends only on $N$. Now, to the left-hand side of \eqref{ext-1}, we will apply a weighted Sobolev inequality that is a routine interpolation between the classical Hardy and Sobolev inequalities, see also \cite[Formula (3.13)]{MP}, obtaining
\begin{equation}\label{weighted-Sob}
	\norm{\nabla(Ef)}^2_{L^2(\R^N)} \geq c_{N,\gamma} \norm{Ef}^2_{L^{2^*}_{\gamma}(\R^N)}\ge c_{N,\gamma} \norm{f}_{L^{2^*}_\gamma(B_1)}^2 ,
\end{equation}
for another constant $ c_{N,\gamma}>0 $ depending only on $ N,\gamma $. Combining \eqref{ext-1} and \eqref{weighted-Sob} gives
\begin{equation}\label{B1-ineq}
	\norm{f}_{L^{2^*}_\gamma(B_1)}^2\leq \frac{c_N}{c_{N,\gamma}} \left(\norm{f}^2_{L^2(B_1)}+\norm{\nabla f}^2_{L^2(B_1)}
	\right) .
\end{equation}
Assume now $ f \in H^1(B_R) $; upon applying \eqref{B1-ineq} to $ x \mapsto f(Rx) $, we end up with
\begin{equation}\label{B2-ineq}
	\norm{f}_{L^{2^*}_\gamma(B_R)}^2\leq \frac{c_N}{c_{N,\gamma}} \left( R^{-2}
	\norm{f}^2_{L^2(B_R)}+\norm{\nabla f}^2_{L^2(B_R)}\right) ;
\end{equation}
on the other hand, using \eqref{weight-cond}, from \eqref{B2-ineq} we infer
\begin{equation*}
	\norm{f}_{L^{2^*}_\rho(B_R)}^2\leq \overline{C} \,  \frac{c_N}{c_{N,\gamma}} \left[ \underline{C}^{-1} \, R^{-2}\left( 1 + R \right)^\gamma
	\norm{f}^2_{L^2_\rho(B_R)}+\norm{\nabla f}^2_{L^2(B_R)}\right] ,
\end{equation*}
whence \eqref{Sobolev.inequality} recalling that $R \ge 1$ and that, trivially, $ H^1(B_R) \supset H^1_\rho(B_R) $.
\end{proof}
The next inequality is a simple consequence of Lemma \ref{weighted.sobolev}.
\begin{lem}[A parabolic Sobolev inequality]\label{Sobolev-lem}
Let $N\ge3$ and $\rho$ be a measurable function satisfying \eqref{weight-cond}. Let us consider the cylinder $Q=B_R\times(T_0,T_1)$ for $R \ge 1$ and $0\leq T_0<T_1$. Then, for any
$ f \in L^{2}\!\left((T_0,T_1);H^{1}_{\rho}(B_R)\right) \cap L^\infty(Q) $ and for all $a\in [1,2^*/2]$, the inequality
\begin{equation}\label{iterative_sobolev_inequality}
	\begin{aligned}
		\iint_{Q} |f(x,t)|^{2a}\,\rho(x)\,dxdt\leq &\, C \iint_Q\left[R^{-2+\gamma} \, |f(x,t)|^{2} \, \rho(x) + | \nabla f(x,t)|^{2}\right]dxdt \\
		& \times \underset{t \in \left (T_0,T_1 \right )}{\esssup} {\left (\int_{B_R}|f(x,t)|^{2\left(a-1\right)\frac{1}{\theta}} \,\rho(x) \,dx \right )}^{\theta}
	\end{aligned}
\end{equation}
holds for a constant $C>0$ depending only on $N,\gamma,\underline{C},\overline{C}$.
\end{lem}
\begin{proof}
Fix any $t\in(T_0,T_1)$. Since $2^*/2$ and $1/\theta$ are conjugate exponents (recall the respective definitions \eqref{def2star} and \eqref{deflambdatheta}), using H\"{o}lder's inequality and~\eqref{Sobolev.inequality} we have
\begin{equation*}
	\begin{aligned}
		\int _{B_{R}} |f(x,t)|^{2a} \,\rho(x)\,dx \leq & \left (\int _{B_{R}}|f(x,t)|^{2^*}\rho(x)\,dx \right )^{\frac{2}{2^*}} \left( \int _{B_{R}}|f(x,t)|^{2\left (a -1 \right )\frac{1}{\theta}} \, \rho(x)\,dx \right )^{\theta} \\
		\leq & \, C \left (\int _{B_{R}}R^{-2+\gamma} \, |f(x,t)|^{2} \, \rho(x)\,dx +  \int_{B_{R}}| \nabla f(x,t)|^{2} \, dx \right ) \\
		 & \times \underset{t \in \left (T_0,T_1 \right )}{\esssup }{\left(\int_{B_{R}}|f(x,t)|^{2\left(a-1\right)\frac{1}{\theta}}\,\rho(x)\,
			dx \right )^{\theta}} .
	\end{aligned}
\end{equation*}
By the assumptions on $f$, the above inequality is valid for a.e.\ $t\in(T_0,T_1)$, so that \eqref{iterative_sobolev_inequality} follows by integrating in time from $T_0$ to $T_1$.
\end{proof}

The following energy estimate, inspired by \cite[Corollary to Lemma 2]{DK}, will be one of the two main ingredients, along with the parabolic Sobolev embedding in Lemma \ref{Sobolev-lem}, used in a Moser iteration to prove Theorem \ref{local moser iter lem}. We emphasize that from now on we are dealing with possibly sign-changing \emph{strong energy solutions} (which have enough regularity to justify the computations that will follow), so we tacitly use the already stated convention $u^p=|u|^{p-1}u$ for $p\geq1$.

\begin{pro}[Local energy estimates]\label{estimates.balls}
Let $ \Omega \subset \R^N $ be a (possibly unbounded) domain and $ 0 \le \tau_1<\tau_2 $. Let $u$ be a strong energy solution to~\eqref{wpme-noid-loc}, in the sense of Definition \ref{def:w-loc}. Let us consider the cylinders
\begin{equation}\label{cyl-energy}Q_1=B_{R_1}\times(T_1,T^*)\qquad\text{and}\qquad Q_0=B_{R_0}\times(T_0,T^*)\,,
\end{equation}
for $1 \le R_1<R_0$ and $0 < T_0<T_1\le T^*$, such that $Q_1\subset Q_0\Subset  \Omega \times (\tau_1,\tau_2) $. Then, for all $p\geq p_0>1$, it holds
\begin{equation}\label{wanted.inq}
	\begin{aligned}
		&\,\sup_{t\in(T_1,T^*)}\int_{B_{R_1}} |u(t)|^p\,\rho \,dx+\iint_{Q_1} \left|\nabla\!\left( u^\frac{p+m-1}{2} \right) \right|^2 dxdt\\
		\leq & \, C \iint_{Q_0} \left[ \frac{1}{T_1-T_0}\,|u|^p +\frac{(1+ R_0)^\gamma}{(R_0-R_1)^2}\,|u|^{p+m-1} \right] \rho\,dxdt\,,
	\end{aligned}
\end{equation}
for a constant $C>0$ depending only on $m,\underline{C},p_0$.
\end{pro}
\begin{proof}
First of all, we choose a space-time cutoff function $\varphi\in C^\infty_c(\Omega \times (\tau_1,\tau_2))$ such that
	\begin{equation}\label{test.function.balls-prop}
0\leq \varphi \leq 1 \, , \qquad \varphi=1 \quad \text{in}\;\; Q_1 \, , \qquad \varphi=0 \quad \text{in} \;  \left( B_{R_0^c} \times (\tau_0,T^*) \right) \cup \left(\Omega \times (\tau_0,T_0)\right)
    \end{equation}
and
	\begin{equation}\label{test.function.balls}
		\left|\nabla \varphi\right|\leq \frac{C_\varphi}{R_0-R_1} \, , \qquad \left| \varphi_t\right|\leq \frac{C_\varphi}{T_1-T_0}\, ,
	\end{equation}
for some numerical constant $ C_\varphi>0 $.
Applying the weak formulation \eqref{w-cond-loc} of $u$ to the test function $\psi = u^{p-1}\,\varphi^2$, we find:
\begin{equation}\label{eq1-bdd}
	\frac1p\iint_{Q_0} \partial_t \!\left(|u|^p\right) \varphi^2\,\rho\,dxdt + \iint_{Q_0} \left\langle \nabla (u^m) \, , \nabla \!\left(u^{p-1}\,\varphi^2\right) \right\rangle dxdt=0\,.
\end{equation}
Notice that
\begin{align*}
	& \left\langle \nabla (u^m) \, , \nabla\! \left(u^{p-1}\,\varphi^2\right) \right\rangle \\ =& \,\frac{4m(p-1)}{(m+p-1)^2} \left|\nabla \!\left(u^\frac{m+p-1}{2}\right)\right|^2 \varphi^2 + \frac{4m}{m+p-1}\left\langle \varphi \,\nabla\!\left( u^{\frac{m+p-1}{2}} \right) , \, u^\frac{m+p-1}{2} \, \nabla \varphi\right\rangle ,
\end{align*}
so \eqref{eq1-bdd} becomes
\begin{equation}\label{ineq2}
	\begin{aligned}
		& \, \frac1p\iint_{Q_0} \partial_t\!\left(|u|^p\right) \varphi^2\,\rho\,dxdt + \frac{4m(p-1)}{(m+p-1)^2}\iint_{Q_0} \left|\nabla \! \left(u^\frac{m+p-1}{2}\right)\right|^2 \varphi^2\,dxdt \\
		= & -\frac{4m}{m+p-1}\iint_{Q_0}\left\langle \varphi \,\nabla\!\left( u^{\frac{m+p-1}{2}} \right) , \, u^\frac{m+p-1}{2} \, \nabla \varphi\right\rangle dxdt\,.
	\end{aligned}
\end{equation}
Now we exploit Young's inequality on the right-hand side of \eqref{ineq2},
obtaining
\begin{equation}\label{ineq3}
	\begin{aligned}
& \, \frac1p\iint_{Q_0} \partial_t\!\left(|u|^p\right) \varphi^2\,\rho\,dxdt + \frac{2m(p-1)}{(m+p-1)^2}\iint_{Q_0} \left|\nabla \! \left(u^\frac{m+p-1}{2}\right)\right|^2 \varphi^2\,dxdt \\
		 \leq & \, \frac{2m}{p-1}\iint_{Q_0}|u|^{m+p-1} \left|\nabla \varphi\right|^2 dxdt \, .
	\end{aligned}
\end{equation}
By the assumed time regularity of $u$, we may integrate by parts in the first term on the left-hand side of \eqref{ineq3} to find that
\begin{equation}\label{inq3}
\begin{aligned}
		& \, \frac{1}{p} \int_{B_{R_0}} |u(T^*)|^p\,\varphi^2(T^*)\,\rho\,dx - \frac1p \int_{B_{R_0}} |u(T_0)|^p\,\varphi^2(T_0)\,\rho\,dx \\
		& + \frac{2m(p-1)}{(p+m-1)^2}\iint_{Q_0}\left|\nabla \!\left(u^\frac{m+p-1}{2}\right)\right|^2 \varphi^2\,dxdt\\
        \leq & \, \frac{2}{p} \iint_{Q_0} |u|^p \, \varphi\, \varphi_t\,\rho\,dxdt + \, \frac{2m}{p-1} \iint_{Q_0} |u|^{p+m-1} \left|\nabla \varphi \right|^2 dxdt \, .
\end{aligned}
\end{equation}
We point out that the above computations are rigorously justified for $p\geq m+1$ only, but in fact estimate \eqref{inq3} also holds for all $ p \in (1,m+1) $ via a non-singular approximation of the test function $ (u^m)^\alpha $ when $\alpha \in (0,1)$, for instance by replacing $ v^\alpha $ with the odd primitive of $ \alpha(|v|\vee \varepsilon)^{\alpha-1} $ and letting $ \varepsilon \to 0 $.

In view of \eqref{weight-cond} and the support properties of $\varphi$, from \eqref{inq3} we can deduce
\begin{equation*}\label{inq.4}
\begin{aligned}
		&\,\frac{p-1}{p}\int_{B_{R_1}} \left|u(T^*)\right|^p \rho\,dx + \frac{2m(p-1)^2}{(p+m-1)^2}\iint_{Q_1}\left|\nabla \!\left(u^\frac{m+p-1}{2}\right)\right|^2\,dxdt \\
		\leq & \, \frac{2(p-1)C_\varphi}{p}\cdot\frac{1}{T_1-T_0}\iint_{Q_0}|u|^p \, \rho\,dxdt+ \frac{2m \, C_\varphi^2}{\underline{C}} \cdot \frac{(1+R_0)^\gamma}{(R_0-R_1)^2}\iint_{Q_0} |u|^{p+m-1} \, \rho\, dxdt \, .
\end{aligned}
\end{equation*}
Let us now take
\begin{equation}\label{defc}
	C_p=\max\!\left(\frac{p}{p-1} \, ,\frac{(p+m-1)^2}{2m(p-1)^2}\right) \cdot \max\left(\frac{2(p-1)\,C_\varphi}{p} \, ,\frac{2m \, C_\varphi^2}{\underline{C}}\right) ;
\end{equation}
we then have
\begin{equation}\label{wanted.inqt}
\begin{aligned}
		&\int_{B_{R_1}} |u(T^*)|^p \,\rho\,dx + \iint_{Q_1}\left|\nabla \!\left( u^\frac{m+p-1}{2} \right) \right|^2\,dxdt \\
		 \leq & \, C_p \left[ \frac{1}{T_1-T_0}\iint_{Q_0} |u|^p \, \rho\,dxdt+ \frac{(1+R_0)^\gamma}{(R_0-R_1)^2}\iint_{Q_0} |u|^{p+m-1} \, \rho \, dxdt\right] .
\end{aligned}
\end{equation}
To obtain inequality~\eqref{wanted.inq} with the supremum on the left-hand side, we first proceed as above with $t^*\in(T_1,T^*)$ replacing $T^*$, getting \eqref{wanted.inqt} for the cylinders $Q_0^*=B_{R_0}\times(T_0,t^*)\subset Q_0$ and $Q_1^*=B_{R_1}\times(T_1,t^*)$. The right-hand side of the resulting inequality is smaller than the corresponding one in \eqref{wanted.inqt}, since $Q_0^*\subset Q_0$ and the multiplying coefficients coincide. We then take the supremum over $t^*$, replacing $C_p$ with $2C_p$. Finally, if $ p \ge p_0 >1 $, it is plain from \eqref{defc} that one can in turn replace $C_p$ with $C=\sup_{p \ge p_0} C_p < +\infty$.
\end{proof}

\begin{proof}[Proof of Theorem \ref{local moser iter lem} (for strong energy solutions)]
Let us give a brief outline of the proof. In Step 1, we suitably combine the Sobolev inequality of Lemma \ref{Sobolev-lem} and the energy estimate of Proposition \ref{estimates.balls} to obtain a fundamental base inequality, see formula \eqref{comb2} below. In Step 2, such an inequality is iterated in an infinite family of nested cylinders to obtain an $L^{p_0}_{\rho}\to L^\infty$ smoothing estimate in parabolic cylinders, which works provided $p_0>m-1$. Clearly, when $m<2$, it is not difficult to pass to the desired $L^1_{\rho}\to L^\infty$ smoothing estimate, which is done in Step 3a. Finally, in Step 3b, we handle the case of general $m>1$, where it is necessary to make a further application of the energy estimate and a delicate interpolation iteration. In the latter case, an $L^1_{\rho}\to L^\infty$ estimate is not available for $1<m<2$, so we resort to obtaining the weaker (but unavoidable) $ \mathsf{S} \to L^\infty$ smoothing estimate.

Throughout, we will implicitly assume that $u$ is in addition a strong energy solution, in the sense of Definition \ref{def:w-loc}.

\smallskip
\noindent\textbf{Step 1: A first inequality.}
Let us consider $f^2=u^{p+m-1}$ in the parabolic Sobolev inequality~\eqref{iterative_sobolev_inequality}, for some $p>1$ and a cylinder
\begin{equation}\label{par-cyl-contained}
	Q_1=B_{R_1}\times(T_1,T^*)\,,\quad\text{for}\;\; R_1 \ge 1\;\; \text{such that} \;\; B_{R_1} \Subset \Omega \, , \ \text{and}\;\;\tau_1 < T_1<T^*< \tau_2 \, .
\end{equation}
Along the proof we will deal with increasing (or decreasing) sequences of parabolic cylinders $\{ Q_k \}$, so from now on, we take for granted that all such $Q_k$ is in fact contained in $ \Omega \times (0,T) $, in the sense that \eqref{par-cyl-contained} holds. With the above choices, we obtain
\begin{equation}\label{comb}
	\begin{aligned}
		\iint_{Q_1} |u|^{a(p+m-1)}\,\rho\,dxdt \leq & \, C\iint_{Q_1}\left[R_1^{-2+\gamma}\,|u|^{p+m-1}\,\rho + \left| \nabla \!\left( u^{\frac{p+m-1}{2}} \right) \right|^{2}\right]dxdt\\
		 & \times \sup _{t \in \left (T_1,T^* \right ) }\!{\left ( \int _{B_{R_{1}}}|u(t)|^{(p+m-1)\left (a - 1\right )\frac{1}{\theta}} \,\rho \,dx \right )}^{\theta} ,
	\end{aligned}
\end{equation}
where, here and in the sequel, $C$ denotes a general positive constant depending only on $N,m,\gamma,\overline{C}, \underline{C}$, and possibly on $ p_0$ for some fixed $p_0>1$, but whose exact value may change contextually. Note that at this stage we can still choose the parameter $a$ freely in $[1,2^*/2]$.

Next, we prepare to apply the energy estimate of Proposition \ref{estimates.balls} to the right-hand side of \eqref{comb}. To this end, let us introduce a larger cylinder
\begin{equation*}
	Q_0=B_{R_0}\times(T_0,T^*)\,,\quad\text{for}\;\;R_0>R_1\;\;\text{and}\;\;0 < T_0<T_1\,.
\end{equation*}
Using the energy inequality~\eqref{wanted.inq}, we may control the gradient term in \eqref{comb} by
\begin{equation*}
	\iint_{Q_1} \left|\nabla  \!\left(u^\frac{p+m-1}{2} \right) \right|^2 dx dt
	\leq C \left[\frac{1}{T_1-T_0}+\frac{(1+ R_0)^\gamma}{(R_{0}-R_{1})^2}\norm{u}^{m-1}_{L^\infty(Q_0)}\right]\!\iint_{Q_{0}}|u|^{p}\, \rho\,dxdt\,.
\end{equation*}
Therefore, \eqref{comb} becomes
\begin{equation*}\label{comb1}
	\begin{aligned}
		\iint_{Q_{1}} |u|^{a(p+m-1)}\,\rho\,dxdt
		\le &\,C\left[\frac{1}{T_{1}-T_{0}}+\left(R_{1}^{-2+\gamma}+\frac{(1+ R_{0})^\gamma}{(R_{0}-R_{1})^2}\right)\norm{u}_{L^\infty(Q_0)}^{m-1}\right] \iint_{Q_{0}} |u|^{p}\,\rho\,dxdt\\
		&\times {\sup _{t \in \left (T_{1},T^* \right )}\left(\int_{B_{R_{1}}}|u(t)|^{(p+m-1)\left (a - 1\right )\frac{1}{\theta}}\,\rho\,dx \right)}^{\theta} .
	\end{aligned}
\end{equation*}
In order to simplify the above expression, from now on we will work under the assumption that $R_0$ and $R_1$ are sufficiently close; quantitatively, we require that:
\begin{equation}\label{close-cond}
	\text{if}\;\;R_0=\alpha\,R_1 \;\;\text{for some}\;\;\alpha>1\,, \ \text{then}\;\;(\alpha-1)^2\leq  c \,\alpha^\gamma\,,
\end{equation}
for some numerical constant $ c>0 $. Then, for all $a\in[1,2^*/{2}]$ and all $p\geq p_0>1$, we have
\begin{equation}\label{comb2}
	\begin{aligned}
		\iint_{Q_{1}} |u|^{a(p+m-1)}\,\rho\, dxdt
		\le &\,C\left[ \frac{1}{T_{1}-T_{0}}+\frac{(1+ R_{0})^\gamma}{(R_{0}-R_{1})^2}\norm{u}_{L^\infty(Q_0)}^{m-1} \right] \iint_{Q_{0}} |u|^{p}\,\rho\,dxdt\\
		&\times {\sup _{t \in \left (T_{1},T^* \right )}\left(\int_{B_{R_{1}}} |u(t)|^{(p+m-1)\left (a - 1\right )\frac{1}{\theta}}\,\rho\,dx \right)}^{\theta} .
	\end{aligned}
\end{equation}

\smallskip

\noindent\textbf{Step 2: An $\boldsymbol{L^{p_0}_\rho\to L^\infty}$ smoothing effect.} We now iterate \eqref{comb2} on the following infinite sequence of nested cylinders indexed by $k\in\mathbb{N}$:
    \begin{gather}
    Q_k=B_{R_k}\times(T_k,T^*)\,, \nonumber \\
	\tau_1<T_0<T_1<\cdots<T_k<T_{k+1}<\cdots\;\;\text{and}\;\;T_k\nearrow T_\infty\in(0,T^*)\,,\label{cyl-nest-t}  \\
	\;\; R_0>R_1>\cdots>R_{k}>R_{k+1}>\cdots\;\;\text{and}\;\;R_k\searrow R_\infty \ge 1\,.\label{cyl-nest-r}
\end{gather}
We need to allow $p$ and $a$ to also depend on $ k \in \mathbb{N} $, so the $k$-th version of \eqref{comb2} becomes
\begin{equation}\label{comb2a}
	\begin{aligned}
		\iint_{Q_{k+1}} |u|^{a_k(p_k+m-1)} \, \rho \, dxdt
		\le & \, C\left[\frac{1}{T_{k+1}-T_{k}}+\frac{(1+ R_{k})^\gamma}{(R_{k}-R_{k+1})^2}\norm{u}_{L^\infty(Q_k)}^{m-1}\right] \iint_{Q_{k}} |u|^{p_k} \, \rho \,dxdt\\
		&\times {\sup _{t \in \left (T_{k+1},T^* \right )}\left(\int_{B_{R_{k+1}}} |u(t)|^{(p_k+m-1)\left (a_k - 1\right )\frac{1}{\theta}}\,\rho\,dx \right)}^{\theta} .
	\end{aligned}
\end{equation}
The above inequality is valid since \eqref{cyl-nest-r} clearly implies \eqref{close-cond} if the latter is true with $ R_1 \equiv R_\infty $ (a fact that we take for granted in the sequel), and the constant $C>0$ is independent of $k$ as long as $p_k\geq p_0$.

Let us use \eqref{comb2a} with the choice that $a_k$ satisfies
\begin{equation}\label{choicea1}
	(p_k+m-1)(a_k-1)\frac{1}{\theta}=p_k\,,
\end{equation}
which is easily seen to be compatible with the condition $a_k\in[1,2^*/2]$. Under this condition, we now apply the energy estimate \eqref{wanted.inq} to the term in \eqref{comb2a} involving the supremum, obtaining
\begin{equation}\label{comb3}
	\begin{aligned}
		\iint_{Q_{k+1}} |u|^{a_k(p_k+m-1)} \,\rho\,dxdt  \leq & \, C \left[\frac{1}{T_{k+1}-T_{k}}+\frac{(1+ R_{k})^\gamma}{(R_{k}-R_{k+1})^2}\norm{u}_{L^\infty(Q_k)}^{m-1}\right]^{1+\theta}\\
        & \times \left(\iint_{Q_k} |u|^{p_k}\,\rho\,dxdt\right)^{1+\theta} .
	\end{aligned}
\end{equation}
Therefore, we begin our iteration of \eqref{comb3} by defining $p_{k+1}$ to satisfy
\begin{equation*}
	p_{k+1}=a_k(p_k+m-1)\,.
\end{equation*}
From this definition and \eqref{choicea1}, it can be checked by a simple recursion that
\begin{equation}\label{iterrelation}
	p_{k+1}=(1+\theta)^{k+1} \left(p_0+\frac{m-1}{\theta}\right)-\frac{m-1}{\theta} \, .
\end{equation}
To simplify the notation in what follows, we also define
\begin{equation}\label{def-ik}
	I_{k}=C \left[\frac{1}{T_{k+1}-T_{k}}+\frac{(1+ R_{k})^\gamma}{(R_{k}-R_{k+1})^2}\norm{u}_{L^\infty(Q_k)}^{m-1}\right] ,
\end{equation}
and notice that the constant $C>0$ in \eqref{comb3} is contained in this term. In this way, \eqref{comb3} is rewritten as
\begin{equation*}
	\norm{u}_{L^{p_{k+1}}_\rho(Q_{k+1})} \leq I_{k}^\frac{1+\theta}{p_{k+1}}  \norm{u}_{L^{p_k}_\rho(Q_{k})}^{\left(1+\theta\right)\frac{p_k}{p_{k+1}}} .
\end{equation*}
After $k$ iterations, we end up with
\begin{equation}\label{iteration}
	\norm{u}_{L^{p_{k+1}}_\rho(Q_{k+1})}
	\leq \left( \prod_{j=0}^{k} I_{j}^{\frac{\left(1+\theta\right)^{k+1-j}}{p_{k+1}}} \right) \norm{u}_{L^{p_0}_\rho(Q_0)}^{\left(1+\theta\right)^{k+1}\frac{p_0}{p_{k+1}}} .
\end{equation}
Now, let us carefully analyze the terms on the right-hand side of \eqref{iteration} as $k\to\infty$, starting from the exponent of the $L^{p_0}_\rho$-norm. Indeed, by \eqref{iterrelation}, it follows that
\begin{equation}\label{k-lim-exp}
	(1+\theta)^{k+1}\,\frac{p_0}{p_{k+1}}=\frac{(1+\theta)^{k+1}\,p_0}{(1+\theta)^{k+1} \left(p_0+\frac{m-1}{\theta}\right)-\frac{m-1}{\theta}} \ \underset{k \to \infty}{\longrightarrow} \ \frac{p_0}{p_0+\frac{m-1}{\theta}} \, .
\end{equation}
Next, in order to estimate the product appearing in \eqref{iteration}, let us choose $\{T_k\}$ and $\{R_k\}$ in \eqref{cyl-nest-t}--\eqref{cyl-nest-r} so that
\begin{equation}\label{radii.times.upper.iteration}
	R_{k}- R_{k+1}= \mathsf{C}_1\,\frac{R_{0}- R_{\infty}}{\left (k+1 \right)^2} \qquad \text{and} \qquad
	T_{k+1}-T_{k} = \mathsf{C}_2\,\frac{ T_{\infty }-T_{0} }{\left (k+1 \right)^4} \, ,
\end{equation}
where $\mathsf{C}_1 = \sum_{k=0}^\infty\,(k+1)^{-2}$ and $\mathsf{C}_2 =\sum_{k=0}^\infty\,(k+1)^{-4}$. The term $I_{k}$ defined in \eqref{def-ik} can then be estimated as
\begin{equation*}
	I_{k}\leq C \left[ \frac{1}{T_{\infty}-T_{0}}+\frac{(1+ R_{0})^\gamma}{(R_{0}-R_{\infty})^2}\norm{u}_{L^\infty(Q_0)}^{m-1}\right](k+1)^{4} \, ,
\end{equation*}
and therefore the above product by
\begin{equation}\label{prod-est}
	\begin{aligned}
		& \prod_{j=0}^{k} I_{j}^{\frac{\left(1+\theta\right)^{k+1-j}}{p_{k+1}}} \\
		\leq & \prod_{j=0}^{k} \left\{ C \left[ \frac{1}{T_{\infty} -T_{0}}+\frac{(1+ R_{0})^\gamma}{(R_{0}-R_{\infty})^2}\norm{u}_{L^\infty(Q_0)}^{m-1}\right](j+1)^4\right\}^{\frac{\left(1+\theta\right)^{k+1-j}}{p_{k+1}}} \\
		\leq & \left\{ C \left[\frac{1}{T_{\infty}-T_{0}}+\frac{(1+ R_{0})^\gamma}{(R_{0}-R_{\infty})^2}\norm{u}_{L^\infty(Q_0)}^{m-1}\right]^{1+\theta}\right\}^{\sum_{j=0}^k\frac{\left(1+\theta\right)^{j}}{p_{k+1}}} \, \prod_{j=0}^{k}  (j+1)^{4 \, \frac{\left(1+\theta\right)^{k+1-j}}{p_{k+1}}}\,.
	\end{aligned}
\end{equation}
Here, there are two terms to study as $k\to\infty$. First, using \eqref{iterrelation}, we have
\begin{equation}\label{i-exp-lim}
	\sum_{j=0}^k\frac{\left(1+\theta\right)^{j}}{p_{k+1}}=\frac{\frac{1}{\theta}\left[\left(1+\theta\right)^{k+1}-1\right]}{(1+\theta)^{k+1} \left(p_0+\frac{m-1}{\theta}\right)-\frac{m-1}{\theta}} \ \underset{k \to \infty}{\longrightarrow} \ \frac{1}{\theta p_0+m-1} \, .
\end{equation}
Hence, it only remains to prove that
\begin{equation}\label{last.term}
	\prod_{j=0}^{\infty} (j+1)^{4 \, \frac{\left(1+\theta\right)^{k+1-j}}{p_{k+1}}}<+\infty\,.
\end{equation}
Indeed, by \eqref{iterrelation} it is plain that $ p_{k+1} \ge (1+\theta)^{k+1}\,p_0 $, thus
we can make the following estimate:
\begin{equation}\label{prod-prod}
		\prod_{j=0}^{k} (j+1)^{4 \, \frac{\left(1+\theta\right)^{k+1-j}}{p_{k+1}}}
		\leq \exp\left(\frac{4}{p_0} \, \sum_{j=0}^{k}(1+\theta)^{-j} \, \log(j+1) \right)  ,
\end{equation}
and the right-hand side clearly converges as $k\to\infty$ to a positive number depending only on $N,\gamma,p_0$. In view of \eqref{k-lim-exp}, \eqref{prod-est}, \eqref{i-exp-lim}, and \eqref{last.term}, we may take $k\to\infty$ in \eqref{iteration} to obtain
\begin{equation}\label{iteration2}
		\norm{u}_{L^{\infty}(Q_{\infty})} \leq C \left[\frac{1}{T_{\infty}-T_{0}}+\frac{(1+ R_{0})^\gamma}{(R_{0}-R_{\infty})^2}\norm{u}_{L^\infty(Q_0)}^{m-1}\right]^{\frac{1+\theta}{\theta p_0+m-1}} \norm{u}_{L_\rho^{p_0}(Q_0)}^{\frac{\theta p_0}{\theta p_0+m-1}} .
\end{equation}
Estimate \eqref{iteration2} is still incomplete, due to the presence of an $L^\infty$-term on its right-hand side. To this end, we perform a further iteration of \eqref{iteration2} on relabeled cylinders $\{Q_k\}$, that are now \emph{increasing} with respect to $k \in \mathbb{N}$:
\begin{equation}\label{inc-cyl}
\begin{gathered}
	Q_k=B_{R_k}\times(T_k,T^*)\,,\\
	T_0>T_1>\cdots>T_k>T_{k+1}>\cdots\;\;\text{and}\;\;T_k\searrow T_\infty\in(0,T^*)\,,\\
	\;\;1 \le R_0<R_1<\cdots < R_{k}<R_{k+1}<\cdots\;\;\text{and}\;\;R_k\nearrow R_\infty\, ,
\end{gathered}
\end{equation}
where, of course, we take for granted that \eqref{close-cond} holds with $ R_0 \equiv R_\infty $ and $ R_1 \equiv R_0 $. That is, we write:
\begin{equation}\label{iteration3}
		\norm{u}_{L^{\infty}(Q_{k})} \leq C \left[ \frac{1}{T_{k}-T_{k+1}}+\frac{(1+ R_{k+1})^\gamma}{(R_{k+1}-R_{k})^2}\norm{u}_{L^\infty(Q_{k+1})}^{m-1}\right]^{\frac{1+\theta}{\theta p_0+m-1}} \norm{u}_{L_\rho^{p_0}(Q_{k+1})}^{\frac{\theta p_0}{\theta p_0+m-1}} .
\end{equation}
As before, we let $T_k$ and $R_k$ satisfy \eqref{radii.times.upper.iteration} with the role of $k$ and $k+1$, of $T_0$ and $T_\infty$, and of $R_0$ and $R_\infty$, interchanged. Let us inspect the term between parentheses
\[
\frac{1}{T_{k}-T_{k+1}}+\frac{(1+ R_{k+1})^\gamma}{(R_{k+1}-R_{k})^2}\norm{u}_{L^\infty(Q_{k+1})}^{m-1} .
\]
Now, suppose that
\begin{equation}\label{linfty contr mod 2}
	\frac{(1+ R_0)^\gamma}{(R_{\infty}-R_{0})^2}\norm{u}_{L^\infty(Q_{0})}^{m-1} \ge  \frac{\kappa}{T_{0}-T_{\infty}}
\end{equation}
for some numerical constant $  \kappa>0 $. Owing to \eqref{radii.times.upper.iteration}, it is \emph{a fortiori} true that
\begin{equation}\label{linfty contr mod k}
\begin{aligned}
	\frac{(1+ R_{k+1})^\gamma}{(R_{k+1}-R_{k})^2}\norm{u}_{L^\infty(Q_{k+1})}^{m-1} \ge \frac{(1+ R_{k})^\gamma}{(R_{k+1}-R_{k})^2}\norm{u}_{L^\infty(Q_{k})}^{m-1} \ge & \,  \frac{(1+ R_{0})^\gamma}{(R_{k+1}-R_{k})^2}\norm{u}_{L^\infty(Q_{0})}^{m-1} \\
    \ge & \,  \frac{ \frac{\mathsf{C}_2}{\mathsf{C}_1^2} \, \kappa }{T_{k}-T_{k+1}}  \, .
    \end{aligned}
\end{equation}
Let us continue the iteration under the assumption \eqref{linfty contr mod 2}, so \eqref{iteration3} simplifies to
\begin{equation}\label{iteration4}
		\norm{u}_{L^{\infty}(Q_{k})} \leq C \left[\frac{(1+ R_{k+1})^\gamma}{(R_{k+1}-R_{k})^2}\right]^{\frac{1+\theta}{\theta p_0+m-1}} \norm{u}_{L^\infty(Q_{k+1})}^{\frac{(1+\theta)(m-1)}{\theta p_0+m-1}} \norm{u}_{L_\rho^{p_0}(Q_{k+1})}^{\frac{\theta p_0}{\theta p_0+m-1} } .
\end{equation}
Upon calling
$$
\alpha = \tfrac{(1+\theta)(m-1)}{\theta p_0+m-1} \qquad \text{and} \qquad
\beta = \tfrac{\theta p_0}{\theta p_0+m-1} \, ,
$$
iterating \eqref{iteration4} $k-1$ times, and again using  \eqref{radii.times.upper.iteration} keeping in mind the above interchanges, we derive
\begin{equation}\label{iteration5}
	\begin{aligned}
		& \norm{u}_{L^{\infty}(Q_{0})} \\
        \leq & \, C^{\sum_{i=0}^{k-1}\alpha^i} \left[\prod_{i=0}^{k-1}(i+1)^{\frac{4\alpha}{m-1}\, \alpha^i}\right]\left[ \frac{(1+ R_{\infty})^\gamma}{(R_{\infty}-R_{0})^2}\right]^{\frac{\alpha}{m-1}\sum_{i=0}^{k-1}\alpha^i}\norm{u}_{L^\infty(Q_{k})}^{\alpha^k} \norm{u}_{L_\rho^{p_0}(Q_{k})}^{\beta\sum_{i=0}^{k-1}\alpha^i}.
	\end{aligned}
\end{equation}
Under the additional requirement that $ p_0>m-1 $, which implies
$\alpha<1$, it is readily seen that \eqref{iteration5} is stable as $k \to \infty$ (in particular the infinite product can be handled similarly to \eqref{prod-prod}), yielding
\begin{equation}\label{iteration6}
		\norm{u}_{L^{\infty}(Q_{0})} \leq C \left[ \frac{(1+ R_{\infty})^\gamma}{(R_{\infty}-R_{0})^2}\right]^{\frac{1+\theta}{\theta[p_0-(m-1)]}} \norm{u}_{L_\rho^{p_0}(Q_{\infty})}^{\frac{p_0}{p_0-(m-1)}} .
\end{equation}
Summing up, we have shown that \eqref{iteration6} holds under \eqref{linfty contr mod 2}, the constant $ C $ also depending on $\kappa$.

\smallskip

\noindent\textbf{Step 3a: Conclusion for $\boldsymbol{m<2}$.}
Our goal is now to interpolate in turn \eqref{iteration6} on increasing cylinders as in \eqref{inc-cyl}. We have the following dichotomy: either
\begin{equation}\label{linfty contr mod bis}
	\frac{(1+ R_0)^\gamma}{(R_{\infty}-R_{0})^2}\norm{u}_{L^\infty(Q_{0})}^{m-1}< \frac{1}{T_{0}-T_{\infty}} \, ,
\end{equation}
or
\begin{equation}\label{linfty contr mod 2 bis}
	\frac{(1+ R_0)^\gamma}{(R_{\infty}-R_{0})^2}\norm{u}_{L^\infty(Q_{0})}^{m-1} \ge  \frac{1}{T_{0}-T_{\infty}} \, .
\end{equation}
In the second case, we observed above that \eqref{linfty contr mod k} follows (with $\kappa = 1$), which is nothing but \eqref{linfty contr mod 2} itself with $ R_0, R_\infty, T_0, T_\infty, Q_0 $ replaced by $ R_{k}, R_{k+1}, T_{k}, T_{k+1}, Q_{k} $, respectively, and $ \kappa = \mathsf{C}_2/\mathsf{C}_1^2 $. As a result, we can infer that the analogue of \eqref{iteration6} between $ Q_{k} $ and $ Q_{k+1} $ holds for all $k \in \mathbb{N}$, with $C$ also depending on $ \mathsf{C}_1,\mathsf{C}_2 $, which are in fact purely numerical constants. Interpolating the $L_\rho^{p_0}$ norm on the right-hand side between $L_\rho^1$ and $L^\infty$ gives
\begin{equation*}\label{iteration7}
	\begin{aligned}
		\norm{u}_{L^{\infty}(Q_{k})} & \leq C\left[\frac{(1+ R_{k+1})^\gamma}{(R_{k+1}-R_{k})^2}\right]^{\frac{1+\theta}{\theta[p_0-(m-1)]}} \norm{u}_{L_\rho^{p_0}(Q_{k+1})}^{\frac{p_0}{p_0-(m-1)}} \\
		& \leq C\left[\frac{(1+ R_{k+1})^\gamma}{(R_{k+1}-R_{k})^2}\right]^{\frac{1+\theta}{\theta[p_0-(m-1)]}}\norm{u}_{L^\infty(Q_{k+1})}^{\frac{p_0-1}{p_0-(m-1)}} \norm{u}_{L_\rho^{1}(Q_{k+1})}^{\frac{1}{p_0-(m-1)}} .
	\end{aligned}
\end{equation*}
We then iterate upon the $L^\infty$-norm in a similar scheme to the one leading to \eqref{iteration6}, obtaining
\begin{equation}\label{aaa}
	\norm{u}_{L^{\infty}(Q_{0})} \leq C \left[ \frac{(1+ R_{\infty})^\gamma}{(R_{\infty}-R_{0})^2}\right]^{\frac{1+\theta}{\theta(2-m)}} \norm{u}_{L_\rho^{1}(Q_{\infty})}^{\frac{1}{2-m}} .
\end{equation}
Note that such an iteration is only possible for $m<2$, since in that case the exponent $(p_0-1)/[p_0-(m-1)]$ is smaller than one. Recall that \eqref{aaa} has been proved under \eqref{linfty contr mod 2 bis}. On the other hand, the estimate
\begin{equation}\label{aaa-bis}
	\norm{u}_{L^{\infty}(Q_{0})} \leq C\left[\frac{(1+ R_{\infty})^\gamma}{(R_{\infty}-R_{0})^2}\right]^{\frac{1+\theta}{\theta(2-m)}} \norm{u}_{L_\rho^{1}(Q_{\infty})}^{\frac{1}{2-m}} + \left[ \frac{(R_{\infty}-R_{0})^2}{(1+R_\infty)^\gamma \, (T_{0}-T_{\infty})} \right]^{\frac{1}{m-1}}
\end{equation}
also encompasses the case where \eqref{linfty contr mod bis} holds instead.
Hence, the thesis (\emph{i.e.}~the first claim in \eqref{l.infty.balls}) follows upon choosing $R_\infty=2R$, $R_0=R $, $ T_0=t_2 $, and $ T_\infty = t_1 $ as in the statement.

\smallskip

\noindent\textbf{Step 3b: Conclusion for all $\boldsymbol{m>1}$.} To obtain the second claim in \eqref{l.infty.balls}, we need to go back to \eqref{comb2} for a fixed $p$ and iterate a suitable interpolation inequality, and we again consider an increasing sequence of cylinders $ {Q_k} $ as in \eqref{inc-cyl}. Also, let us fix {$p^*>1$} and set $a$ and $p_0=a(p^*+m-1)$ to satisfy
\begin{equation*}\label{choicea2}
	(p^*+m-1)(a-1)\frac{1}{\theta}=1 \, .
\end{equation*}
In this way, $p_0=p^*+m-1+\theta$. Keeping these definitions in mind, assume again that \eqref{linfty contr mod 2} holds, hence also \eqref{linfty contr mod k} for all $ k \in \mathbb{N} $, and consider \eqref{comb2} with $R_0 \equiv R_{k+1}$, $R_1 \equiv R_k$, $T_1 \equiv T_k$, $T_0 \equiv T_{k+1}$, to obtain
\begin{equation}\label{comb2'}
\begin{aligned}
		\iint_{Q_{k}} |u|^{p_0}\, \rho \, dxdt \leq & \, C \, \frac{(1+ R_{k+1})^\gamma}{(R_{k+1}-R_{k})^2} \norm{u}_{L^\infty(Q_{k+1})}^{m-1} \iint_{Q_{k+1}} |u|^{p_0-(m-1)-\theta} \, \rho \, dxdt \\
		& \times {\left(\sup _{t \in \left (T_{k},T^* \right ) }\int _{B_{R_{k}}} |u(t)|\,\rho \,dx \right)}^{\theta} .
	\end{aligned}
\end{equation}
Under \eqref{linfty contr mod 2}, we have already observed in Step 3a that the same assumption is satisfied in all of the nested cylinders up to changing the multiplying constant; we are thus in a position to apply \eqref{iteration6} between $ Q_{k} $ and $ Q_{k+1} $, so that \eqref{comb2'} becomes
\begin{equation*}\label{comb3''}\begin{aligned}
		\iint_{Q_{k}} |u|^{p_0} \, \rho\, dxdt \leq & \,  C \left[\frac{(1+ R_{k+2})^\gamma}{(R_{k+2}-R_{k+1})^2}\right]^{\frac{\theta p_0+m-1}{\theta[p_0-(m-1)]}}\norm{u}_{L_\rho^{p_0}(Q_{k+2})}^{\frac{p_0(m-1)}{p_0-(m-1)}}
		\\
        &\times \iint_{Q_{k+2}} |u|^{p_0-(m-1)-\theta} \, \rho \, dxdt \, \left(\sup _{t \in \left (T_{k},T^* \right ) }\int _{B_{R_{k}}} |u(t)| \,\rho \,dx \right)^{\theta},
	\end{aligned}
\end{equation*}
where we have simply noted that $ Q_{k+1} \subset Q_{k+2} $ and $ R_{k+2}-R_{k+1} < R_{k+1}-R_{k} $.
Interpolating the $L_\rho^{p_0-(m-1)-\theta}$ norm on the right-hand side between $L_\rho^1$ and $L_\rho^{p_0}$, and taking appropriate roots, we find
\begin{equation}\label{comb3'}
		\norm{u}_{L_\rho^{p_0}(Q_{k})} \leq C\,\mathsf{S}^{\frac{\theta}{p_0}} \left[\frac{(1+ R_{k+2})^\gamma}{(R_{k+2}-R_{k+1})^2}\right]^{\frac{\theta p_0+m-1}{\theta p_0[p_0-(m-1)]}} \norm{u}_{L_\rho^{p_0}(Q_{k+2})}^{\frac{m-1}{p_0-(m-1)}+\frac{p_0-m-\theta}{p_0-1}}\norm{u}_{L^1_\rho(Q_{k+2})}^{\frac{\theta+m-1}{p_0(p_0-1)}} ,
\end{equation}
where $ \mathsf{S} =\sup _{t \in \left (T_{\infty},T^* \right ) }\int _{B_{R_{\infty}}} |u(t)|\, \rho \,dx$. In order to iterate \eqref{comb3'}, we need that
\begin{equation}\label{condp}
\omega=\tfrac{m-1}{p_0-(m-1)}+\tfrac{p_0-m-\theta}{p_0-1}<1	\qquad \Longleftrightarrow \qquad  p_0>(m-1)\, \tfrac{\theta+m-2}{\theta} \, ,
\end{equation}
a condition that we assume hereafter (it is achieved by choosing $p^*$ large enough). Recalling \eqref{radii.times.upper.iteration} and noticing that
$$
\tfrac{1}{1-\omega}=\tfrac{(p_0-1)[p_0-(m-1)]}{\theta p_0-(m-1)(\theta+m-2)} \, ,
$$
we can indeed iterate \eqref{comb3'} in the by now usual way, obtaining
\begin{equation}\label{final p1 smoothing}
	\begin{aligned}
		\norm{u}_{L_\rho^{p_0}(Q_{0})} \leq & \, C\,\mathsf{S}^{\frac{\theta}{p_0}\frac{1}{1-\omega}} \left[\frac{(1+ R_{\infty})^\gamma}{(R_{\infty}-R_{0})^2}\right]^{\frac{\theta p_0+m-1}{\theta p_0[p_0-(m-1)]}\frac{1}{1-\omega}} \norm{u}_{L^1_\rho(Q_{\infty})}^{\frac{\theta + m-1}{p_0(p_0-1)}\frac{1}{1-\omega}}\\
		= & \, C\, \mathsf{S}^{\frac{\theta(p_0-1)[p_0-(m-1)]}{p_0[\theta p_0-(m-1)(\theta+m-2)]}} \left[\frac{(1+ R_{\infty})^\gamma}{(R_{\infty}-R_{0})^2}\right]^{\frac{(p_0-1)(\theta p_0+m-1)}{\theta p_0 [\theta p_0-(m-1)(\theta+m-2)]}}\\
		& \times\norm{u}_{L^1_\rho(Q_{\infty})}^{\frac{(\theta+m-1)[p_0-(m-1)]}{p_0[\theta p_0-(m-1)(\theta+m-2)]}}\\
		\leq & \, C\,\mathsf{S}^{\frac{(\theta p_0+m-1)[p_0-(m-1)]}{p_0[\theta p_0-(m-1)(\theta+m-2)]}} \left[\frac{(1+ R_{\infty})^\gamma}{(R_{\infty}-R_{0})^2}\right]^{\frac{(p_0-1)(\theta p_0+m-1)}{\theta p_0 [\theta p_0-(m-1)(\theta+m-2)]}} \\
		& \times(T^*-T_\infty)^{\frac{(\theta+m-1)[p_0-(m-1)]}{p_0[\theta p_0-(m-1)(\theta+m-2)]}} \, ,
	\end{aligned}
\end{equation}
where in the last inequality we have used $\int_a^b|f(t)| \, dt\leq(b-a)\sup_{t\in(a,b)}|f(t)|$.

Next we observe that, up to different multiplying constants, estimate \eqref{iteration6} also holds between $ Q_0 $ and $Q_1$ (\emph{i.e.}~with $ R_\infty \equiv R_1 $ and $ Q_\infty \equiv Q_1 $), whereas estimate \eqref{final p1 smoothing} is still true between $ Q_1 $ and $ Q_\infty $ (\emph{i.e.}~with $ R_0 \equiv R_1 $ and $ Q_0 \equiv Q_1 $), so combining them yields
\begin{equation}\label{sinfsmoothing}
	\begin{aligned}
		\norm{u}_{L^{\infty}(Q_{0})} \leq C\, \mathsf{S}^{\frac{\theta p_0+m-1}{\theta p_0-(m-1)(\theta+m-2)}} \left[\frac{(1+ R_{\infty})^\gamma}{(R_{\infty}-R_{0})^2}\right]^{\frac{p_0+\theta+m-1+\frac{m-1}{\theta}}{\theta p_0-(m-1)(\theta+m-2)}}(T^*-T_\infty)^{\frac{\theta+m-1}{\theta p_0-(m-1)(\theta+m-2)}} \, ,
	\end{aligned}
\end{equation}
where we have also exploited the fact that $ (R_1-R_0) $ and $ (R_\infty-R_1) $ are comparable to $ (R_\infty-R_0) $, up to numerical constants (still a consequence of \eqref{radii.times.upper.iteration}). Note that, under the essential condition \eqref{condp} on $p_0$, all of the exponents in \eqref{sinfsmoothing} are formally minimized by taking $p_0\to\infty$. However, this is not possible, because it can be shown, by keeping track of all of the dependencies on $p_0$, that the multiplying constant in \eqref{final p1 smoothing} actually blows up as $ p_0 \to \infty $. Therefore, in order to simplify the appearance of such exponents, we define an arbitrarily small number $\varepsilon > 0$ so that
\begin{equation*}\label{defeps}
	1+\varepsilon=\tfrac{\theta p_0+m-1}{\theta p_0-(m-1)(\theta+m-2)} \, .
\end{equation*}
Under this choice, estimate \eqref{sinfsmoothing} simplifies to
\begin{equation}\label{sinfsmoothing1}
	\begin{aligned}
		\norm{u}_{L^{\infty}(Q_{0})} \leq C \, \mathsf{S}^{1+\varepsilon}\left[ \frac{(1+ R_{\infty})^\gamma}{(R_{\infty}-R_{0})^2}\right ]^{\frac{\varepsilon}{m-1}+\frac{1+\varepsilon}{\theta}} (T^*-T_\infty)^{\frac{\varepsilon}{m-1}} \, ,
	\end{aligned}
\end{equation}
where $C$ now depends only on $N,m,\gamma,\underline{C},\overline{C}$ and on $\varepsilon$, being unstable as $\varepsilon\to0$. Recalling that $ p^\ast>1 $, such a constraint translates into $ \varepsilon<\varepsilon_0 $ for some constant $ \varepsilon_0>0 $ depending on $ m,\theta$, which can in fact be taken equal to $ +\infty $ if $m$ is sufficiently large.

In order to finish, we recall that \eqref{sinfsmoothing1} holds under \eqref{linfty contr mod 2}, but as in Step 3a the estimate
\begin{equation}\label{aaa-bis-gen}
    \norm{u}_{L^{\infty}(Q_{0})} \leq  C\,\mathsf{S}^{1+\varepsilon}\left[\frac{(1+ R_{\infty})^\gamma}{(R_{\infty}-R_{0})^2}\right]^{\frac{\varepsilon}{m-1}+\frac{1+\varepsilon}{\theta}} (T^*-T_\infty)^{\frac{\varepsilon}{m-1}} + \left[ \frac{(R_{\infty}-R_{0})^2}{(1+R_\infty)^\gamma\,(T_{0}-T_{\infty})} \right]^{\frac{1}{m-1}}
\end{equation}
also includes the case in which \eqref{linfty contr mod 2} fails, whence the thesis (\emph{i.e.}~the second claim in \eqref{l.infty.balls}) follows with the same choice of radii and times specified after \eqref{aaa-bis}.
\end{proof}

In order to complete the proof of Theorem \ref{local moser iter lem}, we perform a local approximation of very weak solutions by means of strong energy solutions, passing to the limit in the just proved estimate \eqref{l.infty.balls}. Before, we need some basic compactness estimates (which will also be useful in Section \ref{sec:unique}).

\begin{lem}[Local energy estimates II]\label{energy lem}
Let $ \Omega \subset \R^N $ be (possibly unbounded) domain and $ 0 \le \tau_1<\tau_2 $. Let $u$ be a strong energy solution to~\eqref{wpme-noid-loc}, in the sense of Definition \ref{def:w-loc}. Let us consider the cylinders $ Q_0,Q_1 $ as in \eqref{cyl-energy}, under the same inclusion assumptions. Then,
\begin{equation}\label{h1 energy}
\iint_{Q_1} \left|\nabla\!\left( u^m\right) \right|^2 dxdt \leq C \left( \mathsf{diff}_t^{-1} \vee \mathsf{diff}_r^{-2} \right) \left( 1 +  R_0^\gamma \left\|  u\right\|^{m-1}_{L^\infty(Q_0)} \right) \left\| u \right\|_{L^{m+1}_\rho(Q_0)}^{m+1}
\end{equation}
and
\begin{equation}\label{time energy}
\iint_{Q_1} \left| \partial_t(u^m) \right|^2 dxdt \leq C \left( \mathsf{diff}_t^{-1} \vee \mathsf{diff}_r^{-2} \right)^2 \left( 1 +  R_0^\gamma \left\|  u\right\|^{m-1}_{L^\infty(Q_0)} \right)^2 \, \left\|  u\right\|^{m-1}_{L^\infty(Q_0)} \left\| u \right\|_{L^{m+1}_\rho(Q_0)}^{m+1} ,
\end{equation}
where we set $ \mathsf{diff}_t = T_1-T_0 $ and $ \mathsf{diff}_r = R_0-R_1 $, and $C>0$ is a constant depending only on $ m,\gamma,\underline{C} $.
\end{lem}
\begin{proof}
First of all, we notice that estimate \eqref{h1 energy} is a direct consequence of \eqref{wanted.inq} with $ p=m+1 $, recalling that we are assuming $ R_0>1 $. Hence, let us concentrate on \eqref{time energy}.

The computations that follow are only formal, since strong energy solutions do not have enough regularity to justify them rigorously; however, we will recall at the end of the proof the by-now standard methods that one can use to overcome these technical obstructions.

As a preliminary observation, we have
    \begin{equation}\label{prelim-tdir}
        \left|\partial_t(u^m)\right|= \left(\tfrac{2m}{m+1}\right) |u|^{\frac{m-1}{2}} \left|\partial_t\!\left(u^{\frac{m+1}{2}} \right)\right| ,
    \end{equation}
hence by the local boundedness assumption on $u$ it is enough to estimate the $L^2$ norm of the time derivative of $ u^{(m+1)/2} $. To this end, let $ \varphi $ be the same cutoff function as in \eqref{test.function.balls-prop}--\eqref{test.function.balls}, with $ Q_0 $ replaced by the ``intermediate'' cylinder $ \widehat{Q}_0 $ between $Q_1$ and $ Q_0 $, namely the one with radius and initial time equal to
$$
\widehat{R}_0 = \tfrac{R_0+R_1}{2} \qquad \text{and} \qquad \widehat{T}_0 = \tfrac{T_0+T_1}{2} \, ,
$$
respectively. Note that, with this choice,
\begin{equation}\label{diff-hat}
\widehat{R}_0-R_1 = R_0 - \widehat{R}_0 = \tfrac{R_0-R_1}{2} \qquad \text{and} \qquad T_1-\widehat{T}_0 = \widehat{T}_0 - T_0 =\tfrac{T_1-T_0}{2} \, .
\end{equation}
Testing \eqref{wpme-noid-loc} with $ \partial_t(u^m) \varphi^2 $, and integrating by parts on the right-hand side, yields
    \begin{equation}\label{tdir1}
        \begin{aligned}
            & \quad\ \iint_{\widehat{Q}_0}\left|\partial_t\!\left(u^{\frac{m+1}{2}}\right)\right|^2 \varphi^2 \,\rho \, dxdt \\
            & = c_m \, \iint_{\widehat{Q}_0} \Delta(u^m)\, \partial_t(u^{m}) \,\varphi^2 \, dxdt \\
            &=-c_m \, \iint_{\widehat{Q}_0} \left\langle \nabla (u^m) \, ,\nabla\!\left(\partial_t (u^m)\,\varphi^2\right)\right\rangle dx dt\\
            &=-\frac{c_m}{2} \, \underbrace{\iint_{\widehat{Q}_0} \partial_t\!\left(\left|\nabla (u^m)\right|^2\right)\varphi^2 \, dxdt}_{I}-c_m \, \underbrace{\iint_{\widehat{Q}_0}\left\langle \nabla (u^m) \, ,\nabla\!\left(\varphi^2\right) \right\rangle \partial_t (u^m) \, dxdt}_{II}\,,
        \end{aligned}
    \end{equation}
for a constant $c_m>0$ depending only on $m$. As concerns $I$, we have the estimate
    \begin{equation}\label{tdir2}
    \begin{aligned}
        I & = \int_{B_{\widehat{R}_0}} \left| \nabla(u^m)(T^*) \right|^2 \varphi^2(T^*) \, dx - \iint_{\widehat{Q}_0} \left|\nabla (u^m)\right|^2  \partial_t\!\left(\varphi^2\right) dxdt \\
        & \ge  - \iint_{\widehat{Q}_0} \left|\nabla (u^m)\right|^2  \left|\partial_t\!\left(\varphi^2\right)\right| dxdt \, .
        \end{aligned}
    \end{equation}
    As for $II$ on the other hand, by Young's inequality, \eqref{prelim-tdir}, and \eqref{weight-cond}, we derive
    \begin{equation}\label{tdir3}
        |II|\leq \mathsf{C} \, \varepsilon\,{\widehat{R}_0}^{\gamma}\norm{u}^{m-1}_{L^\infty \left(\widehat{Q}_0\right)} \iint_{\widehat{Q}_0} \left|\partial_t \!\left(u^{\frac{m+1}{2}}\right)\right|^2 \varphi^2 \, \rho \, dx dt+\frac{1}{\varepsilon}\!\iint_{\widehat{Q}_0} \left|\nabla(u^m)\right|^2 \left|\nabla\varphi\right|^2 dxdt \,,
    \end{equation}
 for all $\varepsilon>0$ and a constant $\mathsf{C}>0$ depending only on $m,\gamma,\underline{C}$. Combining \eqref{tdir1}, \eqref{tdir2}, and \eqref{tdir3} with the choice of $\varepsilon$ such that
    \begin{equation*}
c_m \, \mathsf{C} \, \varepsilon\,{\widehat{R}_0}^{\gamma}\norm{u}^{m-1}_{L^\infty \left(\widehat{Q}_0\right)} =\frac{1}{2}\,,
    \end{equation*}
    yields
    \begin{equation}\label{dt-last}
    \begin{aligned}
        & \iint_{\widehat{Q}_0}\left|\partial_t\!\left(u^{\frac{m+1}{2}}\right)\right|^2 \varphi^2 \,\rho \, dxdt \\
        \leq & \, C\left(1+{\widehat{R}_0}^{\gamma}\norm{u}^{m-1}_{L^\infty\left( \widehat{Q}_0 \right)} \right)\iint_{\widehat{Q}_0} \left|\nabla (u^m)\right|^2\left(\left|\nabla\varphi\right|^2 \vee \left|\partial_t\!\left(\varphi^2\right)\right|\right) dx dt \, ,
    \end{aligned}
    \end{equation}
    for a generic constant $C>0$ as in the statement. Hence, estimate \eqref{time energy} is a now a direct consequence of \eqref{dt-last}, the support properties of $ \varphi $, \eqref{h1 energy} applied with $ Q_1 \equiv \widehat{Q}_0 $, and \eqref{diff-hat}.


    Finally, it is a standard fact that strong energy solutions are uniquely determined by their (local) initial and boundary data, see \emph{e.g.}~\cite[Theorem 5.13]{Vazquez} (in the unweighted case -- also known as \emph{weak energy solutions} to Cauchy--Dirichlet problems like \eqref{cd problem}), hence one can always construct them as a limit of regular approximate solutions for which the above computations are justified, and estimates \eqref{h1 energy}--\eqref{time energy} are stable at the limit. Some more details about such a construction will be provided in the incoming end of proof of Theorem \ref{local moser iter lem}, where the much more delicate very weak solutions are dealt with.
\end{proof}

\begin{proof}[End of proof of Theorem \ref{local moser iter lem}]
We have just shown that the smoothing estimate \eqref{l.infty.balls} holds under the additional requirement that $u$ is a strong energy solution. On the other hand, by virtue of Lemma \ref{globloc} and Remark \ref{loc-vw-bdry},
we can assume without loss of generality that our local very weak solution $ u $ is in fact a (bounded) solution of the Cauchy--Dirichlet problem \eqref{cd problem}, in the sense of Definition \ref{vw cd def}, for some $ u_0 \in L^\infty(B_r) $ and $ g \in L^\infty_\sigma(\partial B_r \times (\tau_1,\tau_2)) $, for a.e.~$ r>0 $ such that $ B_{r} \Subset \Omega$.

Our goal is now to approximate $u$ with a sequence $ \{ u_n \} $ of \emph{regular} solutions of a slightly perturbed Cauchy--Dirichlet problem, to which the above proved estimate \eqref{l.infty.balls} is applicable, and then carefully pass to the limit as $n\to \infty$. In order to avoid an excess of technicalities, from now on we will assume in addition that $ u $ is \emph{non-negative}, providing at the end of the proof a brief explanation on how to treat sign-changing solutions as well. So, for each $ n \in \mathbb{N} \setminus\{0\} $, let $ u_n $ be the (non-negative) solution of
\begin{equation}\label{cd problem-bis}
	\begin{cases}
		\rho_n \, \partial_t  u_n = \Delta\!\left(u_n^m \right) & \text{in} \;\; B_{r} \times(\tau_1,\tau_2) \, , \\
		u^m_n = g_n & \text{on} \;\; \partial  B_{r} \times (\tau_1,\tau_2) \, , \\
		u_{n}  = u_{0,n} & \text{on} \;\; B_{r} \times \{ \tau_1 \} \, ,
	\end{cases}
\end{equation}
where $r$ is chosen as above with $ r>2R $. The approximate data $ u_{0,n} $, $ g_n $, and $ \rho_n $ are chosen so as to satisfy the following properties:
\begin{equation}\label{approx-data-loc}
\begin{gathered}
	u_{0,n}\in C^\infty\!\left(\overline{B}_{r}\right), \qquad u_{0,n} \overset{*}{\rightharpoonup} u_0 \quad \text{in} \;\; L^\infty(B_{r}) \, , \qquad \tfrac 1 n \leq u_{0,n}\leq M \,, \\
	g_n\in C^\infty(\partial B_{r}\times[\tau_1,\tau_2])\,, \qquad g_n \overset{*}{\rightharpoonup} g   \quad \text{in}\;\; L^\infty(\partial B_{r}\times(\tau_1,\tau_2))\,, \qquad \tfrac {1}{n^m} \leq g_n \leq M^m \,, \\
	\rho_n \in C^\infty\!\left(\overline{B}_{r}\right) , \quad \tfrac12\,\underline{C} \left( 1 + |x| \right)^{-\gamma}  \le \rho_n(x) \le 2 \, \overline{C} \, |x| ^{-\gamma} \;\;\;\; \forall x \in B_r \, , \quad \rho_n \to \rho \;\;\;\;  \text{a.e.~in} \;\; B_{r} \, ,
\end{gathered}
\end{equation}
the above limits being taken as $ n \to \infty $, and for a suitable $M>0$ independent of $n$. Also, we tacitly assume that the initial and boundary data smoothly match at the corner $\partial B_{r} \times \{ \tau_1 \}$, so that $ u_n $ is indeed smooth up to the parabolic boundary from classical results for quasilinear equations, see for example \cite[Chapter 7]{Lieberman}.

Hence, we have the right to apply \eqref{l.infty.balls} to $u_n$, which gives
		\begin{equation}\label{l.infty.balls-n}
			\begin{aligned}
				R^{-\frac{2-\gamma}{m-1}} \norm{u_n}_{L^\infty\left(\underline{Q}\right)} & \leq C_1 \left[  \left(R^{-(N-\gamma)-\frac{2-\gamma}{m-1}}\norm{u_n}_{L_\rho^{1}\left(\overline{Q}\right)}\right)^{\frac{1}{2-m}}+\left(\tfrac{1}{t_2-t_1}\right)^{\frac{1}{m-1}} \right] \quad \text{for}\;\; m\in(1,2) \, ,\\
				R^{-\frac{2-\gamma}{m-1}} \norm{u_n}_{L^\infty\left(\underline{Q}\right)} & \leq  C_2\left[\left(R^{-(N-\gamma)-\frac{2-\gamma}{m-1}}\,\mathsf{S}_n\right)^{1+\varepsilon}\left(T^*-t_1\right)^{\frac{\varepsilon}{m-1}}+\left(\tfrac{1}{t_2-t_1}\right)^{\frac{1}{m-1}}\right] \quad
				 \text{for}\;\; m>1 \, ,
			\end{aligned}
		\end{equation}
		with
        $$
        \mathsf{S}_n=\sup\limits_{t \in (t_1,T^*) } \int _{B_{2R}} \left|u_n(x,t)\right| \rho(x)\,dx \, ,
        $$
the constants $ C_1,C_2 $ being of the same type as in the statement (in particular independent of $n$). Next, we need to pass to the limit in \eqref{l.infty.balls-n} as $n \to \infty$. To this end, we first observe that by standard comparison principles we have
\begin{equation}\label{loc-unif-est}
	\tfrac{1}{n} \le  u_n \leq M \qquad \text{in}\;\; B_{r} \times (\tau_1,\tau_2) \, .
\end{equation}
In particular, Lemma \ref{energy lem} applied to $ u \equiv u_n $ ensures that
$$
\left\{ u_n^m \right\} \; \; \text{is bounded in} \; \; H^1 \!\left((\hat \tau_1,\hat\tau_2); L^2_\rho\!\left(B_{r_0}\right)\right) \cap L^2\!\left((\hat \tau_1,\hat \tau_2); H^1\!\left(B_{r_0}\right)\right)
$$
for all $ 0<r_0 < r $ and all $ \tau_1 < \hat{\tau}_1 < \hat{\tau}_2 < \tau_2 $. We are therefore in a position to invoke the Aubin--Lions compactness lemma, which yields the existence of an element $ u_\star $ such that
\begin{equation}\label{types of convergence smoothing}
	\begin{aligned}
		u_{n}^m \rightharpoonup u_\star^m \qquad & \text{in}\;\; L^2\!\left((\hat \tau_1,\hat \tau_2);H^1\!\left(B_{r_0}\right)\right) ,  \\
        \partial_t\!\left( u_{n}^m \right) \rightharpoonup \partial_t\!\left( u_\star^m \right) \qquad & \text{in}\;\; L^2\!\left((\hat \tau_1,\hat \tau_2);L^2_{\rho}\!\left( B_{r_0} \right)\right) ,
	\end{aligned}
\end{equation}
as $ n \to \infty $, along a subsequence that we do not relabel. It is plain that, thanks to \eqref{loc-unif-est} and a routine diagonal argument that implies selecting a further subsequence, we can also assume that $ u_\star \in L^\infty\!\left( B_{r} \times (\tau_1,\tau_2) \right) $ and that
\begin{equation}\label{types of convergence smoothing bis}
u_n \xrightharpoonup[n \to \infty]{\ast}  u_\star \qquad \text{in} \;\; L^\infty\!\left( B_{r} \times (\tau_1,\tau_2) \right) .
\end{equation}
The local very weak formulation satisfied by each $ u_n $ reads (for all test function $\psi$ as in Definition \ref{vw cd def})
		\begin{equation}\label{vw cd eq n version}
			\int_{\tau_1}^{\tau_2} \int_{B_r} \left(u_n \, \psi_t \, \rho_n + u_n^m \, \Delta\psi \right) dx dt +\int_\Omega u_{0,n}(x)\,\psi(x,\tau_1)\,\rho_n(x)\,dx= \int_{\tau_1}^{\tau_2} \int_{\partial B_r} g_n \, \partial_{\Vec{\mathsf n}}\psi \, d\sigma dt \, ,
		\end{equation}
and from \eqref{approx-data-loc}, \eqref{types of convergence smoothing}, \eqref{types of convergence smoothing bis}, it is safe to pass to the limit in \eqref{vw cd eq n version}, finding that $ u_\star $ is also a bounded very weak solution of the Cauchy--Dirichlet problem \eqref{cd problem-bis}. Hence, owing to Proposition \ref{local comparison lemma}, we necessarily have $ u=u_\star $. Finally, the validity of \eqref{l.infty.balls} is just a consequence of passing to the limit in \eqref{l.infty.balls-n} as $ n \to \infty $, the estimate being stable thanks to the convergence properties \eqref{types of convergence smoothing}--\eqref{types of convergence smoothing bis}. In particular, we observe that \eqref{types of convergence smoothing} ensures that $ \{ u_n \} $ converges to $ u_\star $ in $ C\!\left([t_1,T^*] ; L^1_\rho(B_{2R}) \right) $, so that $ \mathsf{S}_n \to \mathsf{S} $.

If $ u $ is a sign-changing solution, we cannot approximate it as in \eqref{cd problem-bis}--\eqref{approx-data-loc}, as lifting by $ 1/n $ the initial and boundary data is not enough to remove the degeneracy of the power $ u \mapsto u^m $ near $u=0$. In this case, it is more appropriate to approximate the degenerate power itself, that is, replace $ u^m_n $ with $ \Phi_n(u_n) $ in \eqref{cd problem-bis}, where $ \{ \Phi_n(u) \} $ is a suitable sequence of smooth and non-degenerate real functions converging to $u^m$; for the details of such a construction, we refer to \cite{GMPo} and the references therein. In particular, the energy estimates \eqref{h1 energy}--\eqref{time energy} can also be shown to hold for such approximate solutions (with suitable adaptations of the powers involved), by carrying out a ``localized'' version of \cite[Lemma 3.3]{GMPo}, with completely analogous computations. Having these modifications in mind, it is then not difficult to deduce, as in the last part of the proof, that $u$ is actually a strong energy solution in $ B_r \times (\tau_1,\tau_2) $, and therefore one can apply to it the firstly proved smoothing estimates \eqref{l.infty.balls} for this type of solutions.
\end{proof}

\begin{proof}[Proof of Corollary \ref{glob-unif-est}]
It is an immediate consequence of Theorems \ref{local moser iter lem} and \ref{apriori-bounded}.
\end{proof}

\begin{proof}[Proof of Corollary \ref{very weak strong energy}]
It readily follows from the end of proof of Theorem \ref{local moser iter lem}, where it is in particular shown that locally bounded very weak solutions can be approximated by strong energy solutions, and such approximations have a uniformly bounded local energy (see \eqref{loc-unif-est}--\eqref{types of convergence smoothing}). Rigorously, this is done in cylinders contained in $ \Omega \times (\tau_1,\tau_2) $ that are centered at the origin, however, also recalling Remark \ref{loc-vw-bdry}, there is no obstruction to carrying out the same construction in cylinders centered at an arbitrary $x_0 \in \Omega$ (energy estimates completely analogous to \eqref{h1 energy}--\eqref{time energy} hold).

If $u$ is in addition non-negative, then its \emph{a priori} local boundedness assumption can be dropped thanks to Theorem \ref{apriori-bounded}. Finally, property \eqref{Cont-Lp} follows directly from local boundedness in combination with \eqref{w-hyp-energ}.
\end{proof}

\section{Existence and uniqueness: proofs}\label{sec:unique}
The main goal of this section is to prove the well-posedness of problem \eqref{measure}, when $ \mu $ is a Radon measure with the growth rate identified in Theorem \ref{thm:weighted-AC}. We split it into three subsections: after a first one containing crucial preliminary results, in the second one we show existence (also for sign-changing measures), whereas in the third one we prove uniqueness for non-negative measures.

\subsection{Additional definitions and some preliminary results}\label{subsec:prelim-compact}
We start by introducing a suitable definition of \emph{weak} solution to \eqref{measure} that takes a non-negative \emph{finite} Radon measure as its initial datum.

\begin{den}[Weak energy solutions with finite measure data]\label{def3}
Let $\mu$ be a non-negative finite Radon measure. We say that a non-negative function $ u $ is a weak energy solution of problem \eqref{measure}, with $T=+\infty$, if
\begin{equation*}\label{energymeas}
	\begin{gathered}
		u \in  C\!\left((0,+\infty); L^1_\rho\!\left(\mathbb{R}^N\right) \right)\cap\, L^\infty\!\left( \mathbb{R}^N \times ( \tau , +\infty) \right) \qquad \forall\tau>0 \, , \\
		 u^m \in L^2_{\mathrm{loc}} \big((0,+\infty); \dot H^1\!\left(\mathbb{R}^N \right) \! \big) \, ,
	\end{gathered}
\end{equation*}
$\mu$ is the initial trace of $u$ in the sense of \eqref{measure-data}, and the weak formulation \eqref{energy-formulation} holds
for all $\psi\in C^\infty_c\!\left( \mathbb{R}^N\times (0, +\infty) \right)$.
\end{den}

The following result concerning the well-posedness of the problem \eqref{measure}, in the set of solutions provided by Definition \ref{def3}, is the local (pure Laplacian operator) counterpart of \cite[Theorems 3.2 and 3.4]{GMP} -- see the classical paper \cite{Pierre} in the case $\rho=1$.

\begin{thm}\label{weighted pierre thm}
Let $N\geq3$, $m>1$, and $\rho$ be a measurable function satisfying \eqref{weight-cond}. Let $\mu$ be a non-negative finite Radon measure. Then there exists a unique weak energy solution $u$ of problem \eqref{measure} (with $ T=+\infty $), in the sense of Definition \ref{def3}.
\end{thm}
\begin{proof}
As observed in the Introduction of \cite{GMP}, all of the results of that paper also work in the case of the pure Laplacian operator, which is the one we are interested in. However, it has to be noted that a local additional assumption was required on $ \rho $, specifically the existence of some $ \gamma_0 \in [0,\gamma] $ and constants $c_1,c_2>0$ such that
\begin{equation}\label{loc-rho}
c_1 \, |x|^{-\gamma_0} \le \rho(x) \le c_2 \, |x|^{-\gamma_0} \qquad \text{for a.e.}\;\; x \in \R^N \, .
\end{equation}
Nevertheless, the only point where \eqref{loc-rho} was crucially used is \cite[Theorem 3.7]{GMP} (a complete proof can be found in \cite{M}), which, within the present framework, amounts to the fact that the operator
$$
A[f] = - \rho^{-1} \, \Delta f \qquad \forall f \in D(A) \, ,
$$
with domain
$$
D(A) = \left\{ f \in L^2_\rho\!\left(\R^N\right) : \ \Delta f  \in L^2_{\rho^{-1}}\!\left(\R^N\right) \right\} ,
$$
densely defined, positive, and self-adjoint in $  L^2_\rho\!\left(\R^N\right) $, with quadratic form
$$
Q(f,f) = \int_{\R^N} \left| \nabla f \right|^2 dx \, .
$$
Our goal is to show that such properties still hold under \eqref{weight-cond}, regardless of \eqref{loc-rho}. Since clearly $ C^\infty_c\!\left(\R^N\right) \subset D(A) $, the claim follows if we can prove that $ D(A) \subset \dot{H}^1\!\left( \R^N \right) $ and that for all $ f,g \in D(A) $ it holds
\begin{equation}\label{self-adj-prop}
 \int_{\R^N} f \, A[g] \, \rho \, dx \,   = \int_{\R^N} \left\langle \nabla f  , \nabla g \right\rangle dx  = \int_{\R^N} A[f] \, g \, \rho \, dx\, .
\end{equation}
First of all, let us check that if $ f \in D(A) $ then $ f \in H^1_{\mathrm{loc}}\!\left( \R^N \right) $. Thanks to the upper bound on $ \rho $ in \eqref{weight-cond} and H\"older's inequality, it is readily seen that $ \Delta f \in L^p_{\mathrm{loc}}\!\left( \R^N \right) $ for all $ p $ such that
\begin{equation*}\label{ell-reg}
1 \le p <  \tfrac{2N}{N+\gamma} \, .
\end{equation*}
From classical elliptic regularity, see \emph{e.g.}~\cite[Theorem 9.11]{GT}, we have that $ f \in W^{2,p}_{\mathrm{loc}}\!\left( \R^N\right) $ for all such $p$, which entails $ f \in W^{1,q}_{\mathrm{loc}}\!\left( \R^N\right)  $ for all $ 1 \le q < 2N/(N-2+\gamma) $, whence $ f \in H^1_{\mathrm{loc}}\!\left( \R^N \right) $ owing to $\gamma<2$. Now, let $ \{\phi_R\}_{R \ge 1} $ be a family of cutoff functions as in Definition \ref{den-cutoff}; if we could prove that
\begin{equation}\label{hdot-1}
    \limsup_{R \to +\infty} \int_{\R^N} \left| \nabla\!\left( f \phi_R \right) \right|^2 dx < +\infty \, ,
\end{equation}
then it would follow immediately that $ f \in \dot{H}^1\!\left( \R^N \right) $. We have:
$$
\begin{aligned}
\int_{\R^N} \left| \nabla\!\left( f \phi_R \right) \right|^2 dx = & - \int_{\R^N} f \,\phi_R \left[  f \, \Delta \phi_R + 2 \left\langle \nabla f , \nabla \phi_R \right\rangle + \phi_R \, \Delta f  \right] dx \\
= & - \int_{\R^N} f^2 \,\phi_R \, \Delta \phi_R \, dx + \frac 1 2 \int_{\R^N} f^2 \, \Delta \!\left(\phi_R^2 \right) dx - \int_{\R^N} f \,\phi_R^2 \, \Delta f \, dx \\
\le & \, C \left[ \frac{1}{R^2} \int_{B_{2R}\setminus B_R} f^2 \, dx +  \left\| f \right\|_{L^2_\rho\left( \R^N \right)} \left\| \Delta f \right\|_{L^2_{\rho^{-1}}\left( \R^N \right)} \right] \\
\le & \, C \left[\left\| f \right\|_{L^2_\rho\left( \R^N \right)}^2 +  \left\| f \right\|_{L^2_\rho\left( \R^N \right)} \left\| \Delta f \right\|_{L^2_{\rho^{-1}}\left( \R^N \right)} \right] ,
\end{aligned}
$$
where $C>0$ is a constant that may change from line to line and depends only on $ N,\gamma, \underline{C} $, whence \eqref{hdot-1}. Finally, as a consequence of a routine integration by parts, for all $ g \in D(A) $ and all $ \hat{f} \in \dot{H}^1\!\left( \R^N \right) $ with compact support it  holds
$$
 \int_{\R^N} \hat{f} \, A[g] \, \rho \, dx \, = - \int_{\R^N} \hat{f} \, \Delta g \, dx   = \int_{\R^N} \left\langle \nabla \hat{f} \, , \nabla g \right\rangle dx \, .
$$
On the other hand, by means of a density and truncation argument in $ \dot{H}^1\!\left( \R^N \right) \cap L^2_\rho\!\left( \R^N \right) $, it is easily seen that the same identity holds for all $ f \in D[A] $, thus \eqref{self-adj-prop} follows.
\end{proof}
Following \cite[Section 2]{MP}, we now introduce the following uniqueness space $Y$ for unsigned solutions to \eqref{wpme-funid}:
\begin{equation}\label{def-space-Y}
Y(I)=\left\{f\in L^\infty_{\mathrm{loc}}\left(\R^N\times I\right)\colon \ \sup_{R\geq r}\frac{\norm{f}_{L^\infty(B_R\times I)}}{R^{\frac{2-\gamma}{m-1}}}<+\infty \right\} ,
\end{equation}
where $I\subset\R$ is any time interval and $r\ge 1$ is a fixed parameter (it is plain that the definition of $Y$ does not depend on $r$). In this case, the $_{\mathrm{loc}}$ subscript will refer to $I$ only, that is, we write $f \in Y_{\mathrm{loc}}(I)$ if and only if $f \in Y(J)$ for all subinterval $J\Subset I$.

Next, we state two existence and uniqueness results from \cite{MP}, which will be crucially employed in the proofs of this section. In order to lighten the notation, and in agreement with Definition \ref{def-Morrey}, since the initial data considered therein are always of the form $ \mu \equiv u_0 \, \rho \in X $ for $ u_0 \in L^1_{\rho,\mathrm{loc}}\!\left( \R^N\right) $, we write $ u_0 \in X $ (and similar properties) to actually mean $ \mu \in X $.

\begin{pro}[Theorem 2.2 in \cite{MP}]\label{mp-exist}
Let  $N\geq3$, $m>1$, and $\rho$ be a measurable function satisfying \eqref{weight-cond}. Let $ u_0  \in X $. Then there exists a solution $ u $ of problem \eqref{wpme-funid} with $ T \equiv T(u_0) $, in the sense that it complies with Definition \ref{def:vw-gl},	$u\in C\!\left([0,T(u_0));L^1(\Phi_\alpha)\right) $, and $ u(0) = u_0 $, where
$$
T(u_0) = \frac{C_1}{[\ell(u_0)]^{m-1}} \quad \text{if} \;\; u_0 \in X \setminus X_0 \, , \qquad T(u_0) = +\infty \quad \text{if} \; \; u_0 \in X \setminus X_0 \, ,
$$
for a suitable constant $C_1>0$ depending only on $ N,m,\gamma,\underline{C},\overline{C} $.
Furthermore, the following additional properties are satisfied. Setting (let $ r \ge 1 $)
\begin{equation}\label{def-Tr}
T_r(u_0) = \frac{C_1}{\| u_0 \|_{1,r}^{m-1}} \, ,
\end{equation}
the estimates
\begin{equation}\label{1-r-a-p}
\| u(t) \|_{1,r} \le C_2 \, \| u_0 \|_{1,r} \qquad \forall t \in (0,T_r(u_0))
\end{equation}
and
\begin{equation}\label{smooth-a-p}
\| u(t) \|_{\infty,r} \le C_3 \, t^{-\lambda} \, \| u_0 \|_{1,r}^{\theta \lambda}   \qquad \forall t \in (0,T_r(u_0))
\end{equation}
hold for some constants $ C_2,C_3>0 $ depending only on $ N,m,\gamma,\underline{C},\overline{C} $, and if $v$ is the constructed solution associated with another initial datum $v_0\in X$, then the stability estimates
\begin{equation}\label{l1-weight-contr}
	\norm{u(t)-v(t)}_{L^1(\Phi_\alpha)}\leq \exp\!\left(C_4 \, t^{\theta\lambda}\right) \norm{u_0-v_0}_{L^1(\Phi_\alpha)} \qquad \forall t\in(0,T_r(u_0) \wedge T_r(v_0))
\end{equation}
and
\begin{equation}\label{cutoff-weight-cont}
	\left|u(t)-v(t)\right|_{1,r}\leq \exp\!\left(C_5\, t^{\theta\lambda}\right) \left|u_0-v_0\right|_{1,r} \qquad \forall t\in(0,T_r(u_0) \wedge T_r(v_0))
\end{equation}
hold for a constant $C_4>0$ depending only on $N,m,\gamma,\underline{C},\overline{C},\alpha,r,\norm{u_0}_{1,r} , \norm{v_0}_{1,r} $ and a constant $C_5>0$ depending only on $N,m,\gamma,\underline{C},\overline{C},\norm{u_0}_{1,r},\norm{v_0}_{1,r}$, respectively. Also, $ u_0 \le v_0 $ implies $ u \le v $ in $ \R^N \times (0,T(v_0)) $.

Finally, if in addition $ u_0 \in X_0 $, then $ u \in C\!\left([0,+\infty) ; X \right) $ and
\begin{equation}\label{ess-lim-subc}
\underset{|x| \to +\infty}{ \esslim } \, |x|^{-\frac{2-\gamma}{m-1}} \, u(x,t) = 0 \qquad \forall t >0 \, .
\end{equation}
 \end{pro}

\begin{pro}[Theorem 2.3 in \cite{MP}]\label{mp-uniq}
	Let  $N\geq3$, $m>1$, $T \in (0,+\infty]$, and $\rho$ be a measurable function satisfying \eqref{weight-cond}. Let $u , v \in C\big( [0,T) ; L^1_{\rho,\mathrm{loc}}\!\left( \R^N \right) \! \big) $ be any two solutions of problem \eqref{wpme-funid} with the same initial datum $ u_0 \in X $, in the sense that they comply with Definition \ref{def:vw-gl} and $ u(0)=v(0)=u_0  $. If
	\begin{equation*}\label{suf-cond-pm-uniq}
		u,v\in Y_{\mathrm{loc}}\!\left((0,T)\right)\cap L^\infty_{\mathrm{loc}}\!\left([0,T);X\right) ,
	\end{equation*}
	then $u=v$.
\end{pro}

Now, let us present a simple generalization of \cite[Lemma 8.5]{Vazquez} concerning equicontinuity of solutions in suitable normed spaces.
\begin{lem}\label{time scaling trick}
Let $ \mathsf{Z} $ be a suitable subset of solutions to \eqref{wpme-funid} (containing the null solution), which are defined in a common time interval $  (0,S)$, $S \in (0,+\infty]$, such that the stability estimate
\begin{equation}\label{gen contraction}
	\norm{u(t)-v(t)}\leq F(t)\norm{u_0-v_0} \qquad \forall t \in(0,S)
\end{equation}
is valid for a suitable function $F\in C(\R^+)$, a norm $\norm{\cdot}$, and all $u, v\in \mathsf{Z}$. Let $ u \in \mathsf{Z} $ be such that the time-scaled solution $ \lambda \, u(\lambda^{m-1}\,t)$ still belongs to $\mathsf{Z}$ for all $ \lambda \in \left(1,\overline{\lambda}\right) $, for some $ \overline{\lambda} \in (1,2) $. Then, for all $t_0 \in (0,S)$, it holds
\begin{equation}\label{time scaling trick eq}
	\frac 1 h \norm{u(t_0+h)-u(t_0)}\leq \tfrac{2}{(m-1)} \cdot \tfrac{F(t_0)}{t_0} \, \overline{\lambda}^{(2-m)_+} \norm{u_0} \qquad \text{for all} \; \;  0 < h < \left(\overline{\lambda}^{m-1}-1\right)t_0 \, .
\end{equation}
\end{lem}
\begin{proof}
For any $ \lambda \in \left(1,\overline{\lambda}\right) $, let us consider the time-scaled solution $\tilde{u}(x,t)=\lambda \, u(x,\lambda^{m-1}\,t)$, which, by assumption, belongs to $ \mathsf{Z} $ and takes the initial datum $\tilde{u}_0(x)=\lambda \, u_0(x)$. Given any $ t_0 \in (0,S) $ and an arbitrary increment $ h $ as in the statement, we notice that the choice
\begin{equation}\label{choice-lambda}
\lambda = \left( 1 + \tfrac{h}{t_0} \right)^{\frac{1}{m-1}} < \overline{\lambda}
\end{equation}
is feasible and yields $ \lambda^{m-1}\,t_0=t_0+h $, so that by simple algebraic manipulations we obtain
$$
\begin{aligned}
	u(x,t_0+h)-u(x,t_0)&=\lambda^{-1}\,\tilde{u}(x,t_0)-u(x,t_0)\\
	&=\left(\lambda^{-1}-1 \right)\tilde{u}(x,t_0)+\left[\tilde{u}(x,t_0)-u(x,t_0) \right] .
\end{aligned}
$$
Taking the norm $\norm{\cdot}$ on both sides and applying \eqref{gen contraction} at $ t=t_0 $, first with $ u=\tilde{u} $ and $ v=0 $, then with $u \equiv u$ and $ v=\tilde{u} $, yields
\begin{equation*}
		\norm{u(t_0+h)-u(t_0)} \leq F(t_0)\left(\left|\lambda^{-1}-1\right|\norm{\tilde{u}_0}+\norm{\tilde{u}_0-u_0}\right) \le 2 F(t_0) \,(\lambda-1)\norm{u_0}  .
        \end{equation*}
By \eqref{choice-lambda} and the Lagrange theorem, it is plain that
$$
\lambda-1 = \tfrac{1}{(m-1)t_0} \left( 1 + \tfrac{\xi}{t_0} \right)^{\frac{2-m}{m-1}} h
$$
for some $ \xi \in (0,h) $, whence the thesis follows.
\end{proof}

We use the preceding propositions and lemma to obtain a compactness result in parabolic spaces that will be employed in the proofs of both existence and uniqueness. Also, for the sake of clarity, we recall from Definition \ref{boch} that when we write compactness or convergence in spaces of continuous functions defined in \emph{open} time intervals, we implicitly mean locally uniformly in compact subintervals.

\begin{lem}\label{alaa compact2}
Let $\mathfrak{I}$ be a possibly uncountable set and let $\left\{u_\beta\right\}_{\beta\in \mathfrak{I}}$ be a family of strong energy solutions to \eqref{wpme-noid-loc}, in the sense of Definition \ref{def:w-loc},  with $ \Omega \times (\tau_1,\tau_2) = \R^N \times (0,T) $. Also, suppose that $\left\{u_\beta\right\}$  is bounded in $Y_{\mathrm{loc}}((0,T))$. Then
	\begin{equation}\label{loc-comp-lemma}
    \begin{gathered}
	    \left\{u_\beta\right\}\;\text{is precompact in}\;\;C\!\left((0,T);L^1(\Phi_\alpha)\right) \; \text{and bounded in} \\
        \mathrm{Lip}_{\mathrm{loc}}\!\left( \left(0,T \right) ; X \right) \cap W^{1,\infty}_{\mathrm{loc}} \!\left( (0,T) ; L^1(\Phi_\alpha) \right) ,  \\
        \left\{u_\beta^m\right\}\;\text{is weakly precompact in}\;\; L^2_{\mathrm{loc}}\!\left((0,T);H^1_{\mathrm{loc}}\!\left(\R^N\right)\right) \cap H^{1}_{\mathrm{loc}}\!\left((0,T);L^2_{\rho,\mathrm{loc}}\!\left(\R^N\right)\right) .
        \end{gathered}
    \end{equation}
\end{lem}
\begin{proof}
The second statement in \eqref{loc-comp-lemma} is shown with a completely analogous argument to the one leading to \eqref{types of convergence smoothing}, since boundedness in $ Y_{\mathrm{loc}}((0,T)) $ implies boundedness in $ L^\infty_{\mathrm{loc}}\!\left( \R^N \times (0,T)  \right) $. This ensures the existence of an element $ u_\star $, and a sequence in $ \{ u_\beta \} $ that we simply denote by $ \{ u_k \} $, such that
\begin{equation}\label{types of convergence2}
	\begin{aligned}
		u_{k}^m\rightharpoonup u_\star^m\qquad & \text{in}\;\; L^2_{\mathrm{loc}}\!\left((0,T);H^1_{\mathrm{loc}}\!\left(\R^N\right)\right) ,  \\
        \partial_t\!\left( u_{k}^m \right) \rightharpoonup \partial_t\!\left( u_\star^m \right) \qquad & \text{in}\;\; L^2_{\mathrm{loc}}\!\left((0,T);L^2_{\rho,\mathrm{loc}}\!\left( \R^N \right)\right) ,  \\
		u_{k}\to u_\star \qquad & \text{in}\;\; L^2_{\mathrm{loc}}\!\left((0,T);L^2_{\rho,\mathrm{loc}}\!\left( \R^N \right)\right),
	\end{aligned}
\end{equation}
as $k\to \infty$, the last convergence following from the numerical inequality
\begin{equation*}\label{numer-eq-asy}
	|A-B|\leq m\left||A|^{m-1}A-|B|^{m-1}B\right|^{\frac{1}{m}}\qquad\forall A,B\in\R\,.
\end{equation*}

Next, we intend to improve the convergence of $ \{ u_k \} $ to $ u_\star $ in $C\!\left((0,T);L^1(\Phi_\alpha)\right)$ by employing the Ascoli--Arzel\`a theorem. To this end, let us fix an arbitrary $ \tau \in (0,T/4) $ and set
$$
M_\tau = \sup_{ \substack{ s \in [\tau,T - \tau]  \\ \beta \in \mathfrak{I}} }  2^{\frac{1}{m-1}} \left\| u_\beta(s) \right\|_{1,r} ,
$$
which is a finite quantity thanks to the $ Y_{\mathrm{loc}}((0,T)) $ boundedness assumption and the strong-energy property. Recalling the notation from Proposition \ref{mp-exist}, let us also set
$$
T_{r,\tau} = T_r(\cdot,M_\tau) \, , \qquad C_{4,\tau} = C_4(\cdot,M_\tau) \, , \qquad C_{5,\tau} = C_5(\cdot,M_\tau) \, .
$$
In fact, it is plain that $ T_r(u_0) $ depends on the usual parameters and on $ u_0 $ through $ \| u_0 \|_{1,r} $ only, in a monotone-decreasing way; also, a closer inspection to the proof of \cite[Theorem 2.2]{MP} reveals that both $ C_4 $ and $C_5$ depend in a monotone-increasing way on $ \| u_0 \|_{1,r} \vee \| v_0 \|_{1,r} $, so the above definitions are sensible.

As a result, upon choosing $ S = T_{r,\tau} $ and $ \mathsf{Z} $ as the set of solutions to \eqref{wpme-funid} provided by Proposition \ref{mp-exist} taking any initial datum $ u_0 \in X $ such that $ \| u_0 \|_{1,r} \le M_\tau $, thanks to \eqref{l1-weight-contr} we have that
\begin{equation*}\label{gen contraction applied}
	\norm{u(t)-v(t)}_{L^1(\Phi_\alpha)} \leq \exp\!\left(C_{4,\tau} \, t^{\theta\lambda}\right) \norm{u_0-v_0}_{L^1(\Phi_\alpha)} \qquad \forall t \in(0,T_{r,\tau}) \, ,
\end{equation*}
for all $ u,v \in \mathsf{Z} $. Let $ \hat{\tau} = \tau \wedge (T_{r,\tau}/2) $. In view of Proposition \ref{mp-uniq}, for all $ t_0 \in (2\tau,T-2\tau) $ we can assert that $ u_\beta(t+t_0-\hat{\tau}) $ coincides (at least up to $ (T-t_0+\hat{\tau}) \wedge T_{r,\tau} $) with the solution constructed in Proposition \ref{mp-exist} taking the initial datum $ u_\beta(t_0-\hat{\tau}) $, which belongs to $ \mathsf{Z} $ and we call $w(t)$. Moreover, if we take $ \overline{\lambda}=2^{{1}/{(m-1)}} $, from the definition of $ M_\tau $ we have that also $  \lambda \, w(\lambda^{m-1} \, t) $ belongs to $ \mathsf{Z} $ for all $ \lambda \in \left(1,\overline{\lambda}\right) $. We are thus in position to apply \eqref{time scaling trick eq} to $ u = w$ at $ t_0 = \hat{\tau} $ (this is feasible since $\hat{\tau}<T_{r,\tau}$), which in terms of $ u_\beta $ reads
\begin{equation}\label{time scaling trick applied}
\begin{gathered}
	\frac 1 h \norm{u_\beta(t_0+h)-u_\beta(t_0)}_{L^1(\Phi_\alpha)} \leq \tfrac{2}{(m-1)} \cdot \tfrac{\exp\left(C_{4,\tau} \, {\hat \tau}^{\theta\lambda}\right)}{\hat{\tau}} \, 2^{\frac{(2-m)_+}{m-1}} \norm{u_\beta(t_0-\hat{\tau})}_{L^1(\Phi_\alpha)} \\ \forall  h \in ( 0,\hat{\tau}) \, .
    \end{gathered}
\end{equation}
Thanks to the arbitrariness of $ \tau , t_0 $ and again boundedness in $ Y_{\mathrm{loc}}((0,T)) $, estimate \eqref{time scaling trick applied} ensures that $ \{ u_\beta \} $ is locally uniformly equicontinuous in $ C\!\left((0,T); L^1(\Phi_\alpha) \right) $.

To complete the Ascoli--Arzel\`a argument, we need to exhibit the pointwise relative compactness of the family. That is, we will show that
\begin{equation}\label{pw-rel-comp}
    u_k(t) \underset{k \to \infty}{\longrightarrow} u_\star(t) \qquad \text{in}\;\;L^1(\Phi_\alpha) \;\; \text{for all}\;\;t\in(0,T)\,.
\end{equation}
To this end, as local uniform boundedness plus $ L^2_{\mathrm{loc}}\!\left(\mathbb{R}^N\right) $ convergence implies $ L^1_{\rho,\mathrm{loc}}\!\left(\mathbb{R}^N \right) $ convergence, it is only a question of extending the latter to the aforementioned global space. In fact, it is enough to prove the following smallness claim: given any $ t \in (0,T) $, for all $\varepsilon>0$ there exist $R_\varepsilon \ge 1$ and $k_\varepsilon\in\mathbb{N}$ such that
\begin{equation}\label{small-claim-1}
    \int_{B_{R_\varepsilon}^c}|u_k(t)|\,\Phi_\alpha\,\rho\,dx\leq \varepsilon\qquad\text{for all}\;\;k \ge k_\varepsilon\,.
\end{equation}
Indeed, let us see that \eqref{small-claim-1} implies \eqref{pw-rel-comp}. We have:
\begin{equation*}
\begin{aligned}
    \int_{\R^N}\left|u_k(t)-u_\star(t)\right| \Phi_\alpha \,\rho\,dx&=\int_{B_{R_\varepsilon}} \left|u_k(t)-u_\star(t)\right| \Phi_\alpha\,\rho\,dx+\int_{B_{R_\varepsilon}^c}\left|u_k(t)-u_\star(t)\right|\Phi_\alpha\,\rho\,dx\\
    &\leq \int_{B_{R_\varepsilon}}\left|u_k(t)-u_\star(t)\right| \rho\,dx+2\varepsilon\,,
\end{aligned}
\end{equation*}
for all $k\geq k_\varepsilon$ (clearly \eqref{small-claim-1} also holds at the limit). Taking $ k \to\infty$, using \eqref{types of convergence2}, and then letting $\varepsilon \to 0$ yields the claim. Hence, let us prove \eqref{small-claim-1}: for all $ \ell \in \mathbb{N} $ it holds
\begin{equation*}
\begin{aligned}
    \int_{B_{2^\ell}^c}|u_k(t)|\,\Phi_\alpha\,\rho\,dx&=\sum_{i=\ell}^\infty\int_{B_{2^{i+1}}\setminus B_{2^i}}|u_k(t)|\,\Phi_\alpha\,\rho\,dx\\
    &\leq C\,\sum_{i=\ell}^\infty \left(1+{2^{2i}}\right)^{-\alpha} \, 2^{(i+1)(N-\gamma)} \norm{u_k(t)}_{L^\infty\left(B_{2^{i+1}}\right)}\\
    &\leq C \norm{u_k(t)}_{\infty,r} \, \sum_{i=\ell}^\infty\left(1+{2^{2i}}\right)^{-\alpha}\,2^{(i+1)\left(N-\gamma+\frac{2-\gamma}{m-1}\right)}\,,
\end{aligned}
\end{equation*}
for some generic constant $C>0$ depending only on $N,\gamma, \overline C$, and where in the second inequality we used the definition of the Morrey-type norms \eqref{def-norm-inf-r}. By the condition \eqref{alpha-cond} on $\alpha$, the series converges so, using the fact that $ \left\{ \|u_k(t)\|_{r,\infty} \right\}_k$ is bounded, it is enough to take $ \ell $ sufficiently large depending on $\varepsilon$ to conclude the proof of \eqref{small-claim-1}. Therefore \eqref{pw-rel-comp} holds, hence by the Ascoli--Arzel\`a theorem, we have convergence in $C\!\left((0,T);L^1(\Phi_\alpha)\right)$.

 We now exploit \cite[Theorem 1.1]{BG}, which is applicable in view of \eqref{time scaling trick applied} and by the fact that $u_\beta^m\in H^{1}_{\mathrm{loc}}\big((0,T);L^2_{\rho,\mathrm{loc}}\!\left(\mathbb R^N \right)\!\big)$, to deduce that $u_\beta\in W^{1,1}_{\mathrm{loc}}\big((0,T); L^1_{\rho,\mathrm{loc}}\!\left(\mathbb R^N\right)\!\big)$. We can then pass to the limit as $h\to0 $ in \eqref{time scaling trick applied}, obtaining
\begin{equation}\label{W11}
\left\| \partial_t u_\beta(t_0) \right\|_{L^1(\Phi_\alpha)} \le C_\tau \qquad \text{for a.e.} \;\; t_0 \in (2\tau,T-2\tau) \, ,
\end{equation}
where $ C_\tau>0 $ is a suitable constant that depends on $\tau$ (as well as on all of the usual parameters) via the quantity $ \hat{\tau} $ and the constant $ C_{4,\tau}$ defined above, and in addition on
$$
\sup_{ \substack{ t \in (\tau,T - \tau)  \\ \beta \in \mathfrak{I}} }  \norm{u_\beta(t)}_{\infty,r}  .
$$
As a consequence of \eqref{W11} and the arbitrariness of $ \tau $, we can therefore assert that $ \{ u_\beta \}$ is bounded in $ W^{1,\infty}_{\mathrm{loc}} \!\left( (0,T) ; L^1(\Phi_\alpha) \right) $.

Finally, by repeating the same argument that led to \eqref{time scaling trick applied}, taking advantage of \eqref{cutoff-weight-cont} in the place of \eqref{l1-weight-contr}, we end up with the estimate
\begin{equation*}\label{time scaling trick applied X}
	\frac 1 h \left| u_\beta(t_0+h)-u_\beta(t_0) \right|_{1,r} \leq \tfrac{2}{(m-1)} \cdot\tfrac{\exp\left(C_{5,\tau} \, {\hat \tau}^{\theta\lambda}\right)}{\hat{\tau}} \, 2^{\frac{(2-m)_+}{m-1}} \left| u_\beta(t_0-\hat{\tau}) \right|_{1,r}  \qquad \forall  h \in ( 0,\hat{\tau}) \, ,
\end{equation*}
so that boundedness in $ \mathrm{Lip}_{\mathrm{loc}}\!\left( \left(0,T \right) ; X \right) $ follows straightforwardly.
\end{proof}

\begin{rem}[Constructed solutions from \cite{MP} are strong energy]\label{ws}\rm
If one recalls the method employed to construct solutions in \cite[Theorem 2.2]{MP} (which is rewritten here as Proposition \ref{mp-exist}), it follows easily from Lemma \ref{energy lem} that those solutions are in fact strong energy, in the sense of Definition \ref{def:w-loc}. Indeed, let us consider the sequence $\left\{u_n\right\}$, where each $u_n$ is the weak energy the solution of problem \eqref{wpme-funid}, guaranteed by Proposition \ref{pro1}, which takes the truncated initial datum
\begin{equation*}
    u_{0n}=\left[(u_0\land n)\lor (-n)\right] \chi_{B_n}\,.
\end{equation*}
Owing to \eqref{strong energy}, such $u_n$ are strong energy solutions, and by the smoothing estimate of type \eqref{smooth-a-p} (or more precisely \cite[Proposition 3.4]{MP} -- which holds \emph{a priori} at the level of $u_n$) it follows that
\begin{equation*}
    \left\{u_n\right\} \;\; \text{is bounded in}\;\;Y_{\mathrm{loc}}\!\left((0,T)\right) ,
\end{equation*}
for some $T>0$ that can be sharply estimated according to the mentioned proposition. Hence, the energy estimates of Lemma \ref{energy lem} apply to $u_n$ and are stable with respect to $n$, so that by letting $n \to \infty$ we find that also the the corresponding limit $ u $ (that has already been identified in \cite{MP} as the solution provided by Proposition \ref{mp-exist}) is strong energy.
\end{rem}

\subsection{Existence of solutions taking unsigned measures}\label{existence}

\begin{proof}[Proof of Theorem \ref{mBCP}]
Let $ \{\eta_\varepsilon \}_{\varepsilon>0}$ be a standard sequence of mollifiers, and let us consider the following approximate initial datum obtained by convolution:
\begin{equation*}
    \mu_{\varepsilon} = \rho^{-1} \,  (\eta_\varepsilon\ast\mu)\,.
\end{equation*}
By the definition of Euclidean convolution, it can be easily checked that 
 there exists $\varepsilon_0$ such that, for all $\varepsilon<\varepsilon_0$ one has:
\begin{equation}\label{norm-est-mu}
    \norm{\mu_\varepsilon}_{1,r} \leq 2\norm{\mu}_{1,r} .
\end{equation}
Then, from the initial datum $\mu_\varepsilon$, let us introduce the constructed solution
\begin{equation*}
    u_{\varepsilon}\in C\!\left([0,T(\mu_\varepsilon));L^1(\Phi_\alpha)\right)
\end{equation*}
of \eqref{wpme-funid}, provided by Proposition \ref{mp-exist}.

In preparation for using Lemma \ref{alaa compact2}, let us apply the estimates \eqref{1-r-a-p}--\eqref{smooth-a-p} to $ u_\varepsilon $:
\begin{equation*}\label{appl smoothing estimate}
	\left\| u_{\varepsilon}(t) \right\|_{1,r} \leq C_2 \left\| \mu_{\varepsilon} \right\|_{1,r}  \qquad \forall t \in \left( 0 , T_r(\mu_{\varepsilon}) \right) ,
\end{equation*}
\begin{equation*}\label{appl key estimate}
	\left\| u_{\varepsilon}(t) \right\|_{\infty,r} \leq C_3 \, t^{-\lambda} \left\| \mu_{\varepsilon} \right\|_{1,r}^{ \theta \lambda} \qquad \forall t \in \left( 0 , T_r(\mu_{\varepsilon}) \right) .
\end{equation*}
Combining the latter with \eqref{def-Tr} and \eqref{norm-est-mu}, we deduce
\begin{equation}\label{appl smoothing estimate bis}
	\left\| u_{\varepsilon}(t) \right\|_{1,r} \leq 2\,C_2 \left\| \mu \right\|_{1,r} \qquad \forall t \in \left( 0 , \mathsf{T}_r(\mu) \right)
\end{equation}
and
\begin{equation}\label{appl key estimate bis}
	\left\| u_{\varepsilon}(t) \right\|_{\infty,r} \leq 2^{\theta\lambda} \, C_3 \, t^{-\lambda} \left\| \mu \right\|_{1,r}^{ \theta \lambda} \qquad \forall t \in \left( 0 , \mathsf{T}_r(\mu) \right) ,
\end{equation}
for all $\varepsilon\in(0,\varepsilon_0)$, where we set
$$
\mathsf{T}_r(\mu)= \frac{C_1}{2^{m-1} \|\mu\|_{1,r}^{m-1}} \, .
$$
From \eqref{appl smoothing estimate bis}--\eqref{appl key estimate bis} and Remark \ref{ws}, it is plain that $\left\{u_\varepsilon\right\}_{\varepsilon \in (0 , \varepsilon_0)}$ satisfies all of the hypotheses of Lemma \ref{alaa compact2}. We thus obtain a subsequence of $\varepsilon \to 0$ (that for simplicity we do not relabel) and a limit $u$ such that
\begin{equation}\label{appl conv 2}
	\begin{aligned}
		u_{\varepsilon}^m\rightharpoonup u^m \! & \qquad  \text{in}\;\; L^2_{\mathrm{loc}}\!\left((0, \mathsf{T}_r(\mu));H^1_{\mathrm{loc}}\!\left(\R^N\right)\right) \cap H^{1}_{\mathrm{loc}}\!\left((0,\mathsf{T}_r(\mu) );L^2_{\rho,\mathrm{loc}}\!\left(\R^N\right)\right),\\
		u_{\varepsilon}\to u & \qquad  \text{in}\;\; C\!\left(\left(0,\mathsf{T}_r(\mu)\right);L^1(\Phi_\alpha)\right).
	\end{aligned}
\end{equation}
This convergence is clearly enough to see that $u$ is a very weak solution to \eqref{wpme-noid} (at least up to $ T=\mathsf{T}_r(\mu) $ -- see below), in the sense of Definition \ref{def:vw-gl}, and it satisfies the statements \eqref{def-t-u0m} (upon letting $r \to +\infty$), \eqref{smoothing estimatem}, and \eqref{key estimatem}. However, we still must confirm that $\mu$ is the initial datum of $u$ in the sense of \eqref{measure-data}.

To prove that the initial trace is in fact $\mu$, let us return to the approximations $u_\varepsilon$ for all $\varepsilon \in (0,\varepsilon_0)$.
First of all, we claim that for all $\varphi \in C_c^\infty\!\left(\mathbb R^N\right)$ it holds
\begin{equation}\label{eps-claim}
    \int_{\R^N}u_\varepsilon(s)\,{\varphi}\,\rho\,dx=\int_{\R^N}\left(\eta_\varepsilon\ast\mu\right) {\varphi}\,dx+\int_0^s\int_{\R^N}u_\varepsilon^m\,\Delta {\varphi}\,dxdt\,,
\end{equation}
for all $s\in[0,\mathsf{T}_r(\mu)/2]$. Indeed, such a claim can easily be verified by using the extra regularity property of $u_\varepsilon$ ensured by Proposition \ref{mp-exist} and Remark \ref{ws}. Now, let us estimate quantitatively the second term on the right-hand side of \eqref{eps-claim},
assuming without loss of generality that $ \operatorname{supp} \varphi \subset B_R$ for some $R>0$. We have:
\begin{equation*}\label{nlest}
\begin{aligned}
    \int_0^s\int_{\R^N}\left|u_\varepsilon\right|^m \left|\Delta {\varphi}\right| dxdt & \leq C\norm{\Delta {\varphi}}_{L^\infty(B_R)}\int_0^s\norm{u_\varepsilon(t)}_{L^\infty(B_R)}^{m-1}\left(\int_{B_R}|u_\varepsilon(t)|\,\rho \, dx \right) dt \\
    & \leq C \norm{\Delta{\varphi}}_{L^\infty(B_R)} \norm{\mu}_{1,r}^{\theta\lambda(m-1)+1}\int_0^s t^{-\lambda(m-1)}\,dt\\
    & = C \norm{\Delta {\varphi}}_{L^\infty(B_R)} \norm{\mu}_{1,r}^{\theta\lambda(m-1)+1} s^{\theta\lambda} \, ,
\end{aligned}
\end{equation*}
for a constant $C>0$ depending only on $N , m ,\gamma ,\underline{C} , \overline{C}$, and $R$. Note that in the first inequality we have used \eqref{weight-cond} and in the second step we have crucially exploited \eqref{appl smoothing estimate bis}--\eqref{appl key estimate bis}, whereas the time integration readily follows from \eqref{spexprel}. This estimate allows us to pass to the limit in \eqref{eps-claim}, first in $\varepsilon\to0$ and then in $s\to0$, to conclude that
\begin{equation}\label{in-dat}
    \lim_{s\to0}\int_{\R^N} u(s)\,{\varphi}\,\rho\,dx=\int_{\R^N}{\varphi}\,d\mu\,,
\end{equation}
for all ${\varphi}\in C^\infty_c\!\left(\R^N\right)$. Due to the density of $ C^\infty_c\!\left(\R^N\right) $ in $ C_c\!\left(\R^N\right) $ with respect to the strong $ L^\infty $ topology, it is straightforward to check that \eqref{in-dat} also holds for arbitrary $ \varphi \in C_c\!\left(\R^N\right) $, whence \eqref{measure-data}.

Next, let us prove \eqref{Lip-X}. Still by Lemma \ref{alaa compact2}, we have
\begin{equation}\label{unif}
\|u_\varepsilon(t+h)-u_\varepsilon(t)\|_{L^1(\Phi_\alpha)}
\le C_\tau \, h \qquad \forall t \in \left(2\tau, \mathsf{T}_r(\mu)-2\tau\right )  , \ \forall h\in(0,\tau) \, ,
\end{equation}
and
\begin{equation}\label{unif-2}
\left| u_\varepsilon(t+h)-u_\varepsilon(t) \right|_{1,r}
\le C_\tau \, h \qquad \forall t\in\left(2\tau, \mathsf{T}_r(\mu)-2\tau\right) , \ \forall h\in(0,\tau) \, ,
\end{equation}
for any $\tau\in \left(0, \mathsf{T}_r(\mu)/4 \right)$, where $ C_\tau>0 $ is a general constant independent of $ \varepsilon $. The fact that $ u \in \mathrm{Lip}_{\mathrm{loc}}\!\left( \left( 0 , \mathsf{T}_r(\mu) \right) ; X \right) $ is a direct consequence of passing to the limit in \eqref{unif-2} as $ \varepsilon \to 0 $, recalling that the $ | \cdot |_{1,r} $ norm is lower semicontinuous with respect to convergence in $ L^1_{\rho,\mathrm{loc}}\!\left( \R^N \right) $ (see \emph{e.g.}~\cite[Proposition A.1]{MP}). Similarly,  we can obtain the analogue of \eqref{unif} with $u$ replacing $u_\varepsilon$. Moreover, in this case we can also make sure that $ \partial_t u $ exists (and is bounded) in $ L^1(\Phi_\alpha) $. Indeed, using the fact that $\partial_t(u^m)\in L^2_{\mathrm{loc}}\big(\big(0,\mathsf{T}_r(\mu)\big);L^2_{\rho,\mathrm{loc}}\big(\mathbb{R}^N\big)\big)$, as follows from the first part of \eqref{appl conv 2}, we can apply again \cite[Theorem 1.1]{BG} exactly as in the proof of Lemma \ref{alaa compact2}, dividing by $h$ and letting $ h \to 0 $ in \eqref{unif} to find that $u \in W^{1,\infty}_{\mathrm{loc}}\!\left(\left(0,\mathsf{T}_r(\mu)\right);L^1(\Phi_\alpha) \right)$. We conclude by noting that $r$ is arbitrary and using a diagonal argument, to deduce that $u$ is in fact is well defined, and the corresponding regularity properties hold, up to time $T=\mathsf{T}(\mu)$.

The comparison inequality $ u \le v $ when $ \mu \le \nu $ is an immediate consequence of the same inequality in Proposition \ref{mp-exist} and the above approximation method.

Finally, let us briefly deal with the case $ \mu \in X_0 $. From \eqref{smoothing estimatem}, we infer that if $ \mu \in X_0 $ then also $ u(t) \in X_0 $ for all $ t \in \left( 0  , \mathsf{T}(\mu) \right) $, so \eqref{subcritical-growth} is a consequence of the same property \eqref{ess-lim-subc} in Proposition \ref{mp-exist} (via the uniqueness Proposition \ref{mp-uniq}). On the other hand, since $ L^1_{\rho}\!\left( \R^N \right) $ functions are dense in $ X_0 $, see \cite[Proposition A.2]{MP}, the fact that $ \partial_t u $ exists at least in $ L^1_{\rho,\mathrm{loc}}\!\left( \R^N \right) $ (consequence of the above argument) allows us to divide by $ h $ and let $ h \to 0 $ in \eqref{unif-2}, to find that $ \partial_t u $ also exists as a weak derivative in $ X_0 $ and is bounded with values in such a space, ensuring that $ W^{1,\infty}_{\mathrm{loc}}\!\left(\left(0,\mathsf{T}_r(\mu)\right); X_0 \right)  $.
\end{proof}

\subsection{Uniqueness without global assumptions}\label{uniqueness}
Before proving Theorem \ref{uniq thm}, we need an additional compactness result, which is a sort of global version of Lemma \ref{alaa compact2}.

\begin{lem}\label{alaa compact1}
Let $\mathfrak{I}$ be a possibly uncountable set and let $\left\{u_\beta\right\}_{\beta\in \mathfrak{I}}$ be a family of non-negative weak energy solutions to {\eqref{wpme-funid}}, in the sense of Definition \ref{def2}. Suppose that
\begin{equation}\label{L1-bound}
 \sup_{\beta \in \mathfrak{I}} \left\| u_\beta(0) \right\|_{L^1_\rho(\mathbb{R}^N)} < +\infty   \qquad \text{and} \qquad \exists R_0 \ge 1 : \ \operatorname{supp} u_\beta(0) \subset B_{R_0} \quad \forall \beta \in \mathfrak{I} \, .
\end{equation}
	Then
	\begin{equation}\label{AA-1}
    \begin{gathered}
	    \left\{u_\beta\right\}\;\text{is precompact in}\;\;C\!\left((0,+\infty);L^1_\rho\!\left(\mathbb{R}^N\right)\right),\\
        \left\{u_\beta^m\right\}\;\text{is weakly precompact in}\;\; L^2_{\mathrm{loc}}\big((0,+\infty);\dot{H}^1\!\left(\R^N\right)\!\big) \, .
        \end{gathered}
    \end{equation}
    \end{lem}
\begin{proof}
We begin by noticing that each $ u_\beta $ is also a strong energy solution, thanks to \eqref{strong energy}. Therefore, we may apply the global (in space) version of Proposition \ref{estimates.balls} obtaining, in particular,
\begin{equation}\label{h1 energy - lemma}
\int_{\R^N} \left|u_\beta(t_2)\right|^{m+1}  \rho \, dx + \int_{t_1}^{t_2} \int_{\R^N} \left|\nabla\!\left( u^m_\beta\right) \right|^2 dxdt =  \int_{\R^N} \left|u_\beta(t_1)\right|^{m+1} \rho \, dx
\end{equation}
for all $ t_2>t_1>0 $, which simply follows by \eqref{eq1-bdd} with (formally) $\varphi = \chi_{[t_1,t_2]} $. Discarding the first term on the left-hand side of \eqref{h1 energy - lemma}, and estimating the right-hand side by
$$
\int_{\R^N} \left|u_\beta(t_1)\right|^{m+1} \rho \, dx \le K^m\,t_1^{-\lambda m}\norm{u_\beta(0) }_{L^1_\rho(\mathbb{R}^N)}^{\theta\lambda m + 1} ,
$$
where we used \eqref{m37a}--\eqref{basic smoothing}, we easily obtain the second claim in \eqref{AA-1} by virtue of \eqref{L1-bound}.

As for the first claim, we notice that we are in a special case of Lemma \ref{alaa compact2} (owing to \eqref{basic smoothing}), so we already have that
\begin{equation}\label{loc-strong-comp}
   \left\{u_\beta\right\} \; \text{is precompact in}\;\;C\!\left((0,+\infty);L^1_{\rho,\mathrm{loc}}\!\left(\R^N\right)\right).
\end{equation}
Now it is just a question of improving \eqref{loc-strong-comp} to be global in space. To this end, let us study a convergent sequence
\begin{equation}\label{loc-strong-conv}
    u_{k} \underset{k \to \infty}{\longrightarrow} u_\star \qquad \text{in}\;\; C\!\left((0,+\infty);L^1_{\rho,\mathrm{loc}}\!\left(\R^N\right)\right) ;
\end{equation}
we claim that for all $\varepsilon,S>0$ there exists  $R_{\varepsilon}\ge 1$ (also depending on $S$) such that
\begin{equation}\label{small-space-claim}
    \sup_{t \in(0,S)}\norm{u_k(t)}_{L^1_\rho \left(B_{R_{\varepsilon}}^c\right)}\leq \varepsilon \qquad \forall k \in \mathbb{N} \,.
\end{equation}
Before proving the claim, let us quickly observe that it is sufficient in order to pass to global convergence. Indeed, for all $t\in(0,S)$ and all $k>k_{\varepsilon}$, it holds
\begin{equation*}
\begin{aligned}
    \norm{u_k(t)-u_\star(t)}_{L^1_\rho\left(\mathbb{R}^N\right)} & = \norm{u_k(t)-u_\star(t)}_{L^1_\rho\left(B_{R_\varepsilon}\right)} +\norm{u_k(t)-u_\star(t)}_{L^1_\rho\left(B_{R_\varepsilon}^c\right)}\\
    &\leq \norm{u_k(t)-u_\star(t)}_{L^1_\rho\left(B_{R_\varepsilon}\right)}+2\varepsilon\,.
\end{aligned}
\end{equation*}
We may now take the supremum over $t\in(t_1,t_2)$, for arbitrary $ 0<t_1<t_2<S $, and pass to the limit first as $k\to\infty$, using \eqref{loc-strong-conv}, and then as $\varepsilon\to0$.

Hence, let us prove \eqref{small-space-claim}. We start from the identity
$$
\int_{\R^N} \left(1-\phi_{R}\right)  u_k(t) \, \rho \, dx = \int_{\R^N} \left(1-\phi_{R}\right) u_k(0) \, \rho \, dx - \int_0^t \int_{\R^N} u_k^m \, \Delta \phi_{R} \, dx ds \, ,
$$
where $\phi_{R}$ is the cutoff function defined in Definition \ref{den-cutoff}.
By the support assumption in \eqref{L1-bound}, the first term on the right-hand side above vanishes provided $ R \ge R_0 $. Using again \eqref{m37a}--\eqref{basic smoothing}, along with \eqref{cutoff-est}, we have
$$
\begin{aligned}
\int_{B_{2R}} u_k(t) \, \rho \, dx \le \int_{\R^N} \left(1-\phi_{R}\right)  u_k(t) \, \rho \, dx  \le & \, \frac{C}{R^2} \int_0^t \int_{B_{2R} \setminus B_R} u_k^m \, dx ds \\
\le & \, \frac{3^\gamma C}{\underline{C} \, R^{2-\gamma}} \int_0^t \int_{B_{2R} \setminus B_R} u_k^m \, \rho \, dx ds \\
\le & \, \frac{3^\gamma C \, K^{m-1}}{\underline{C} \, R^{2-\gamma}} \left\| u_k(0) \right\|_{L^1_\rho\left(\R^N\right)}^{\lambda \theta(m-1) + 1} \int_0^t s^{- \lambda (m-1)} \, ds \\
\le & \,  \frac{3^\gamma C \, K^{m-1}}{\lambda \theta \,  \underline{C} \, R^{2-\gamma}} \, \sup_{\beta \in \mathfrak{I}}\left\| u_\beta(0) \right\|_{L^1_\rho\left(\R^N\right)}^{\lambda \theta(m-1) + 1} S^{\lambda \theta} \, ,
\end{aligned}
$$
whence the claim readily follows by choosing $R$ large enough.
\end{proof}

\begin{proof}[Proof of Theorem \ref{uniq thm}]\, \\[0.1cm]
\noindent{\textbf{Preliminary claim:}} Let $u$ be any non-negative solution to \eqref{wpme-noid}, in the sense of Definition \ref{def:vw-gl}. Then
\begin{equation}\label{claim-U}
u\in Y_{\mathrm{loc}}{((0,T))}\cap C\!\left((0,T);L^1\!\left(\Phi_\alpha\right)\right).
\end{equation}
This claim states that non-negative solutions are \emph{almost} in the uniqueness class of Proposition \ref{mp-uniq}, the only obstruction being the lack of regularity at $t=0$.
Indeed, as a consequence, we find that such solutions are uniquely determined by their value at \emph{all previous positive times $t>0$}, which is of course in general not automatic.

To prove \eqref{claim-U}, let us fix an arbitrary $R \ge 1$, arbitrary times $0<t_1<t_2<T-\delta$  (with $\delta>0$ sufficiently small), and consider the cylinders
$$
  \underline{Q}=B_{R}\times(t_2,T-\delta) \, , \qquad \overline{Q}=B_{2R}\times(t_1,T-\delta) \, .
$$
By Corollary \ref{glob-unif-est} we know that \eqref{l.infty.balls} holds, that is (focusing on the dependence on $R$ only),
\begin{equation*}
	R^{-\frac{2-\gamma}{m-1}} \norm{u}_{L^\infty\left(\underline{Q}\right)} \leq
		c_2\left[\left(R^{-(N-\gamma)-\frac{2-\gamma}{m-1}} \, \sup\limits_{t \in (t_1,T-\delta) }\int _{B_{2R}}u(t)\,\rho\,dx\right)^{1+\varepsilon}+1\right] ,
\end{equation*}
where $c_2>0$ is a constant depending only on $N,m,\gamma,\underline{C},\overline{C},\varepsilon, \delta, t_1,t_2, T$. Upon applying \eqref{weighted-AC-inq} with the appropriate choice of times (and time shifting) we get
\begin{equation*}\label{weighted-AC-inq-bis}
			\sup\limits_{t \in (t_1,T-\delta) }\int _{B_{2R}}u(t)\,\rho\,dx \leq C \left[\delta^{-\frac{1}{m-1}}R^{N-\gamma+\frac{2-\gamma}{m-1}}+\delta^{\frac{N-\gamma}{2-\gamma}} \norm{u}_{B_1\times \left(t_1+\frac{\delta}{
            4},T-\frac{\delta}{2}\right)}^{1+\frac{N-\gamma}{2-\gamma}(m-1)}\right] ,
            		\end{equation*}
where $C>0$ is a constant as in \eqref{weighted-AC-inq}, and on the right-hand side we simply estimated mean values with the $ L^\infty $ norm. Using also the local boundedness ensured by Theorem \ref{apriori-bounded}, and combining the previous two inequalities, it follows
\begin{equation}\label{norma X}
	R^{-\frac{2-\gamma}{m-1}} \norm{u}_{L^\infty\left(\underline{Q}\right)} \leq c_2 \, ,
\end{equation}
for a (possibly different) $c_2>0$ having all of the above dependencies. From \eqref{norma X} and the arbitrariness of $ R,t_1,t_2,\delta $, it is apparent that $u\in Y_\mathrm{loc}\left((0,T)\right)$. In particular, thanks to \cite[Proposition 4.1 and Remark 4.2]{MP}, we can assert that for almost every $\tau\in(0,T)$ the solution $u |_{t>\tau} $ coincides with the one constructed in Proposition \ref{mp-exist} taking the initial datum $u(\tau)$, at least up to time $\mathsf{T}=\left[ \tau+T(u(\tau)) \right] \land T $. Hence, it follows that
\begin{equation}\label{mathsfT}
    u\in C\!\left([\tau,\mathsf{T});L^1\!\left(\Phi_\alpha\right)\right) ,
\end{equation}
 In the case where $T > \tau+T(u(\tau)) $, it is sufficient to iterate the above procedure recalling \eqref{norma X}, which ensures that $ t \mapsto \ell(u(t)) $ is locally bounded in $(0,T)$. In this way, we can obtain \eqref{mathsfT} with $T$ replacing $ \mathsf{T} $, so the claim is proved owing to the arbitrariness of $\tau$.
\smallskip

\noindent{\textbf{Step 1: Set-up.}}
In view of the Preliminary claim, in particular, it holds that $u(\tau)\in X \cap L^\infty_{\mathrm{loc}}\!\left( \R^N \right) $ for all $\tau\in(0,T)$. Now, given a cutoff function $\phi_{2^{k}}$ as in Definition \ref{den-cutoff} for $k\in\mathbb{N}$, we consider the initial datum
$$
\phi_{2^{k}}\,u(\tau)\in L^1_\rho\!\left(\mathbb{R}^N\right) \cap L^\infty\!\left( \R^N \right) .
$$
We then apply Proposition \ref{pro1} to construct a unique non-negative, weak energy solution $w_{\tau,k}\in C\!\left([\tau,+\infty);L^1_\rho\!\left(\mathbb{R}^N\right)\right) \cap\, L^\infty\!\left(\R^N\times(\tau,+\infty)\right)$ to \eqref{wpme-funid} in the sense of Definition \ref{def2} (with time origin shifted to $\tau$), whose initial datum is precisely $\phi_{2^{k}}\,u(\tau)$. By local comparison, \emph{i.e.}~Corollary \ref{global comparison lemma}, and the stated regularity of $w_{\tau,k}$ and $u$, we have
\begin{equation}\label{tauk comparison}
	w_{\tau,k} \leq u \qquad \text{in}\;\; \R^N \times (\tau,T) \, .
\end{equation}
Next, we observe that  both $w_{\tau,k}$ and $u\vert_{t\in(\tau, T)}$ (thanks to the Preliminary Claim) belong to the uniqueness class of Proposition \ref{mp-uniq}, so they coincide with the constructed solutions, in the sense of Proposition \ref{mp-exist}, which take the same initial data. As a result, we may apply the $L^1(\Phi_\alpha)$ stability estimate \eqref{l1-weight-contr} to obtain
\begin{equation*}
	\norm{u(t)-w_{\tau,k}(t)}_{L^1(\Phi_\alpha)}\leq \exp\left(C_4\,(t-\tau)^{\theta\lambda}\right) \left\| u(\tau)-\phi_{2^{k}}\,u(\tau)\right\|_{L^1(\Phi_\alpha)} \qquad \forall t\in(\tau,\mathsf{T}_r) \, ,
\end{equation*}
where $C_4$ depends only on $N,m,\gamma,\underline{C},\overline{C},\alpha,r,\norm{u(\tau)}_{1,r}$, and we let $ \mathsf{T}_r $ denote the analogue of $ \mathsf{T} $ with $ T(u(\tau)) $ replaced by $ T_r(u(\tau)) $. Recalling the support properties of $ \phi_{2^k} $ and that by \eqref{tauk comparison} these solutions are ordered, the previous inequality can be written as follows:
\begin{equation}\label{tauk contraction}
	\int_{\R^N}\left[u(t)-w_{\tau,k}(t)\right]\Phi_\alpha \, \rho\,dx\leq \exp\left(C_4\,(t-\tau)^{\theta\lambda}\right)\int_{B^c_{2^{k}}}u(\tau)\,\Phi_\alpha \, \rho\,dx \qquad \forall t\in(\tau,\mathsf{T}_r) \, .
\end{equation}

\smallskip
\noindent{\textbf{Step 2: Compactness.}}
First we notice that $u$ has, by Theorem \ref{thm:weighted-AC}, a unique non-negative Radon measure $\mu$ as its initial trace, in the sense of \eqref{measure-data}.
In the following two steps, we will pass to the limit (along some sequence) as $\tau\to0$ in \eqref{tauk comparison}--\eqref{tauk contraction}, and find a solution $w_k$ of problem \eqref{measure}, in the sense of Definition \ref{def3}, which takes as its initial datum $\phi_{2^k}\,\mu$.
At this stage, the compactness tool that provides existence is Lemma \ref{alaa compact1}. To this end, we must consider the family (indexed by $\tau$) defined by the time-shifted solutions
\begin{equation*}
    \mathtt{w}_{\tau,k}(t)=w_{\tau,k}(\tau+t)\,,
\end{equation*}
which are defined for all $t\geq0$. Let us check that the leftmost hypothesis in \eqref{L1-bound} is satisfied by the family $\left\{\mathtt{w}_{\tau,k}\right\}_{\tau\in(0,T/8)}$, the rightmost one being true by construction. Indeed, we have
\begin{equation}\label{EBB}
    \sup_{\tau\in(0,T/8)}\norm{\mathtt{w}_{\tau,k}(0)}_{L^1_\rho\left(\mathbb{R}^N\right)}=\sup_{\tau\in(0,T/8)}\int_{B_{2^k}}u(\tau)\,\rho\,dx<+\infty\,,
\end{equation}
the boundedness being due to the Aronson--Caffarelli-type estimate \eqref{AC unif bound-2} of Corollary \ref{small time AC lem}.
Therefore, we are in a position to apply Lemma \ref{alaa compact1} to find a sequence $\tau_j\to0$ and a corresponding limit $w_k$ such that
\begin{equation}\label{appl conv}
	\begin{aligned}
    \mathtt{w}_{\tau_j,k} \to w_k \qquad & \text{in}\;\; C\!\left((0,+\infty);L^1_\rho\!\left(\mathbb{R}^N\right)\right) , \\
		\mathtt{w}_{\tau_j,k}^m \rightharpoonup  {w}_{k}^m \qquad & \text{in} \;\; L^2_{\mathrm{loc}}\big((0,+\infty);\dot H^1\!\left(\R^N\right)\!\big) \, ,
	\end{aligned}
\end{equation}
as $j \to \infty$. The above convergence is nearly enough to conclude that $w_k$ is solution to \eqref{measure} in the sense of Definition \ref{def3}. The only hypothesis we still lack is that $w_k\in L^\infty\!\left(\R^N\times(s,+\infty)\right)$ for all $s>0$ (apart from the initial trace -- we will address it in the next step). On the other hand, using the smoothing estimate $\eqref{basic smoothing}$ at the level of $\mathtt{w}_{\tau,k}$ yields
\begin{equation*}
    \sup_{t\in(s,+\infty)}\norm{\mathtt{w}_{\tau,k}(t)}_{L^\infty\left(\R^N\right)}\leq K \,s^{-\lambda}\norm{\mathtt{w}_{\tau,k}(0)}^{\theta\lambda}_{L^1_\rho\left(\mathbb{R}^N\right)} =K\,s^{-\lambda} \left(\int_{B_{2^k}}u(\tau)\,\rho\,dx\right)^{\theta\lambda} ,
\end{equation*}
where the right-hand side is finally bounded in $\tau$ owing to \eqref{EBB}. Hence, passing to the limit as $\tau \to 0$ guarantees the missing hypothesis.

Also, we observe that the stated convergence \eqref{appl conv} allows us to pass to the limit in \eqref{tauk comparison} to obtain
\begin{equation}\label{k comparison}
	w_{k}(t) \leq u(t) \qquad \text{a.e.\ in}\;\;\R^N\,, \ \text{for all}\;\;t\in(0,T)\,.
\end{equation}
In fact, by the previous construction and the ordering statement of Proposition \ref{pro1}, we even have
\begin{equation}\label{mono-comp}
	w_k(t)\leq w_{k+1}(t)  \qquad \text{a.e.\ in}\;\;\R^N\,, \ \text{for all}\;\;t\in(0,+\infty)\,.
    \end{equation}

\smallskip
\noindent{\textbf{Step 3: Approximate initial trace and uniqueness.}}
Now we must show that the initial trace of $w_k$ is $\phi_{2^{k}}\,\mu$, in the sense of \eqref{measure-data}. To this end, it is enough to proceed exactly as in the proof of Theorem \ref{mBCP}, see in particular the computations from \eqref{eps-claim} onward. The only differences are first in the fact that the convolution in \eqref{eps-claim} is replaced by  $\phi_{2^k} \, u(\tau) \, \rho$, which clearly converges to $\phi_{2^k} \, \mu$ as $\tau\to0$ in the sense of measures. Besides, the role of the smoothing estimate \eqref{appl key estimate bis} is taken here by \eqref{basic smoothing}, upon noticing that the time dependence of the two formulas is identical. In particular, one easily proves that
\begin{equation}\label{time continuity}
\left\vert\int_{\mathbb R^N}w_k(t)\,\varphi\,\rho\,dx-\int_{\mathbb R^N} \varphi \,\phi_{2^k}\,d\mu\right|\le C \, S_T^{\theta\lambda(m-1)+1} \, t^{\theta\lambda} \qquad \forall t\in \left( 0 , T_r\!\left(\cdot,S_T\right) \right) ,
\end{equation}
where
\begin{equation}\label{def-S-T}
S_T = \sup_{\tau \in (0,T/8)}\norm{u(\tau)}_{1,r}
\end{equation}
and $C>0$ is a constant that depends only on $ N,m,\gamma,\underline{C}, \overline{C}, R, \| \Delta \varphi \|_{L^\infty(B_R)}$. Note that $ S_T $ is finite still by virtue of Corollary \ref{small time AC lem}, so $ T_r\!\left(\cdot,S_T\right) $ is also well defined according to the same notation used in the proof of Lemma \ref{alaa compact2}. Therefore, letting $ t \to 0 $ in \eqref{time continuity} we infer that $ \phi_{2^k} \, \mu$ is indeed the initial trace of $w_k$.
Hence, by the stated uniqueness Theorem \ref{weighted pierre thm}, we find that $w_k$ \emph{depends only on the initial measure $\phi_{2^k}\,\mu$} and not on its specific method of construction via $u$.

\smallskip
\noindent{\textbf{Step 4: Compactness II and identification of the initial trace.}}
We shall now take $ k \to \infty$ on the sequence $\{ w_k \}$. First of all, we note that
\begin{equation*}
    \left\{w_k\right\} \; \; \text{is bounded in}\;\; Y_{\mathrm{loc}}\!\left((0,T)\right) ,
\end{equation*}
which in fact easily follows from \eqref{k comparison} and the Preliminary claim. Next, we intend to apply the compactness Lemma \ref{alaa compact2} to the family $\left\{w_k\right\}_{k\in\mathbb{N}}$. To this end, we observe that such solutions are strong energy, as follows by noting that this holds for the approximating solutions $\mathtt{w}_{\tau_j,k}$, and that the energy estimates of Lemma \ref{energy lem} with $ u \equiv \mathtt{w}_{\tau_j,k} $ are stable as $j\to\infty$.
Therefore, Lemma \ref{alaa compact2} is indeed applicable, ensuring the existence of a subsequence $k_j\to \infty$ and a limit
\begin{equation}\label{claim-w}
    w\in C\!\left((0,T);L^1(\Phi_\alpha)\right)\cap L^2_{\mathrm{loc}}\!\left((0,T);H^1_{\mathrm{loc}}\!\left(\R^N\right)\right)
\end{equation}
such that
\begin{equation}\label{strong conv}
\begin{aligned}
	w_{k_j}\to w \qquad & \text{in}\;\;C\!\left((0,T);L^1(\Phi_\alpha)\right) ,  \\
    w_{k_j}^m \rightharpoonup w^m \qquad & \text{in}\;\; L^2_{\mathrm{loc}}\!\left((0,T);H^1_{\mathrm{loc}}\!\left(\R^N\right)\right),\\
    \end{aligned}
\end{equation}
as $j\to\infty$. This is enough to determine that $w$ is a solution to \eqref{wpme-noid} in the sense of Definition \ref{def:vw-gl}. As by construction $w\geq0$, we know from Theorem \ref{thm:weighted-AC} that $w$ has an initial trace $\mu_w$. Furthermore, by the previous two steps, we have that $w$ \emph{depends only on the initial measure $\mu$}, also recalling the monotonicity property \eqref{mono-comp}, which makes sure that $ w $ is actually independent of the specific subsequence $ \{ w_{k_j} \} $.

Finally, we can let $k \to \infty$ in \eqref{k comparison} to obtain
\begin{equation}\label{comparisonvw}
w(t) \leq u(t) \qquad \text{a.e.\ in}\;\;\R^N\,, \ \text{for all}\;\;t\in(0,T)\,.
\end{equation}
Passing to the limit in \eqref{time continuity}, first as $k \to \infty$ then as $ t \to 0 $, which is justified by the above steps, we obtain
\[
\lim_{t\to0}\int_{\R^N}w(t)\,\varphi\,\rho\,dx=\int_{\R^N}\varphi\,d\mu
\]
for all $\varphi\in C_c^\infty\!\left(\R^N\right)$, so that $\mu_w=\mu$.

\smallskip
\noindent{\textbf{Step 5: Conclusion.}}
So far we have constructed a non-negative solution $w$ to \eqref{measure}, in the sense of Definition \ref{def:vw-gl-m}, which takes the same initial datum as $u$ (and $v$ thanks to \eqref{same data}). Moreover, as noted above, $w$ only depends on the initial datum $\mu$.
Therefore, if we can show that $u=w$, it would follow just in the same way that $v=w$, which would conclude the proof.

Hence, let us prove that $u=w$. To this end, we will pass to suitable limits in \eqref{tauk contraction}. Indeed, letting first $\tau\to0$ and then $k\to\infty$ (ignoring the subindices), we have
\begin{equation}\label{contraction}
\begin{gathered}
	\int_{\R^N}\left[u(t)-w(t)\right]\Phi_\alpha \, \rho\,dx\leq \limsup_{k \to \infty} \, \limsup_{\tau \to 0} \, \exp\left(C_4 \, T^{\theta\lambda}\right)\int_{B^c_{2^{k}}}u(\tau)\,\Phi_\alpha \, \rho\,dx \\ \forall t\in\left(0, T_r(\cdot,S_T) \wedge T \right) ,
    \end{gathered}
\end{equation}
where $S_T$ is defined in \eqref{def-S-T} and we are tacitly using the constant $C_4$ as in \eqref{tauk contraction} with $ \| u(\tau) \|_{1,r} $ replaced by $S_T$, so as to remove dependence on $\tau$. Note that the passages to the limit on the left-hand side of \eqref{contraction} are justified thanks to \eqref{appl conv} and \eqref{strong conv}. Now we need to show that the integral on the right-hand side of \eqref{contraction} converges to $0$ as $k\to\infty$, uniformly in $\tau$. On the other hand, this is a simple consequence of the following estimate:
\begin{equation}\label{est-series}
	\begin{aligned}
		\int_{B^c_{2^{k}}} u(\tau) \, \Phi_\alpha\,\rho\,dx&=\sum_{i=k}^{\infty}\int_{B_{2^{i+1}}\setminus B_{2^{i}}} u(\tau) \, \Phi_\alpha\,\rho\,dx\\
		&\leq S_T \, \sum_{i=k}^{\infty}\left(1+2^{2i}\right)^{-\alpha}\,2^{(i+1)\frac{(N-\gamma)(m-1)+2-\gamma}{m-1}} \, ,
	\end{aligned}
\end{equation}
where we are implicitly assuming $ \tau \in (0,T/8) $ and that the $ \| \cdot \|_{1,r} $ norm is taken for some fixed $ r < 2^k $. Due to the lower bound on $\alpha$ in \eqref{alpha-cond} and Corollary \ref{small time AC lem}, we can deduce that the left-hand side of \eqref{est-series} vanishes as $ k \to \infty $, uniformly in $ \tau \in (0,T/8) $. In conclusion, we have established that the right-hand side of \eqref{contraction} is indeed zero, which in combination with \eqref{comparisonvw} entails
\begin{equation*}
w(t) = u(t) \qquad \text{a.e.\ in}\;\;\R^N\,, \ \text{for all}\;\; t\in\left(0, T_r(\cdot,S_T) \wedge T \right) .
\end{equation*}
If $ T_r(\cdot,S_T) \ge T $ the proof is complete; otherwise, it is enough to observe that due to \eqref{claim-U}, \eqref{claim-w}, and \eqref{comparisonvw} both $ u\vert_{t \in (s,T)} $ and $ w\vert_{t \in (s,T)} $ fall in the same uniqueness class of Proposition \ref{mp-uniq} for any $ s \in \left(0,T_r(\cdot,S_T)\right) $.
\end{proof}

\section{\emph{A priori} local boundedness: proofs}\label{sec:loc-bdd}
In this section we prove Theorem \ref{apriori-bounded}, namely that non-negative solutions to \eqref{wpme-noid-loc} are locally bounded. In order to simplify the discourse, throughout we will only consider \emph{global solutions} to \eqref{wpme-noid}, but by minor modifications the same result can be obtained for local solutions to \eqref{wpme-noid-loc} (see Remark \ref{LS} below). Specifically, we will carefully modify the strategy of \cite{DK2} to handle the weight $\rho$; the spatial inhomogeneity of the equation calls for a more general approach that, in the final arguments of our proof, involves regularization in \emph{time} only.

Let us begin by observing that very weak solutions to \eqref{wpme-noid} are also solutions of local problems in $B_R \times(0,T)$, at least for a.e.\ $R>0$.
That is, we have
\begin{equation}\label{loc-prob}
\rho \, u_t = \Delta\!\left(u^m \right)\qquad \text{in}\; \; B_R\times(0,T) \, ,
\end{equation}
in the sense that
\begin{equation}\label{vw-loc}
\int_{0}^{T} \int_{B_R} \left(u \, \psi_t \, \rho + u^m \, \Delta\psi \right) dx dt = \int_{0}^{T} \int_{\partial B_R} u^m \, \partial_{\Vec{\mathsf n}}\psi \, d\sigma dt
\end{equation}
for all $\psi\in C^\infty_c\!\left(\overline{B}_R\times(0,T)\right)$ that vanish on $\partial B_R\times(0,T)$. In fact, this is a direct consequence of Lemma \ref{globloc}.

Following the overall strategy of \cite{DK2}, but with significant technical obstructions coming from the presence of the weight, we will prove Theorem \ref{apriori-bounded} by studying \emph{local potentials} of $u$; more precisely, for all $x\in B_R$ and a.e.\ $t\in(0,T)$ we define
\begin{equation}\label{pot-u}
w(x,t)=\int_{B_R} G_R(x,y)\,u(y,t)\,\rho(y)\,dy \, ,
\end{equation}
where $G_R $ is the Green's function of $-\Delta$ in $B_R$ with homogeneous Dirichlet boundary conditions. One can easily see, for instance estimating $G_R$ by the global Green's function of $-\Delta$ and using Young's convolution inequality, that
\begin{equation*}\label{q-w}
w(t)\in L^q(B_R)  \qquad\text{for all}\;\; 1 \le q<\tfrac{N}{N-2} \, .
\end{equation*}
We refer \emph{e.g.}~to \cite[Th\'eor\`eme 9.1]{Stamp} for a detailed proof. By integrating \eqref{pot-u} in time and using the same argument, together with the space-time integrability of $u$, we also get
\begin{equation}\label{q-w-2}
\int_0^tw(s)\,ds\in L^q(B_R) \qquad\text{for all}\;\; 1 \le q<\tfrac{N}{N-2} \, .
\end{equation}
Next, we define the boundary correction function
\begin{equation*}\label{def-h}
h(x,t)=-\int_0^t\int_{\partial B_R}\partial_\varrho G_R(x,y)\, u^m(y,s)\,d\sigma(y)\,ds \qquad \text{in}\;\; B_R \times (0,T) \, ,
\end{equation*}
where $\partial_\varrho \equiv \partial_{\Vec{\mathsf n}} $ is the outer normal (radial) derivative on $ \partial B_R $. Here and in the sequel, without loss of generality, we are assuming that $u\in L^1_{\rho,\mathrm{loc}}\!\left(\mathbb R^N\times [0,T)\right)$ and $u^m\in L^1_{\mathrm{loc}}\!\left(\mathbb R^N\times [0,T)\right)$; this is achieved by a routine time shift, and it is not restrictive since we aim at proving \emph{local} boundedness.
Note that, by construction, $h(t)$ is non-negative, harmonic, and integrable in $B_R$ for all $t\in(0,T)$. More precisely, there exists a constant $C_{N,R}>0$ such that
\[
\int_{B_R}h(x,t)\,dx \le C_{N,R} \, \int_{\partial B_R}\int_0^t u^m\,d\sigma ds \, ,
\]
which, in combination with the mean-value principle for harmonic functions, entails that $ h \in L^\infty_{\mathrm{loc}}\!\left( B_R \times [0,T) \right) $ (see also \cite[Lemma 4.5]{DK2}). As a consequence, we have
$$
W=w-h\in L^1\!\left( B_R \times (0,T-\tau) \right) \qquad \text{for all}\; \; \tau \in (0,T) \, .
$$
Furthermore, the function $W$ satisfies the following dual version of the weighted porous medium equation:
\begin{equation}\label{dual-eq}
\begin{cases}
	-\Delta W =  u \, \rho &\text{in} \;\; B_R\times(0,T)\, ,\\
	-W_t =  u^m  & \text{in} \;\; B_R\times(0,T)\,,
\end{cases}
\end{equation}
in the sense that
\begin{equation}\label{dual-eq-ell-vws}
\int_{0}^{T}\int_{B_R}\left(W\,\Delta \psi+u\,\psi\,\rho\right) dxdt = 0
\end{equation}
and
\begin{equation}\label{dual-eq-par-vws}
\int_{0}^{T}\int_{B_R}\left(W\,\psi_t-u^m\,\psi\right) dx dt  =0 \, ,
\end{equation}
for all $\psi\in C^\infty_c\!\left(B_R\times(0,T)\right)$. Identity \eqref{dual-eq-ell-vws} is a direct consequence of the definition of $w$ and the harmonicity of $h$, whereas \eqref{dual-eq-par-vws} can be rigorously proved by replacing in \eqref{vw-loc} a test function $\psi$ with its Dirichlet potential $ x \mapsto \int_{B_R} G_R(x,y) \, \psi(y,t) \, dy $.

\begin{lem}\label{pot-bounded}
Let $u$ be a non-negative local solution to \eqref{loc-prob}, in the sense of \eqref{vw-loc}, and let $W$ be the corresponding solution to \eqref{dual-eq} constructed above, in the sense of \eqref{dual-eq-ell-vws}--\eqref{dual-eq-par-vws}. Then $W\in L^\infty_{\mathrm{loc}}\!\left(B_R\times(0,T)\right)$.
\end{lem}
\begin{proof} \, \\[0.1cm]
\noindent{\textbf{Preparation:}} As a first step, in \eqref{dual-eq-par-vws}, let us take $\psi(x,t) = \varphi_n(x)\,\eta_n(t)$, where $ \{ \varphi_n \}$ is a smooth approximation of $\delta_{y}$ for some $y\in B_R$ and $\{\eta_n\}$ suitably converges to $\chi_{(t_1,t_2)}$ as $n \to \infty$, for some $0<t_1<t_2<T$. Taking the limit, it follows that
\begin{equation}\label{ae W}
	W(y,t_2)-W(y,t_1)=-\int_{t_1}^{t_2}u^m(y,t)\,dt \le 0 \, ,
\end{equation}
for a.e.\ $t_1,t_2\in(0,T)$, with $t_1<t_2$, and for a.e.\ $y\in B_R$. Note that \eqref{ae W} actually holds for \emph{all} Lebesgue points of $W\in L^1_{\mathrm{loc}}\!\left((0,T);L^1(B_R)\right)$. Let us then denote by $\mathcal{T}\subset(0,T)$ the set of such times.

Because $u\geq0$, it follows that $t \mapsto W(t)$ is essentially non-increasing. However, for our purposes, it will be convenient for $W(t)$ to be defined and continuous with values in $L^1(B_R)$ and non-increasing \emph{for all times} $t\in(0,T)$. To this end, it is enough to define
\begin{equation}\label{def-completion}
	\overline{W}(t)=
	W(t_1)-\int_{t_1}^{t}u^m(s) \, ds \, ,
\end{equation}
where $t_1$ is any fixed element of $\mathcal{T}$. From \eqref{ae W} we get that such a function is defined for all $t\in(0,T)$, coincides with $W$ a.e.\ in $ B_R\times(0,T)$, and the mapping $t\mapsto \overline{W}(t)$ is continuous in $L^1(B_R)$ and essentially non-increasing, in the sense that $ \overline{W}(x,t_2) \le \overline{W}(x,t_1) $ for a.e.\ $ x \in B_R $ (possibly depending on $ t_1,t_2 $) and for all $ 0<t_1<t_2<T $. Clearly, from \eqref{def-completion} and the integrability of $ u^m $, we have $\overline{W}\in C\!\left((0,T);L^1(B_R)\right)$, and $\overline{W}$ continues to satisfy \eqref{dual-eq-ell-vws}--\eqref{dual-eq-par-vws}. For the rest of the proof, we will implicitly work with $\overline{W}$ in the place of $W$ (but without changing notations).

Let us fix any $ 0 < t_\infty<S<T$ and $0<R_\infty<R$. Then we take any monotone sequences $\{t_0,t_1,\cdots\}\subset(0,T)$ with $t_k\nearrow t_\infty$ and $\{R_1,R_2,\cdots\}\subset(R_\infty,R)$ with $R_k\searrow R_\infty$. We aim at proving that $ W \in L^\infty\!\left(B_{R_\infty} \times (t_\infty,S) \right) $, which implies the thesis due to the arbitrariness of $R_\infty,t_\infty,S$. From here on, in addition to the continuity of $W$ in $ L^1(B_R) $, we will also tacitly assume that $ W \ge 0 $. In general, this is not guaranteed, but it can be achieved by adding to $W$ any constant $c$ such that $ c \ge h $ in $ B_{R_1} \times (t_0,S) $, which is possible thanks to the local boundedness of $h$.
\smallskip
\\\noindent{\textbf{Base case:}}
We begin our considerations by noticing that
\begin{equation}\label{s0-int}
	\int_{t_1}^{t_2}u^m(t)\,dt=W(t_1)-W(t_2)\leq_{(a)} W(t_1)\leq_{(b)} (t_1-t_0)^{-1}\int_{t_0}^{t_1}W(t)\,dt \in L^{q}\!\left(B_{R_1}\right) ,
\end{equation}
for the same exponents $q$ as in \eqref{q-w-2} (recalling that $ h $ is locally bounded), where in $(a)$ and $(b)$ we have used the above observed non-negativity and monotonicity of $W$, respectively. Hence,
\begin{equation}\label{s1-int}
	\int_{t_1}^{t_2}u(t)\,dt \leq  \left(\int_{t_1}^{t_2}u^m(t)\,dt\right)^{1/m}(t_2-t_1)^{\frac{m-1}{m}}\in L^{mq}\!\left(B_{R_1}\right) .
\end{equation}
In order to perform a bootstrap argument via elliptic regularity, we now look for a $p>1$ such that
\begin{equation}\label{bootstrap}
	\left( \int_{t_1}^{t_2}u(t)\,dt \right) \rho \in L^p\!\left(B_{R_1}\right) .
\end{equation}
To this end, upon applying H\"{o}lder's inequality we obtain
\begin{align*}
	\int_{B_{R_1}}\left(\int_{t_1}^{t_2} u(x,t)\,dt\right)^p [\rho(x)]^p \, dx \leq & \left[\int_{B_{R_1}}\left(\int_{t_1}^{t_2} u(x,t)\,dt\right)^{mq}dx\right]^{\frac{p}{mq}}\\ &\times\left[\int_{B_{R_1}}[\rho(x)]^{\frac{pmq}{mq-p}}\,dx\right]^{\frac{mq-p}{mq}} .
\end{align*}
As a consequence of \eqref{weight-cond} and \eqref{s1-int}, the right-hand side of the above formula is finite provided
\begin{equation}\label{eq-gammapmq}
	\frac{\gamma pmq}{mq-p}<N \, ;
\end{equation}
since $q<N/(N-2)$, this is always achievable if
\begin{equation}\label{p0cond}
	p<\frac{Nm}{N-2+\gamma m} \, .
\end{equation}
Crucially, since $ m>1$ and $\gamma\in(0,2)$, we may indeed take $p>1$ in \eqref{p0cond}.
In conclusion, we have \eqref{bootstrap} for all $p>1$ complying with \eqref{p0cond}. In particular, exploiting again the time monotonicity of $W$, we can assert that there exists a $p_1>1$ such that
\begin{equation*}
	\int_{t_1}^{t_2}W(t)\,dt \in L^{p_1}\!\left(B_{R_1}\right) \qquad\text{and}\qquad \left(\int_{t_1}^{t_2}u(t)\,dt\right) \rho\in L^{p_1}\!\left(B_{R_1}\right) .
\end{equation*}
\smallskip
\noindent{\textbf{Induction step:}} Let us now iterate the above arguments at a general step $k\in\mathbb{N} \setminus \{ 0 \} $. Our induction hypothesis is that
\begin{equation}\label{ind-hypo}
	\int_{t_k}^{t_{k+1}}W(t)\,dt \in L^{p_k}\!\left(B_{R_k}\right) \qquad\text{and}\qquad \left( \int_{t_k}^{t_{k+1}} u(t)\,dt \right) \rho \in L^{p_k}\!\left(B_{R_k}\right) ,
\end{equation}
for some $p_k>1$. By the base case, we know that such an assertion is true for $k=1$. Now we aim to quantitatively improve the above integrability by resorting to elliptic regularity, noticing that
\begin{equation*}\label{int-eq}
	-\Delta \left( \int_{t_k}^{t_{k+1}}W(t)\,dt \right)=\left(\int_{t_k}^{t_{k+1}}u(t)\,dt \right) \rho \, ,
\end{equation*}
at least in the sense of distributions (an immediate consequence of \eqref{dual-eq-ell-vws}).

Now, without loss of generality, we may assume that $p_k<{N}/{2}$. If, on the other hand, $p_k>{N}/{2}$, then an application of the Calder\'{o}n--Zygmund elliptic estimates in conjunction with a second-order Sobolev embedding (see \cite[Theorem 9.11]{GT} and \cite[Theorem 6 in Section 5.6]{Evans}, respectively) immediately gives that
\begin{equation*}\label{hold-est}
	\int_{t_k}^{t_{k+1}}W(t)\,dt \in C^{\alpha}\!\left(B_{R_{k+1}}\right),
\end{equation*}
for some $\alpha\in(0,1)$. By the monotonicity of $W$, we can therefore infer that
$$
W\in L^\infty\!\left(B_{R_{k+1}}\times(t_{k+1},S)\right) \subset L^\infty\!\left(B_{R_{\infty}}\times(t_{\infty},S)\right) ,
$$
whence the thesis follows. In the borderline case $p_k={N}/{2}$, by the critical Sobolev embedding we have
\begin{equation*}\label{crit-sob}
	\int_{t_k}^{t_{k+1}}W(t)\,dt \in L^q\!\left(B_{R_{k+1}}\right) \qquad \forall q\in(1,\infty) \, .
\end{equation*}
As before, using \eqref{dual-eq-par-vws} along with the monotonicity of $W$ gives
\begin{align*}
	\int_{t_{k+1}}^{t_{k+2}}u^m(t)\,dt&=W(t_{k+1})-W(t_{k+2})\leq W(t_{k+1})\\&\leq(t_{k+1}-t_{k})^{-1}\int_{t_k}^{t_{k+1}}W(t)\,dt
	\in L^q\!\left(B_{R_{k+1}}\right) \qquad\forall q\in(1,\infty) \, .
\end{align*}
Proceeding as in the base case, we conclude that
$$
\left( \int_{t_{k+1}}^{t_{k+2}} u(t)\,dt \right) \rho \in L^p\!\left(B_{R_{k+1}} \right)
$$
provided \eqref{eq-gammapmq} holds,
which, recalling that here $q>1$ can be taken arbitrarily large, is equivalent to
\begin{equation*}
	p<\frac{N}{\gamma} \, .
\end{equation*}
Since $\gamma<2<N$, we may take $p_{k+1}=p>{N}/{2}$, so again applying the monotonicity of $W$, after one additional step we end up with \eqref{ind-hypo} for $k \equiv k+1$ and a supercritical exponent $p_{k+1}>{N}/{2}$.

Let us now return to the subcritical case $ p_k < N/2 $. In this range, Calder\'{o}n--Zygmund estimates, second-order Sobolev embeddings, and the monotonicity of $W$ yield
\begin{equation}\label{is w-int}
	W(t_{k+1})\leq(t_{k+1}-t_{k})^{-1}\int_{t_{k}}^{t_{k+1}} W(t)\,dt\in L^{q_{k+1}}\!\left(B_{R_{k+1}}\right),
\end{equation}
with
\begin{equation}\label{pk-cond}
	\frac{1}{q_{k+1}}=\frac{1}{p_k}-\frac{2}{N} \, .
\end{equation}
Similarly to \eqref{s0-int}--\eqref{s1-int}, from \eqref{is w-int} we get
\begin{equation*}
	\int_{t_{k+1}}^{t_{k+2}}u(t)\,dt \in L^{mq_{k+1}}\!\left(B_{R_{k+1}}\right) \subset L^{\hat{m}q_{k+1}}\!\left(B_{R_{k+1}}\right) \qquad \forall \hat{m} \in (1,m] \, .
\end{equation*}
For the following arguments, it will be convenient to take
$$
\hat{m}=\min\!\left\{\tfrac{m+1}{2} , \, \tfrac{N}{N-\gamma} , \, \tfrac{2}{\gamma}\right\} .
$$
As in the base case, we need to find a suitable $p_{k+1}>1$ such that
\begin{equation}\label{is u-int}
	\left( \int_{t_{k+1}}^{t_{k+2}} u(t) \, dt \right) \rho \in L^{p_{k+1}}\!\left(B_{R_{k+1}}\right)
\end{equation}
and $p_{k+1}\leq q_{k+1}$, the latter condition being necessary to bootstrap the argument involving elliptic regularity. Again, by H\"{o}lder's inequality, we have
\begin{align*}
	\int_{B_{R_{k+1}}}\left(\int_{t_{k+1}}^{t_{k+2}} u(x,t)\,dt\right)^{p_{k+1}}[\rho(x)]^{p_{k+1}}\,dx \leq & \left[\int_{B_{R_{k+1}}}\left(\int_{t_{k+1}}^{t_{k+2}} u(x,t)\,dt\right)^{mq_{k+1}}dx\right]^{\frac{p_{k+1}}{mq_{k+1}}}\\
&\times\left(\int_{B_{R_{k+1}}}[\rho(x)]^{\frac{p_{k+1}mq_{k+1}}{mq_{k+1}-p_{k+1}}}\,dx\right)^{\frac{mq_{k+1}-p_{k+1}}{mq_{k+1}}},
\end{align*}
so we must require
\begin{equation*}
	\frac{\gamma\,p_{k+1}\,m\,q_{k+1}}{m\,q_{k+1}-p_{k+1}}<N \, ,
\end{equation*}
or equivalently, using \eqref{pk-cond},
\begin{equation*}
	\frac{1}{p_{k+1}}>\frac{1}{m\,q_{k+1}}+\frac{\gamma}{N}=\frac{1}{m\,p_k}-\frac{2}{Nm}+\frac{\gamma}{N}=\frac{1}{m}\left(\frac{1}{p_k}-\frac{2}{N}\right)+\frac{\gamma}{N} \, .
\end{equation*}
In particular, we can choose $p_{k+1}$ satisfying
\begin{equation}\label{p-step cond}
	\frac{1}{p_{k+1}}=\frac{1}{\hat{m}\,q_{k+1}}+\frac{\gamma}{N}=\frac{1}{\hat{m}\,p_k}-\frac{2}{N\hat{m}}+\frac{\gamma}{N}=\frac{1}{\hat{m}}\left(\frac{1}{p_k}-\frac{2}{N}\right)+\frac{\gamma}{N} \, .
\end{equation}
We next need to check that $p_{k+1}\leq q_{k+1}$, which is implied by
\begin{equation*}
	\hat{m}\leq\frac{N}{N-\gamma} \, .
\end{equation*}
Therefore, with the above choices, we have \eqref{is u-int} for $p_{k+1}$ satisfying \eqref{p-step cond} and $p_{k+1}\leq q_{k+1}$. This completes the induction step; indeed, applying \eqref{is w-int} and the monotonicity of $W$ as before, we obtain exactly \eqref{ind-hypo} for $k\equiv k+1$.
\smallskip
\\ \noindent{\textbf{Conclusion:}}
It is now left to study the behavior of $\{p_k\}$ as $k\to\infty$. Rewriting \eqref{p-step cond}, we have
\begin{equation*}
	\frac{1}{p_{k+1}}=\frac{1}{\hat{m}\,p_k}+\frac{\gamma}{N}-\frac{2}{N\hat{m}} \, .
\end{equation*}
As by definition $\hat{m}\leq {2}/{\gamma}$, it follows that
\begin{equation*}
	p_{k+1}\geq\hat{m}\,p_{k} \, ,
\end{equation*}
so $\{ p_k \}$ has at least exponential growth as $k\to\infty$. Therefore, in a finite number $K$ of iterations, we necessarily have $p_{K}\geq{N}/{2}$, and the thesis follows according to the argument above involving a supercritical or critical Sobolev embedding in one or two additional steps, respectively.
\end{proof}

\begin{cor}\label{pot-bounded-2}
Let $u$ be a non-negative solution to \eqref{loc-prob}, in the sense of \eqref{vw-loc}, and let $w$ be defined by \eqref{pot-u}. Then $w\in L^\infty_{\mathrm{loc}}\!\left(B_R\times(0,T)\right)$.
\end{cor}
\begin{proof} We have already noticed that $ h \in L^\infty_{\mathrm{loc}}\!\left( B_R \times [0,T) \right) $, so the statement follows by the very definition of $w$ and the local boundedness of $W$ proved in Lemma \ref{pot-bounded}.
\end{proof}

\begin{cor}\label{coro-conv}
Let $u$ be a non-negative solution to \eqref{loc-prob}, in the sense of \eqref{vw-loc}. Then, for all $ \tau \in (0,T/2) $, we have
\begin{equation}\label{bdd-conv-1}
(x,t) \mapsto \int_{t-\frac \tau 2}^{t+\frac \tau 2} u^{m}(x,s) \, ds  \in L^\infty_{\mathrm{loc}}\!\left(B_R \times (\tau,T-\tau)\right)
\end{equation}
and
\begin{equation}\label{bdd-conv-2}
 (x,t) \mapsto \int_{t-\frac \tau 2 }^{t+\frac \tau 2} u(x,s) \, ds  \in L^\infty_{\mathrm{loc}}\!\left(B_R \times (\tau,T-\tau)\right)  .
\end{equation}
\end{cor}
\begin{proof}
 Let us take any $ R' \in (0,R)$. Following the notations of Lemma \ref{pot-bounded}, from \eqref{s0-int} (with suitable choices of $ t_1,t_2 $) we infer
$$
\begin{gathered}
 \int_{t-\frac \tau 2}^{t+\frac \tau 2} u^{m}(x,s) \, ds \le W\!\left(x,t-\tfrac \tau 2\right) \le \left\| W \right\|_{L^\infty \left( B_{R'} \times \left(\frac{\tau}{2},T-\frac{3}{2} \tau\right) \right)} \\
 \text{for a.e.} \; \; (x,t) \in B_{R'} \times (\tau,T-\tau) \, ,
 \end{gathered}
$$
which settles \eqref{bdd-conv-1}. On the other hand, \eqref{s1-int} reads
$$
	 \int_{t-\frac \tau 2}^{t+\frac \tau 2} u(x,s) \, ds \leq  \left(\int_{t-\frac \tau 2}^{t+\frac \tau 2} u^m(x,s)\,ds \right)^{1/m}\tau^{\frac{m-1}{m}} \, ,
$$
therefore \eqref{bdd-conv-2} is a direct consequence of the first estimate.
\end{proof}

The following technical result, which we state here without proof, was established in \cite{DK2}.

\begin{lem}[Lemma 4.7 in \cite{DK2}]\label{tough-lem}
Let $w$ be a smooth function in a neighborhood of $Q=\Omega\times(t_1,t_2)$, where $\Omega\subset \mathbb R^N$ is a bounded smooth domain, which satisfies
\begin{equation*}\label{conds-on-w}
	w_t \leq 0 \quad \text{and} \quad \Delta w\leq0 \qquad \text{in} \;\; Q \, .
\end{equation*}
If $\int_\Omega |\Delta w(x,t)|\,dx\leq M$ and $|w(x,t)|\leq M$ for all $ (x,t) \in Q $ for some constant $M>0$, then for every compact set $ K  \Subset Q $ there exists a constant $C>0$, depending on $K,Q$, such that
\begin{equation}\label{tough-estimate}
	\iint_{K} w_t \, \Delta w \, dxdt \leq C M^2 \, .
\end{equation}
\end{lem}

\begin{lem}\label{imp-int}
Let $u$ be a non-negative solution to \eqref{loc-prob}, in the sense of \eqref{vw-loc}. Then $u\in L^{m+1}_{\rho,\mathrm{loc}}\!\left( B_R\times (0,T)\right)$.
\end{lem}
\begin{proof} We follow a similar procedure to the proof of \cite[Theorem 4.6]{DK2}, providing some detail for the reader's convenience. First of all, for a locally integrable function $f$ in $ \R^{N+1} $ and every $ \varepsilon>0 $, we set
\begin{equation*}\label{def-moll}
	T_\varepsilon [f](x,t)=\iint_{\R^{N+1}} f\!\left(x-\xi,t-s\right) \eta_\varepsilon(\xi,s)\,d\xi ds \, ,
\end{equation*}
where $ \left\{ \eta_\varepsilon \right\} \subset C^\infty_c\!\left(\R^{N+1}\right)$ is a standard family of mollifiers. Let $Q$ be as in Lemma \ref{tough-lem}, with $ \Omega=B_{R'} $ for an arbitrary $ R' \in (0,R) $ and arbitrary $0<t_1<t_2<T$. Let $0<\tau_1<\tau_2$ be such that $[t_1,t_2]\subset (\tau_1, \tau_2)$, with $\tau_1 $ a Lebesgue time for $ u $ in $ L^1_{\rho,\mathrm{loc}}\!\left(\R^N\right) $, and let
$$
v(x,t)=\int_{B_{{R}}} G_{{R}}(x,y)\, u (y,\tau_1)\, \rho(y)\,dy -
	\int_{\tau_1}^{t} u^m(x,s)\,ds \qquad \text{in} \;\; B_{{R}} \times(\tau_1,\tau_2) \, .
$$
It is readily seen that
$$
v_t = -u^m \qquad \text{and} \qquad \Delta v = - u \, \rho \, ,
$$
at least in the $ L^1(B_R \times (\tau_1,\tau_2)) $ sense. Therefore, it is plain that for all small enough $\varepsilon>0$ the function $ v_\varepsilon = T_\varepsilon[v] $ (with $v$ set to zero outside $ B_R \times (\tau_1,\tau_2) $) is a smooth in a neighborhood of $ B_{R'} \times (t_1,t_2) $ and satisfies
\begin{equation*}\label{eqs-for-v-moll}
	\partial_t v_\varepsilon = - T_\varepsilon[u^m] \leq 0 \quad \text{and} \quad \Delta v_\varepsilon  = - T_\varepsilon [u \, \rho] \leq 0 \qquad \text{in} \;\; B_{R'} \times (t_1,t_2) \, .
\end{equation*}
Also, it is readily seen from \eqref{vw-loc}, by testing with suitable cutoff functions, that
\begin{equation}\label{cutoff-L1}
\int_{B_{R'}} u(t) \, \rho \, dx \le \int_{B_{R}} u(\tau_1) \, \rho \, dx + C_{R,R'} \int_{\tau_1}^{\tau_2} \int_{B_R} u^m \, dxds  \qquad \text{for a.e.} \; \; t \in (\tau_1,\tau_2) \, ,
\end{equation}
where $ C_{R,R'} >0$ is a constant depending on $R,R'$.

Now we may apply Lemma \ref{tough-lem} to $w=v_\varepsilon $, upon noticing that the right-hand side of \eqref{tough-estimate} is stable as $ \varepsilon \to 0 $ in view of Corollaries \ref{pot-bounded-2}--\ref{coro-conv} and \eqref{cutoff-L1}. The thesis then follows by Fatou's lemma and the arbitrariness of $ R',t_1,t_2 $.
\end{proof}

We are now ready to prove Theorem \ref{apriori-bounded}, at least, in the case of global non-negative solutions.

\begin{proof}[Proof of Theorem \ref{apriori-bounded} (for global solutions)]
We split the argument into two steps: first, the existence of a locally bounded solution $\overline{u}$ that satisfies the same (localized) initial and boundary conditions as the given solution $u$, and second, the identification of these two solutions. To this end, we will work in a cylinder $B_r \times (\tau_1,\tau_2) $, for arbitrary (up to an a.e.\ constraint) {$r>1$} and $0<\tau_1<\tau_2<T$. We will show that $u$ is locally bounded in $B_r \times(\tau_1,\tau_2)$, whence the statement follows owing to the arbitrariness of $r,\tau_1,\tau_2$.

\smallskip
\noindent{\textbf{Step 1: Existence.}}
Similarly to what was done in the Proof of Lemma \ref{imp-int}, for a locally integrable function $f$ in $ \mathbb{R}^{N+1} $ and each $\varepsilon>0$ we set
\begin{equation*}\label{def-time-moll}
	S_\varepsilon [f] (x,t)=\int_{\R} f(x,t-s)\,\eta_\varepsilon(s)\, ds \, ,
\end{equation*}
where $ \left\{ \eta_\varepsilon \right\} \subset C^\infty_c(\R) $ is a smooth and non-negative approximation of the Dirac delta in $\mathbb{R}$, specifically, a family of mollifiers such that each $\eta_\varepsilon$ is supported in $ (-\varepsilon/2,\varepsilon/2) $. Let us apply such a time-regularizing operator to $ u $ and $ u^m $ (extended to zero outside $(0,T)$); from the definition of $\eta_\varepsilon$, it is plain that
$$
S_\varepsilon[u](x,t) \le \| \eta_\varepsilon \|_\infty \int_{t-\frac \varepsilon 2}^{t+\frac{\varepsilon}{2}} u(x,s) \, ds \qquad \text{and} \qquad S_\varepsilon[u^m](x,t) \le  \| \eta_\varepsilon \|_\infty \int_{t-\frac \varepsilon 2}^{t+\frac{\varepsilon}{2}} u^m(x,s) \, ds \, ,
$$
hence by Corollary \ref{coro-conv} (with any $R>r$) we immediately infer that
\begin{equation}\label{bdd-eps-loc}
S_\varepsilon[u] ,  S_\varepsilon[u^m]  \in L^\infty( B_r \times (\varepsilon_0 , T-\varepsilon_0) ) \, ,
\end{equation}
where from here on we assume $ \varepsilon < \varepsilon_0 $ for some fixed $ \varepsilon_0 \in (0,T/2) $.

We now consider the constructed solution $u_\varepsilon$ of problem \eqref{cd problem} in $B_r\times(\tau_1,\tau_2)$, in the sense of Definition \ref{vw cd def}, with initial datum $ S_\varepsilon [u](\tau_1)$ and boundary datum $ S_\varepsilon [u^m]|_{\partial B_r } = S_\varepsilon \!\left[u^m|_{\partial B_r }\right]$, where $\tau_1$ and $ r $ are almost arbitrarily chosen so that the corresponding traces of $u$ on $ \mathbb{R}^ N\times \{ \tau_1 \} $ and $ u^m $ on $ \partial B_r \times (0,T) $, respectively, are well defined, in agreement with Lemma \ref{globloc}. Also, we additionally require $ \varepsilon_0 $ to be so small that $ \varepsilon_0 < \tau_1 < \tau_2 < T-\varepsilon_0 $. Note that the existence and boundedness of such a solution is standard and can be shown exactly as in the end of proof of Theorem \ref{local moser iter lem} (all of the initial and boundary data are bounded thanks to \eqref{bdd-eps-loc}).

Next, we aim at showing that the family $ \{ u_\varepsilon \} $ is uniformly (with respect to $ \varepsilon $) locally bounded in $ B_r \times (\tau_1,\tau_2) $. By construction, each $u_\varepsilon$ is a strong energy solution in the sense of Definition \ref{def:w-loc}, so that the computations that follow are justified. First of all, we need to establish uniform boundedness in $ L^1_{\rho,\mathrm{loc}}( B_r \times (\tau_1,\tau_2) ) $; to this end, we reason as in \cite[Proposition 2.1]{DK2}. Let us plug into the very weak formulation satisfied by $ u_\varepsilon $ the purely spatial test function $ \eta$, chosen as  the (positive) solution of the elliptic equation $-\Delta \eta = 1$ in $B_r$, with homogeneous Dirichlet boundary conditions (that is the Dirichlet potential of $ 1 $ in $B_r$). Up to a usual time approximation of $ \chi_{[\tau_1,t]} $, where $t \in (\tau_1,\tau_2)$ is arbitrary, we obtain the identity
\begin{equation}\label{boundedness}
\int_{B_r} u_\varepsilon(t)\, \eta\, \rho \, dx+\int_{\tau_1}^{t}\int_{B_r}u_\varepsilon^m\,dxds=
 \int_{B_r} S_\varepsilon[u](\tau_1) \, \eta \, \rho \, dx + \int_{\tau_1}^{t}\int_{\partial B_r} S_\varepsilon\!\left[u^m|_{\partial B_r }\right] \partial_{\Vec{\mathsf{n}}} \eta \,  d\sigma ds \, .
\end{equation}
Since $ \tau_1 $ is a Lebesgue time for $ t \mapsto u(t) $ as a curve with values in $ L^1_{\rho}(B_r) $ and $ u^m|_{\partial B_r} \in L_\sigma^1(\partial B_r \times (\tau_1,\tau_2) ) $, it is plain that the right-hand side of \eqref{boundedness} is convergent as $\varepsilon\to 0$, whence
\begin{equation}\label{boundedness 2}
\begin{gathered}
\{ u_\varepsilon \} \; \; \text{is bounded in} \; \;  L^\infty\!\left((\tau_1,\tau_2); L^1_{\rho,\mathrm{loc}}(B_r)\right)  , \\
\left\{ u^m_\varepsilon \right\} \; \text{is bounded in} \; \;  L^1(B_r\times(\tau_1,\tau_2)) \, .
\end{gathered}
\end{equation}
Before passing to the limit as $ \varepsilon \to 0 $, let us comment that each  $u_\varepsilon $ is continuous in $ [\tau_1,\tau_2) $ with values in $L^p_\rho(B_r)$ for all $ p\in[1,\infty)$, a fact that will be used in Step 2. Indeed, we already have continuity in $ (\tau_1,\tau_2)$, as a consequence of the fact that $u_\varepsilon$ is a (bounded) strong energy solution. In order to prove continuity down to $t=\tau_1$, we first note that $u_\varepsilon(t)\rightharpoonup u_\varepsilon(\tau_1)$ in $L^p_\rho(B_r)$ as $\tau\to\tau_1$, for all $ p\in[1,\infty)$, which can be easily seen from the definition of very weak solution. On the other hand, for all non-negative $\phi\in C_c^\infty(B_r)$ it holds (formally using $ u^{p-1} \, \phi \,  \chi_{[\tau_1,t]} $ as a test function)
\begin{equation}\label{uppertau}
\int_{B_r}u_\varepsilon^p(t)\,\phi\,\rho\,dx\le \int_{B_r}u_\varepsilon^p(\tau_1)\,\phi\,\rho\,dx+C_\phi \, (t-\tau_1)   \, ,
\end{equation}
for all $t\in(\tau_1,\tau_2)$ and a suitable constant $C_\phi>0$ that depends on $\phi$ and the $L^\infty(B_r\times(\tau_1,\tau_2))$ norm of $u_\varepsilon$. From \eqref{uppertau}, the just mentioned weak convergence in $L^p_{\rho}(B_r)$ as $t\to\tau_1$, and the boundedness of $ u_\varepsilon  $, it is easy to deduce that also strong convergence in $L^p_\rho(B_r)$ holds.

Let us go back to the convergence of $ \{ u_\varepsilon \} $ for $ \varepsilon \to 0 $. Thanks to \eqref{boundedness 2} and the $ L^\infty $ interior estimates of Theorem \ref{local moser iter lem} (also recall Remark \ref{GB}), we can infer that
\begin{equation}\label{boundedness 3}
\{ u_\varepsilon \} \; \; \text{is bounded in} \; \;  L^\infty_{\mathrm{loc}}(B_r\times(\tau_1,\tau_2)) \, .
\end{equation}
As already observed, the solutions $ \{ u_\varepsilon \} $ are strong energy and thus they satisfy the energy estimates of Lemma \ref{energy lem}. Hence, due to \eqref{boundedness 3}, we can appeal to the local version of the Aubin--Lions compactness lemma as in the proof of Lemma \ref{alaa compact2}, deducing in particular that there exists a non-negative function $ \overline{u} \in L^\infty_{\mathrm{loc}}(B_r\times(\tau_1,\tau_2)) $ such that (up to a subsequence that we do not relabel)
\begin{equation}\label{ptwse-eps}
 u_\varepsilon \underset{\varepsilon \to 0}{\longrightarrow} \overline{u}  \qquad \text{a.e.\ in} \; \;  B_r\times(\tau_1,\tau_2) \, .
 \end{equation}
Although we will not directly exploit this fact, we notice that $ \overline{u} $ turns out to be a very weak solution to \eqref{cd problem} (with $ \Omega = B_r $) whose initial and boundary data coincide with those of $u$.

\smallskip
\noindent{\textbf{Step 2: Identification.}}
As in the beginning of the current section, let $G_r$ denote the Dirichlet Green's function of $ -\Delta $ in $B_r$. For all $\varepsilon\in(0,\varepsilon_0)$, we set
$$
\begin{gathered}
	w_\varepsilon(x,t)=\int_{B_r} G_r(x,y)\, S_\varepsilon[u](y,t)\,\rho(y)\,dy \, ,\\
	\overline{w}_\varepsilon(x,t)=\int_{B_r}G_r(x,y)\,u_\varepsilon(y,t)\,\rho(y)\,dy \, ,
\end{gathered}
$$
along with difference potential $E_\varepsilon=\overline{w}_\varepsilon-w_\varepsilon$. From its definition, it is plain that for every $ t \in (\varepsilon_0,T-\varepsilon_0) $ the function $ E_\varepsilon(t) $ satisfies
\begin{equation}\label{E-eps-1}
\begin{cases}
   -\Delta E_\varepsilon(t) = \left( u_\varepsilon(t) - S_\varepsilon[u](t)\right) \rho & \text{in} \; \; B_r \, , \\
   E_\varepsilon(t) = 0 & \text{on} \; \; \partial B_r \, ,
\end{cases}
\end{equation}
and thanks to elliptic regularity up to the boundary (see \emph{e.g.}\ \cite[Lemma 9.17]{GT}) and \eqref{bdd-eps-loc}, we have that $ E_\varepsilon(t) \in W^{2,p}(B_r) \cap W^{1,q}_0(B_r) $ for all $ p \in \left[1,N/\gamma\right) $ and all $ q \in \left[1, N/(\gamma-1)_+ \right) \supset [0,2] $. Since the operator $S_\varepsilon$ commutes with both $ \rho \,\partial_t $ and $\Delta$, we observe that $ S_\varepsilon[u] $ satisfies the differential equation $ \rho \, \partial_t S_\varepsilon[u] = \Delta S_\varepsilon[u^m] $, and $  S_\varepsilon[u^m] $ has the same trace as $ u_\varepsilon^m $ on $ \partial B_r \times (\tau_1,\tau_2) $; hence, it is readily seen that
\begin{equation*}\label{E-eps-2}
   \partial_t E_\varepsilon =  S_\varepsilon[u^m] - u_\varepsilon^m \qquad \text{in} \; \; B_r \times (\tau_1,\tau_2) \, .
\end{equation*}
In particular, we deduce that $ E_\varepsilon \in C^{1}( [\tau_1,\tau_2); L^p(B_r))  $ for all $ p \in [1,\infty) $. Moreover, from \eqref{E-eps-1} and again \eqref{bdd-eps-loc}, for all $t \in (\tau_1,\tau_2)$ and all $ |h|>0 $ sufficiently small we have:
$$
\begin{aligned}
& \int_{B_r} \left| \nabla E_\varepsilon(t+h) \right|^2 dx -  \int_{B_r} \left| \nabla E_\varepsilon(t) \right|^2 dx  \\
= &  -\int_{B_r} E_\varepsilon(t+h) \, \Delta E_\varepsilon(t+h) \, dx  +  \int_{B_r}E_\varepsilon(t) \, \Delta E_\varepsilon(t) \, dx \\
= & \int_{B_r} \left[ E_\varepsilon(t) - E_\varepsilon(t+h) \right] \Delta \!\left[E_\varepsilon(t) + E_\varepsilon(t+h) \right] dx \\
= & \int_{B_r} \left[ E_\varepsilon(t) - E_\varepsilon(t+h) \right] \left( S_\varepsilon[u](t) + S_\varepsilon[u](t+h) - u_\varepsilon(t) - u_\varepsilon(t+h) \right) \rho  \, dx \, .
\end{aligned}
$$
Since both $ S_\varepsilon[u] $ and $ u_\varepsilon $ belong to $ C^0\!\left([\tau_1,\tau_2);L^p_\rho(B_r)\right) $ for all $ p \in [1,\infty) $, we can safely divide by $ h $ and let $ h \to 0 $ in the above identities, which yields
$$
t \mapsto \int_{B_r} \left| \nabla E_\varepsilon(t) \right|^2 dx \in C^1([\tau_1,\tau_2))
$$
with
\begin{equation}\label{E-eps-diff}
\frac{d}{dt} \int_{B_r} \left| \nabla E_\varepsilon(t) \right|^2 dx = 2 \int_{B_r} \left(u_\varepsilon^m(t) - S_\varepsilon[u^m](t) \right)\left( S_\varepsilon[u](t) - u_\varepsilon(t) \right) \rho \, dx \, .
\end{equation}
Integrating \eqref{E-eps-diff} from $\tau_1$ to  any $\tau_1<\tau<\tau_2$ and exploiting the fact that $ \overline{w}_\varepsilon $ and $w_\varepsilon$ coincide on $ B_r \times \{ \tau_1 \} $, we find
$$
0 \le \int_{B_r} \left| \nabla E_\varepsilon(\tau) \right|^2 dx = 2 \int_{\tau_1}^{\tau} \int_{B_r} \left(u_\varepsilon^m - S_\varepsilon[u^m] \right)\left( S_\varepsilon[u] - u_\varepsilon \right) \rho \, dx dt \, ,
$$
which entails
\begin{equation}\label{crucial-ineq}
\begin{aligned}
 & \int_{\tau_1}^{\tau_2} \int_{B_r} S_\varepsilon[u^m] \, S_\varepsilon[u] \, \rho \, dx dt + \int_{\tau_1}^{\tau_2} \int_{B_r} u_\varepsilon^{m+1} \, \rho \, dx dt  \\
 \le & \int_{\tau_1}^{\tau_2} \int_{B_r} u_\varepsilon^m \, S_\varepsilon[u] \, \rho \, dx dt + \int_{\tau_1}^{\tau_2} \int_{B_r} S_\varepsilon[u^m] \, u_\varepsilon \, \rho \, dx dt \, .
 \end{aligned}
\end{equation}
In particular, H\"{o}lder's inequality yields
\begin{equation}\label{crucial-ineq-bis}
\begin{aligned}
 \int_{\tau_1}^{\tau_2} \int_{B_r} u_\varepsilon^{m+1} \, \rho \, dx dt
\le & \left( \int_{\tau_1}^{\tau_2} \int_{B_r} u_\varepsilon^{m+1} \, \rho \, dx dt \right)^{\frac{m}{m+1}} \left(  \int_{\tau_1}^{\tau_2} \int_{B_r} S_\varepsilon[u]^{m+1} \, \rho \, dx dt\right)^{\frac{1}{m+1}} \\
& + \left( \int_{\tau_1}^{\tau_2} \int_{B_r} S_\varepsilon[u^m]^{\frac{m+1}{m}} \, \rho \, dx dt \right)^{\frac{m}{m+1}} \left( \int_{\tau_1}^{\tau_2} \int_{B_r} u_\varepsilon^{m+1} \, \rho \, dx dt \right)^{\frac{1}{m+1}} .
\end{aligned}
\end{equation}
By virtue of Lemma \ref{imp-int} and standard properties of one-variable convolutions, we have
\begin{equation}\label{crucial-ineq-ter}
S_\varepsilon[u] \underset{\varepsilon \to 0}{\longrightarrow} u \quad \text{in} \;\; L^{m+1}_\rho(B_r \times(\tau_1,\tau_2)) \qquad \text{and} \qquad S_\varepsilon[u^m] \underset{\varepsilon \to 0}{\longrightarrow} u^m \quad \text{in} \;\; L^{\frac{m+1}{m}}_\rho(B_r \times(\tau_1,\tau_2)) \, ;
\end{equation}
we can therefore deduce from \eqref{crucial-ineq-bis} that $ \{ u_\varepsilon \} $ is bounded in $ L^{m+1}_\rho(B_r \times(\tau_1,\tau_2)) $, which, in combination with \eqref{ptwse-eps}, entails
\begin{equation}\label{crucial-ineq-quater}
u_\varepsilon \xrightharpoonup[\varepsilon \to 0]{} \overline{u} \quad \text{in} \;\; L^{m+1}_\rho(B_r \times(\tau_1,\tau_2)) \qquad \text{and} \qquad u_\varepsilon^m \xrightharpoonup[\varepsilon \to 0]{} \overline{u}^m \quad \text{in} \;\; L^{\frac{m+1}{m}}_\rho(B_r \times(\tau_1,\tau_2)) \, .
\end{equation}
Thanks to \eqref{crucial-ineq-ter}--\eqref{crucial-ineq-quater}, we can then pass to the limit safely in \eqref{crucial-ineq} as $ \varepsilon \to 0 $ to obtain
$$
\int_{\tau_1}^{\tau_2} \int_{B_r} \left( u^{m+1} + \overline{u}^{m+1} \right) \rho \, dx dt \le \int_{\tau_1}^{\tau_2} \int_{B_r} \left( \overline{u}^m \, u + u^m \, \overline{u} \right) \rho \, dx dt \, ,
$$
that is
$$
\int_{\tau_1}^{\tau_2} \int_{B_r} \left( u^{m} - \overline{u}^{m} \right)\left( u - \overline{u} \right) \rho \, dx dt \le 0 \, ,
$$
whence $ u=\overline{u} $ in $ B_r \times (\tau_1,\tau_2) $ and the thesis follows.
\end{proof}


        \begin{rem}[On \emph{local} non-negative solutions]\label{LS}\rm
    There is no significant obstruction to extending the above arguments to local non-negative very weak solutions, in the sense of Definition \ref{vw loc def}. First of all, by a standard covering argument, it is enough to show that, given any $ B_R(x_0) \Subset \Omega $, the solution $u$ is locally bounded in $ B_R(x_0) \times (\tau_1,\tau_2) $. Also, it is not restrictive to be able to do that under the additional constraint $R \le r_0$, for some fixed $r_0>0$. Then, one can  reason as follows. If $ \Omega $ contains the origin $ x_0 = 0 $, one can choose any $R_*>0$ such that $ B_{R_*} \Subset \Omega  $ and prove that $u$ is locally bounded in $ B_{R_*} \times  (\tau_1,\tau_2) $ exactly as above, having in mind Remark \ref{GB}. On the other hand, in any ball $ B_{R}(x_0) \Subset \Omega $ such that $ |x_0| \ge R_* $ and $ |R| \le R_*/2 $, it is plain that $\rho$ is bounded above and below by positive constants, therefore the arguments of this section, and in particular the local estimates of Theorem \ref{local moser iter lem}, apply with $\gamma=0$ (see again Remarks \ref{GB} and \ref{loc-vw-bdry}). Note that all of these types of balls are enough to cover $\Omega$. Finally, if $ \Omega $ does not contain the origin, one can simply ignore $R_*$.
\end{rem}

\subsection*{Acknowledgments} G.G., M.M., and T.P.\ are members of the Gruppo Nazionale per l'Ana\-lisi Mate\-matica, la Probabilit\`a e le loro Applicazioni (GNAMPA, Italy) of the Istituto Nazionale di Alta Matematica (INdAM, Italy). Their research was in part carried out within the project ``Partial differential equations and related geometric-functional inequalities'', ref. 20229M52AS and the project ``Geometric-Analytic Methods for PDEs and Applications (GAMPA)'', ref. 2022SLTHCE -- funded by European Union -- Next Generation EU within the PRIN 2022 program (D.D. 104 -- 02/02/2022 Ministero dell’Universit\`a e della Ricerca). N.S.\ has been supported by the project CONVIVIALITY ANR-23-CE40-0003 of the French National Research Agency.


\begin{thebibliography}{999}

\bibitem{Andreucci} D. Andreucci, \emph{$L^\infty_{\rm loc}$-estimates for local solutions of degenerate parabolic equations}, SIAM J. Math. Anal. {\bf 22} (1991), 138--145.

\bibitem{Aronson-IT} D.G. Aronson, \emph{Widder's inversion theorem and the initial distribution problems}, SIAM J. Math. Anal. {\bf 12} (1981), 639--651.

\bibitem{AB} D.G. Aronson, Ph. B\'enilan, \emph{R\'egularit\'e des solutions de l’\'equation des milieux poreux dans $\mathbb R^N$} (French), C. R. Acad. Sci. Paris S\'er. A-B \bf 288 \rm (1979), A103--A105.

\bibitem{AC} D.G. Aronson, L.A. Caffarelli, \emph{The initial trace of a solution of the porous medium equation}, Trans. Amer. Math. Soc. {\bf 280} (1983), 351--366.

\bibitem{AP} D.G. Aronson, L.A. Peletier, \emph{Large time behaviour of solutions of the porous medium equation in bounded domains}, J. Differential Equations {\bf 39} (1981), 378--412.

\bibitem{BPSV} B. Barrios, I. Peral, F. Soria, E. Valdinoci, \emph{A Widder's type theorem for the heat equation with nonlocal diffusion}, Arch. Ration. Mech. Anal. {\bf 213} (2014), 629--650.

\bibitem{BCP} Ph. B\'enilan, M.G. Crandall, M. Pierre, \emph{Solutions of the porous medium equation in $\mathbb R^N$ under optimal conditions on initial values}, Indiana Univ. Math. J. {\bf 33} (1984), 51--87.

\bibitem{BG} Ph. B\'enilan, R.F. Gariepy, \emph{Strong solutions in $L^1$ of degenerate parabolic equations}, J. Differential Equations {\bf 19} (1995), 473--502.





\bibitem{BS1} M. Bonforte, N. Simonov, \emph{Quantitative a priori estimates for fast diffusion equations with Caffarelli-Kohn-Nirenberg weights. Harnack inequalities and H\"older continuity}, Adv. Math. {\bf 345} (2019), 1075--1161.

\bibitem{BS2} M. Bonforte, N. Simonov, \emph{Fine properties of solutions to the Cauchy problem for a fast diffusion equation with Caffarelli-Kohn-Nirenberg weights}, Ann. Inst. H. Poincar\'e{} C Anal. Non Lin\'eaire {\bf 40} (2023), 1--59.

\bibitem{BSV} M. Bonforte, Y. Sire, J.L. V\'{a}zquez, \emph{Optimal existence and uniqueness theory for the fractional heat equation}, Nonlinear Anal. {\bf 153} (2017), 142--168.


\bibitem{BV} M. Bonforte, J.L. V\'{a}zquez, \emph{A priori estimates for fractional nonlinear degenerate diffusion equations on bounded domains}, Arch. Ration. Mech. Anal., {\bf 218} (2015), 317--362.



\bibitem{BK} H. Br\'ezis, S. Kamin, \emph{Sublinear elliptic equations in $\mathbb{R}^n$}, Manuscripta Math. \bf 74 \rm (1992), 87--106.

 \bibitem{Cohn} D.L. Cohn, ``Measure Theory'', Birkh\"{a}user Adv. Texts Basler Lehrb\"{u}cher [Birkh\"{a}user Advanced Texts: Basel Textbooks], Birkh\"{a}user/Springer, New York, 2013.

\bibitem{DK} B.E.J. Dahlberg, C.E. Kenig, \emph{Nonnegative solutions of the porous medium equation}, Comm. Partial Differential Equations {\bf 9} (1984), 409--437.

\bibitem{DK2} B.E.J. Dahlberg, C.E. Kenig, \emph{Weak solutions of the porous medium equation}, Trans. Amer. Math. Soc. {\bf 336} (1993), 711--725.

\bibitem{DasK} P. Daskalopoulos, C.E. Kenig, ``Degenerate Diffusions. Initial Value Problems and Local Regularity Theory'',
EMS Tracts Math., 1, European Mathematical Society (EMS), Z\"urich, 2007.

\bibitem{DiBV} E. DiBenedetto, U. Gianazza, V. Vespri,
``Harnack's Inequality for Degenerate and Singular Parabolic Equations'', Springer Monogr. Math. Springer, New York, 2012.

\bibitem{Evans} L.C. Evans, ``Partial Differential Equations'', Grad. Stud. Math., 19, American Mathematical Society, Providence, RI, 1998.

\bibitem{EvansGar} L.C. Evans, R.F. Gariepy, ``Measure Theory and Fine Properties of Functions'', Revised edition, Textb. Math., CRC Press, Boca Raton, FL, 2015.

\bibitem{Ev} N. Evseev, \emph{Vector-valued Sobolev spaces based on Banach function spaces}, Nonlinear Anal. \textbf{211} (2021), Paper No. 112479, 15 pp.


\bibitem{GT} D. Gilbarg, N.S. Trudinger, ``Elliptic Partial Differential Equations of Second Order'', Reprint of the 1998 edition, Classics Math., Springer-Verlag, Berlin, 2001.


\bibitem{GQSV}  I. Gonzálvez, F. Quirós, F. Soria, Z. Vondraček, \emph{On the nonlocal heat equation for certain Lévy operators and the uniqueness of positive solutions}, arXiv preprint: \url{https://arxiv.org/abs/2504.04246}.

\bibitem{GMPo} G. Grillo, M. Muratori, M.M. Porzio, \emph{Porous media equations with two weights: smoothing and decay properties of energy solutions via {P}oincar\'e{} inequalities}, Discrete Contin. Dyn. Syst. \bf 33 \rm (2013), 3599--3640.

\bibitem{GMP3} G. Grillo, M. Muratori, F. Punzo, \emph{Conditions at infinity for the inhomogeneous filtration equation}, Ann. Inst. H. Poincar\'e C Anal. Non Lin\'eaire \bf 31 \rm (2014), 413--428.

\bibitem{GMP} G. Grillo, M. Muratori, F. Punzo, \emph{Fractional porous media equations: existence and uniqueness of weak solutions with measure data}, Calc. Var. Partial Differential Equations {\bf 54} (2015), 3303--3335.

\bibitem{GMP2} G. Grillo, M. Muratori, F. Punzo, \emph{Fast diffusion on noncompact manifolds: well-posedness theory and connections with semilinear elliptic equations},  Trans. Amer. Math. Soc. \bf 374 \rm (2021), 6367--6396.

\bibitem{GMV1} G. Grillo, M. Muratori, J.L. V\'{a}zquez, \emph{The porous medium equation on Riemannian manifolds with negative curvature. The large-time behaviour}, Adv. Math. {\bf 314} (2017), 328--377.

\bibitem{GMV2} G. Grillo, M. Muratori, J.L. V\'{a}zquez, \emph{The porous medium equation on Riemannian manifolds with negative curvature: the superquadratic case}, Math. Ann. {\bf 373} (2019), 119--153.



\bibitem{KRV} S. Kamin, G. Reyes, J.L. V\'{a}zquez, \emph{Long time behavior for the inhomogeneous PME in a medium with rapidly decaying density}, Discrete Contin. Dyn. Syst. {\bf 26} (2010), 521--549.

\bibitem{KRo} S. Kamin, P. Rosenau, \emph{Propagation of thermal waves in an inhomogeneous medium}, Comm. Pure. Appl. Math. {\bf 34} (1981), 831--852.

\bibitem{KRo2} S. Kamin, P. Rosenau, \emph{Nonlinear diffusion in a finite mass medium}, Comm. Pure. Appl. Math. {\bf 35} (1982), 113--127.


\bibitem{Lieberman} G.M. Lieberman, ``Second Order Parabolic Differential Equations'', World Scientific Publishing Co., Inc., River Edge, NJ, 1996.

\bibitem{Mik} J. Mikusi\'nski, ``The Bochner Integral'', Lehrb\"ucher und Monographien aus dem Gebiete der Exakten Wissenschaften, Mathematische Reihe, Band 55, Birkh\"auser Verlag, Basel-Stuttgart, 1978.

\bibitem{M} M. Muratori, \emph{The fractional Laplacian in power-weighted $L^p$ spaces: integration-by-parts formulas and self-adjointness}, J. Funct. Anal. {\bf 271} (2016), 3662--3694.

\bibitem{MP} M. Muratori, T. Petitt, \emph{An inhomogeneous porous medium equation with large data: well-posedness}, J. Differential Equations {\bf 377} (2023), 712--758.

\bibitem{MPQ} M. Muratori, T. Petitt, F. Quir\'{o}s, \emph{An inhomogeneous porous medium equation with non-integrable data: asymptotics}, arXiv preprint: \url{https://arxiv.org/abs/2403.12854}.

\bibitem{NR} S. Nieto, G. Reyes, \emph{Asymptotic behavior of the solutions of the inhomogeneous porous medium equation with critical vanishing density},  Commun. Pure Appl. Anal. {\bf 12} (2013), 1123--1139.

\bibitem{Pierre} M. Pierre, \emph{Uniqueness of the solutions of $u_t - \Delta \varphi(u) = 0$ with initial datum a measure}, Nonlinear Anal. {\bf 6} (1982), 175--187.

\bibitem{RV1} G. Reyes, J.L. V\'{a}zquez, \emph{The Cauchy problem for the inhomogeneous porous medium equation}, Netw. Heterog. Media {\bf 1} (2006), 337--351.

\bibitem{RV2} G. Reyes, J.L. V\'{a}zquez, \emph{The inhomogeneous PME in several space dimensions. Existence and uniqueness of finite energy solutions}, Commun. Pure Appl. Anal. {\bf 7} (2008), 1275--1294.

\bibitem{RV3} G. Reyes, J.L. V\'{a}zquez, \emph{Long time behavior for the inhomogeneous PME in a medium with slowly decaying density}, Comm. Pure Appl. Anal. {\bf 8} (2009), 493--508.

\bibitem{Stamp} G. Stampacchia, \emph{Le probl\`eme de Dirichlet pour les \'equations elliptiques du second ordre \`a coefficients discontinus} (French), Ann. Inst. Fourier (Grenoble) \bf 15 \rm (1965), 189--258.


\bibitem{Vazquez} J.L. V\'{a}zquez, ``The Porous Medium Equation. Mathematical Theory'', Oxford Math. Monogr.,
The Clarendon Press, Oxford University Press, Oxford, 2007.

\bibitem{Vazquez-hyp} J.L. V\'{a}zquez, \emph{Fundamental solution and long time behavior of the porous medium equation in hyperbolic space}, J. Math. Pures Appl. {\bf 104} (2015), 454--484.

\bibitem{Wid} D.V. Widder, \emph{Positive temperatures on an infinite rod}, Trans. Amer. Math. Soc. {\bf 55} (1944), 85--95.

\bibitem{WidBook} D.V. Widder, ``The Heat Equation'', Pure Appl. Math., Vol. 67, Academic Press [Harcourt Brace Jovanovich, Publishers], New York-London, 1975.

\end{thebibliography}
\end{document}